\documentclass[a4paper,11pt,times,preprint]{elsarticle}
\usepackage{fullpage}
\usepackage{amsmath,amssymb,amsthm,subfig,bm,units}
\usepackage[mathscr]{euscript}
\usepackage{mdframed,booktabs,rotating,multirow}
\usepackage[title,toc,titletoc,page]{appendix}

\theoremstyle{definition}
\newtheorem{theorem}{Theorem}
\newtheorem{lemma}{Lemma}
\newtheorem{corollary}{Corollary}
\newtheorem{remark}{Remark}

\DeclareMathOperator{\sgn}{sgn}
\newcommand{\an}{\text{~~and~~}}

\title{{\bf An Approximate Solver for Multi-medium Riemann Problem
with Mie-Gr{\"u}neisen Equations of State}}

\author[pku]{Li Chen}
\ead{cheney@pku.edu.cn}
\author[pkuextra]{Ruo Li\corref{cor}}
\ead{rli@math.pku.edu.cn}
\author[pku,nint]{Chengbao Yao}
\ead{yaocheng@pku.edu.cn}

\cortext[cor]{Corresponding author}
\address[pku]{School of Mathematical Sciences,
	Peking University, Beijing, China}
\address[pkuextra]{HEDPS \& CAPT, LMAM \& 
	School of Mathematical Sciences,
	Peking University, Beijing, China}
\address[nint]{Northwest Institute of Nuclear Technology, Xi'an, China}   

\begin{document}
\begin{abstract} 
We propose an approximate solver for multi-medium Riemann problems with
materials described by a family of general Mie-Gr{\"u}neisen equations of state,
which are widely used in practical applications.  The solver provides the
interface pressure and normal velocity by an iterative method. The
well-posedness and convergence of the solver is verified with mild assumptions
on the equations of state. To validate the solver, it is employed in computing
the numerical flux on phase interfaces of a numerical scheme on Eulerian grids
that was developed recently for compressible multi-medium flows. Numerical
examples are presented for Riemann problems, air blast and underwater explosion
applications.
\end{abstract}

\maketitle

\section{Introduction}
Numerical simulations of compressible multi-medium flow are of great interest in
practical applications, such as mechanical engineering, chemical industry, and
so on. Many conservative Eulerian algorithms perform very well in single-medium
compressible flows. However, when such an algorithm is employed to compute
multi-medium flows, numerical inaccuracies may occur at the material interfaces
\cite{Abgrall1996, Saurel2007, Liu2003}, due to the great discrepancy of
densities and equations of state across the interface. The simulation of
compressible multi-medium flow in an Eulerian framework requires special
attention in describing the interface connecting distinct fluids, especially for
the problems that involve highly nonlinear equations of state. Several
techniques have been taken to treat the multi-medium flow interactions. See
\cite{Liu2003, Abgrall2001, Karni1994, Arienti2004, Shyue2001, Saurel1999,
Price2015} for instance.

A typical procedure of multi-medium compressible flows in Eulerian grids mainly
consists of two steps, the interface capture and the interaction between
different fluids. There are mainly two different approaches in literatures, the
diffuse interface method (DIM) and the sharp interface method (SIM). The former
\cite{Abgrall1996, Abgrall2001, Saurel1999, Saurel2009, Petitpas2009,
Ansari2013} assumes the interface as a diffuse zone, and smears out the
interface over a certain number of grid cells to avoid pressure oscillations.
Diffuse interfaces correspond to artificial mixtures created by numerical
diffusion, and the key issue is to fulfill interface conditions within this
artificial mixture. The latter assumes the interface to be a sharp contact
discontinuity, and different fluids are immiscible. Several approaches such as
the volume of fluid (VOF) method \cite{Scardovelli1999, Noh1976}, level set
method \cite{Sethian2001, Sussman1994}, moment of fluid (MOF) method
\cite{Ahn2007, Dyadechko2008, Anbarlooei2009} and front-tracking method
\cite{Glimm1998, Tryggvason2001} are used extensively to capture the interface.
A key element for both diffuse and sharp interface methods, is to determine
the interface states. The accurate prediction of the interface states can be
used to stabilize the numerical diffusion in diffuse interface methods, and to
compute the numerical flux and interface motion in sharp interface methods. One
common approach is to solve a multi-medium Riemann problem which contains the
fundamentally physical and mathematical properties of the governing equations.
Indeed, the Riemann problem plays a key role in understanding the wave
structures, since a general fluid flow may be interpreted as a nonlinear
superpositions of the Riemann solutions. 

The solution of a multi-medium Riemann problem depends not only on the initial
states at each side of the interface, but also on the forms of equations of
state. For some simple equations of state, such as ideal gas, covolume gas or
stiffened gas, the solution of the Riemann problem can be achieved to any
desired accuracy with an exact solver.  While the Riemann problems for the above
equations of state have been fully investigated in \cite{Godunov1976, Plohr1988,
Gottlieb1988, Toro2008} for instance, there exist some difficulties in the cases
of general nonlinear equations of state due to their high nonlinearity. A
variety of methods to solve the corresponding Riemann problems have then been
proposed. For example, Larini \textit{et al.} \cite{Larini1992} developed an
exact Riemann solver and applied their methods to a fifth-order virial equation
of state (EOS), which is particularly suited to the gaseous detonation products
of high explosive compounds. Shyue \cite{Shyue2001} developed a Roe's
approximate Riemann solver for the Mie-Gr{\"u}neisen EOS with variable
Gr{\"u}neisen coefficient.  Quartapelle \textit{et al.} \cite{Quartapelle2003}
proposed an exact Riemann solver by applying the Newton-Raphson iteration to the
system of two nonlinear equations imposing the equality of pressure and of
velocity, assuming as unknowns the two values of the specific volume at each
side of the interface, and implemented it for the van der Waals gas.  Arienti
\textit{et al.} \cite{Arienti2004} applied a Roe-Glaster solver to compute the
equations combining the Euler equations involving chemical reaction with the
Mie-Gr{\"u}neisen EOS.  More recently, Rallu \cite{Rallu2009} and Farhat
\textit{et al.} \cite{Farhat2012} utilized a sparse grid technique to tabulate
the solutions of exact multi-medium Riemann problems.  Lee \textit{et al.}
\cite{Lee2013} developed an exact Riemann solver for the Mie-Gr{\"u}neisen EOS
with constant Gr{\"u}neisen coefficient, where the integral terms are evaluated
using an iterative Romberg algorithm.  Banks \cite{Banks2010} and Kamm
\cite{Kamm2015} developed a Riemann solver for the convex Mie-Gr{\"u}neisen EOS
by solving a nonlinear equation for the density increment involved in the
numerical integration of rarefaction curves, and chose the JWL
(Jones-Wilkins-Lee) EOS	as a representative case.

In this paper, we propose an approximate multi-medium Riemann solver for a
family of general Mie-Gr{\"u}neisen equations of state in an iterative manner,
which provides a strategy to reproduce the physics of interaction between two
mediums across the interface. Several mild conditions on the coefficients of
Mie-Gr{\"u}neisen EOS are assumed to ensure the convexity of equations of state
and the well-posedness of our Riemann solver. The algebraic equation of the
Riemann problem is solved by an inexact Newton method \cite{Dembo1982}, where
the function and its derivative are evaluated approximately depending on the
wave configuration. And the convergence of our Riemann solver is analyzed. To
validate the proposed solver, we employed it in the computation of two-medium
compressible flows with Mie-Gr{\"u}neisen EOS, as an extension of the numerical
scheme that was developed recently for two-medium compressible flows with ideal
gases
\cite{Guo2016}.

The rest of this paper is arranged as follows. In Section \ref{sec:riemann}, a
solution strategy for the multi-medium Riemann problem with arbitrary
Mie-Gr{\"u}neisen equations of state is presented. In Section \ref{sec:aps}, the
procedures of our approximate Riemann solver are outlined, and the
well-posedness and convergence are analyzed. In Section \ref{sec:model}, the
application of our Riemann solver in two-medium compressible flow calculations
\cite{Guo2016} is briefly introduced. In Section \ref{sec:num}, several
classical Riemann problems and applications for air blast and underwater
explosions are carried out to validate the accuracy and robustness of our
solver. Finally, a short conclusion is drawn in Section \ref{sec:conclusion}.


\section{Multi-medium Riemann Problem}
\label{sec:riemann}

Let us consider the following one-dimensional multi-medium Riemann problem of
the compressible Euler equations 
\begin{subequations}
\begin{equation}
  \dfrac{\partial}{\partial\tau}
  \begin{bmatrix}
    \rho \\ \rho u \\ E
  \end{bmatrix}
  +\dfrac{\partial}{\partial\xi}
  \begin{bmatrix}
    \rho u \\ \rho u^2+p \\ (E+p)u
  \end{bmatrix}
  =\bm 0,\quad
  E=\rho e+\dfrac{1}{2}\rho u^2.
\end{equation}
Here $\tau$ is time and $\xi$ is spatial coordinate, and $\rho$, $u$, $p$, $E$
and $e$ are the density, velocity, pressure, total energy and specific internal
energy, respectively. The system has initial values
\begin{equation}
  [\rho,u,p]^\top (\xi,\tau=0) =
  \begin{cases}
    [\rho_l,u_l,p_l]^\top, & \xi<0, \\
    [\rho_r,u_r,p_r]^\top, & \xi>0.
  \end{cases}
\end{equation}
\label{system:oneriemann}
\end{subequations}

Here the equations of state for both mediums under our consideration can be
classified into the family known as \emph{the Mie-Gr\"uneisen EOS}. The
Mie-Gr\"uneisen EOS can be used to describe a lot of real materials, for
example, the gas, water and gaseous products of high explosives
\cite{Arienti2004, Liu2003, Kamm2015}, which is particularly useful in those
practical applications we are studying now. The general form of the
Mie-Gr\"uneisen EOS is given by
\begin{equation}
p(\rho, e)=\varGamma(\rho)\rho e + h(\rho),
\label{eq:particulareos}
\end{equation}
where $\varGamma(\rho)$ is the Gr{\"u}neisen coefficient, and $h(\rho)$ is a
reference state associated with the cold contribution resulting from the
interactions of atoms at rest \cite{Heuze2012}. Thus the EOS of the multi-medium
Riemann problem \eqref{system:oneriemann} is given by
\[
p(\rho, e) = 
\begin{cases}
\varGamma_l(\rho)\rho e+h_l(\rho), \\
\varGamma_r(\rho)\rho e+h_r(\rho),
\end{cases}
\]
for the medium on the left and the right sides, respectively. For the ease of
our analysis, we impose on $\varGamma(\rho)$ and $h(\rho)$ the following
assumptions: \bigskip

\textbf{(C1)} $\varGamma'(\rho) \le 0,~(\rho\varGamma(\rho))' \ge 
0,~(\rho\varGamma(\rho))''\ge 0$;
\medskip

\textbf{(C2)} $\lim\limits_{\rho\rightarrow+\infty}\varGamma(\rho)=
  \varGamma_\infty>0,~\varGamma(\rho)\le \varGamma_{\infty}+2$;
\medskip

\textbf{(C3)} $h'(\rho)\ge 0,~h''(\rho)\ge 0$.
\bigskip
\begin{remark}
  An immediate consequence is that the Gr{\"u}neisen coefficient
  $\varGamma(\rho)$ must be nonnegative since
  $\varGamma(\rho) \ge-\varGamma'(\rho)\rho \ge 0$ by the condition
  \textbf{(C1)}.
\end{remark} 
A lot of equations of state of our interests fulfill these assumptions.
Particularly, we collect some equations of state in Appendix A which are used in
our numerical tests as examples. These examples include ideal gas EOS, stiffened
gas EOS, polynomial EOS, JWL EOS, and Cochran-Chan EOS. The coefficients
$\varGamma(\rho)$, $h(\rho)$ and their derivatives for these equations of state 
are all provided therein.

The Riemann problem for general convex equations of state have been fully
analyzed, for example, in \cite{Smith1979, Menikoff1989}. Here the problem is
more specific, thus we can present the structure of the solution in a
straightforward way. The property of EOS is essential on the wave structures in
the solution of Riemann problems. It is pointed out that the wave structures are
composed solely of elementary waves \cite{Menikoff1989} if the \emph{fundamental
derivative} \cite{Thompson1971}
\[
\mathscr G=\dfrac{1}{c}\cdot\left.\dfrac{\partial\rho c}{\partial\rho}\right|_s=
1+\dfrac{\rho}{2c^2}\cdot\left.\dfrac{\partial c^2}{\partial \rho}\right|_s,
\]
keeps positive 
\footnote{ 
  When the positivity condition $\mathscr G>0$ is violated, other configurations
  of waves may occur, such as composite waves, split waves, expansive shock
  waves or compressive rarefaction waves \cite{Menikoff1989}. The above
  anomalous wave structures are common issues in phase transitions. For further
  discussions on these anomalous wave structures, we refer the readers to
  \cite{Weyl1949, Thompson1973, Liu1975, Liu1976, Pego1986, Bethe1998,
  Bates2002, Voss2005, Muller2006, Fossati2014} for instance.}, 
where $s$ is the specific entropy and $c$ is the speed of sound. For the
Mie-Gr{\"u}neisen EOS \eqref{eq:particulareos}, the speed of sound can
be expressed as
\[
c(p,\rho)=\sqrt{\left.\dfrac{\partial p}{\partial \rho}\right|_e
	+\dfrac{p}{\rho^2}
	\left.\dfrac{\partial p}{\partial e}\right|_{\rho}}=
\sqrt{\left(\dfrac{1}{\rho} +
	\dfrac{\varGamma_k'(\rho)}{\varGamma_k(\rho)}\right)(p-h_k(\rho))
	+\dfrac{p}{\rho}\varGamma_k(\rho) + h_k'(\rho)},
\]
and a tedious calculus gives us that the fundamental derivative is
\begin{equation}
\begin{split}
\mathscr G = &\dfrac{1}{c^2}\left(
\dfrac12\left(\rho(\rho\varGamma_k(\rho))''
+{(\rho\varGamma_k(\rho))'}(2+\varGamma_k(\rho))\right)e+\dfrac12\rho
h_k''(\rho) \right. \\
&\left.+\dfrac{p}{2\rho}(\varGamma_k^2(\rho)+2(\rho\varGamma_k(\rho))')
+\dfrac12(2+\varGamma_k(\rho))h_k'(\rho)\right).
\end{split}
\label{eq:fundd}
\end{equation}
We conclude that
\begin{lemma}
  The solution of the multi-medium Riemann problem 
  \eqref{system:oneriemann} consists of only elementary waves if the
  conditions \textbf{(C1), (C2)} and \textbf{(C3)} are fulfilled.
\end{lemma}
\begin{proof}
  With the conditions \textbf{(C1), (C2)} and \textbf{(C3)}, a direct
  check on the terms in \eqref{eq:fundd} gives us that
  \[
  \mathscr G > 0,
  \] 
  then the result in \cite{Menikoff1989} is applied to give us the
  conclusion.
\end{proof}
Precisely, the solution of \eqref{system:oneriemann} consists of four
constant regions connected by a linearly degenerate wave and two
genuinely nonlinear waves (either shock wave or rarefaction wave,
depending on the initial states), as is shown schematically in Fig.
\ref{fig:Riemann-Fluid}. The linearly degenerate wave is actually the
phase interface. 

\begin{figure}[htbp]
\centering
\includegraphics[width=.7\textwidth]{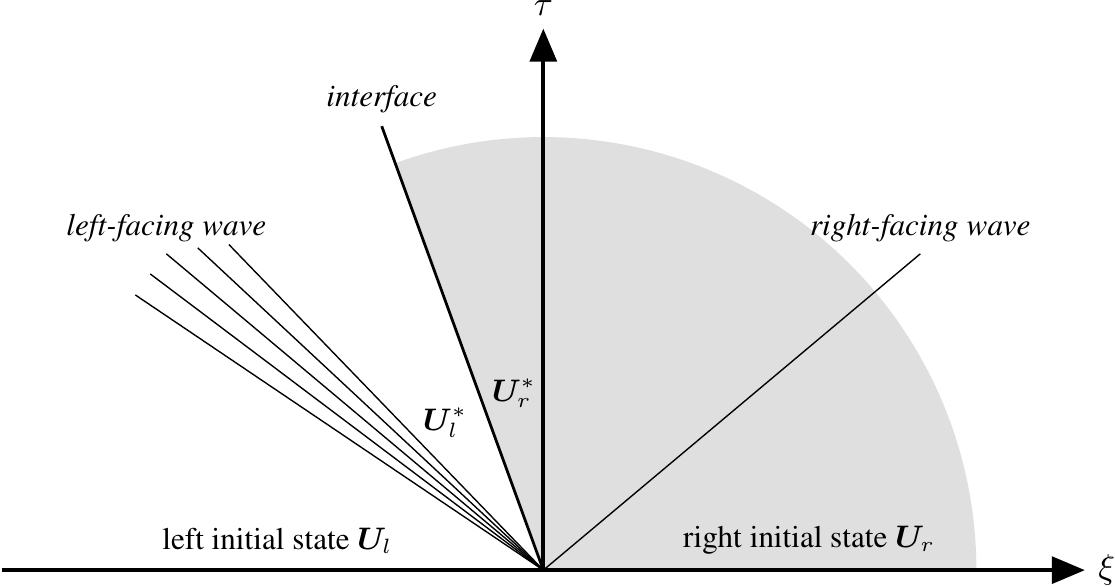}
\caption{Typical wave structure of the Riemann problem 
in $\xi - \tau$ space.}
\label{fig:Riemann-Fluid}
\end{figure}

Following the convention on notations, we denote the pressure and the velocity
by $p^*$ and $u^*$ in the star region, respectively, which have the same value
crossing over the phase interface. This allows us to derive a single nonlinear
algebraic equation for $p^*$. Then we will solve this algebraic equation by an
iterative method. At the beginning of each step of the iterative method, a
provisional value of $p^*$ determines the wave structures from four possible
configurations. The wave structures then prescribe the specific formation of the
algebraic equation. Based on the residual of the algebraic equation, the value
of $p^*$ is updated, which closes a single loop of the iterative method. 

Below let us give the details of the plan above. Firstly we need to study the
relations of the solution across a nonlinear wave, saying a shock wave or a
rarefaction wave. For convenience, we use the subscript $k = l$ or $r$ standing
for either the left initial state $l$ or the right initial state $r$ hereafter.

\begin{itemize}
\item[-] {\bf Solution through a shock wave}

If $p^*>p_k$, the corresponding nonlinear wave is a shock wave, and 
the star region state $\bm U_k^*$ is connected to the adjacent initial
state $\bm U_k$ through a Hugoniot locus. The Rankine-Hugoniot jump 
conditions \cite{Toro2008} yield

\begin{equation}
\mp(u^*-u_k) = 
\left((p^*-p_k)\left(\dfrac{1}{\rho_k}-\dfrac{1}{\rho^*_k}
\right)\right)^{1/2},
\label{eq:diffushock}
\end{equation}
and 
\begin{equation}
e(p^*,\rho^*_k) - e(p_k,\rho_k) + \dfrac{1}{2}(p^*+p_k)
\left(\dfrac{1}{\rho^*_k} -\dfrac{1}{\rho_k} \right) = 0,
\label{eq:eose}
\end{equation}
where 
\[
e(p,\rho)=\dfrac{p-h_k(\rho)}{\varGamma_k(\rho)\rho}.
\]
Multiplying both sides of the equality \eqref{eq:eose} by
$\rho_k\rho_k^*\varGamma_k(\rho_k)\varGamma_k(\rho_k^*)$ gives rise to
\[
\begin{split}
&\varGamma_k(\rho_k)\rho_k (p^*-h_k(\rho_k^*)) -
\varGamma_k(\rho_k^*) \rho_k^* (p_k-h_k(\rho_k)) \\
&-\dfrac{1}{2}\varGamma_k(\rho_k^*)\varGamma_k(\rho_k)(p^*+p_k)
(\rho_k^*-\rho_k)=0.
\end{split}
\]
For convenience we introduce the \emph{Hugoniot function} as follows
\begin{equation}
\begin{split}
\varPhi_k(p,\rho) &:=\varGamma_k(\rho_k)\rho_k (p-h_k(\rho)) -
\varGamma_k(\rho) \rho (p_k-h_k(\rho_k))\\
&-\dfrac{1}{2}\varGamma_k(\rho_k)(p+p_k)\varGamma_k(\rho)
(\rho-\rho_k),
\end{split}
\label{eq:phi}
\end{equation}
then the relation \eqref{eq:eose} boils down to the algebraic equation
$\varPhi_k(p^*,\rho_k^*)=0$. The derivative of $\varPhi_k$ with respect to the
density is found to be
\begin{equation}
\begin{split} 
\dfrac{\partial\varPhi_k}{\partial\rho}(p,\rho) 
&= -\varGamma_k(\rho_k)\rho_k
\left(h_k'(\rho)-\dfrac{p+p_k}{2}\varGamma_k'(\rho)\right)\\
&-\varGamma_k(\rho_k)(\varGamma_k(\rho)+\rho\varGamma_k'(\rho))
\left(\rho_ke_k+\dfrac{p+p_k}{2}\right).
\end{split}
\label{eq:phide}
\end{equation}
As a result, the slope of the Hugoniot locus can be found by the method of
implicit differentiation, namely,
\[
\chi(p,\rho) := \left.\dfrac{\partial p}{\partial \rho}\right|_{\varPhi_k}=
-\dfrac{2\partial\varPhi_k(p,\rho)/\partial\rho}{
\varGamma_k(\rho_k)(2\rho_k-{\varGamma_k(\rho)}(\rho-\rho_k))}.
\]
Before we discuss the properties of the Hugoniot function $\varPhi_k(p,\rho)$,
let us introduce the compressive limit of the density $\rho_{\max}$ such that
$\chi(p,\rho_{\max})=\infty$. By definition $\rho_{\max}$ solves the algebraic
equation $2\rho_k-\varGamma_k(\rho)(\rho-\rho_k)=0$. This quantity is uniquely
defined since the function
\[
W(\rho):=\left(\dfrac{\rho}{\rho_{k}}-1\right)\varGamma_k(\rho)-2,
\]
is monotonically increasing in the interval $(\rho_k,+\infty)$ by the condition
\textbf{(C1)}, and 
\[
W(\rho_k)=-2,\quad
\lim_{\rho\rightarrow+\infty}W(\rho)
\ge \lim_{\rho\rightarrow+\infty}
\left(\dfrac{\rho}{\rho_{k}}-1\right)\varGamma_{\infty}-2=+\infty.
\]
We have the following results on the function $\varPhi_k(p,\rho)$
\begin{lemma}\label{thm:phi}
  Assume that the functions $\varGamma_k(\rho)$ and $h_k(\rho)$ satisfy the
conditions \textbf{(C1), (C2)} and \textbf{(C3)}. Then $\varPhi_k(p,\rho)$
defined in \eqref{eq:phi} satisfies the following properties: 
  \begin{itemize}
    \item[(1).] $\varPhi_k(p,\rho_k)>0$;
    \item[(2).] $\varPhi_k(p,\rho_{\max})<0$;
    \item[(3).] ${\partial\varPhi_k}(p,\rho)/{\partial \rho}<0$;
    \item[(4).] ${\partial^2\varPhi_k}(p,\rho)/{\partial\rho^2}<0$ if
      $\varGamma''_k(\rho)=0$.
  \end{itemize}
\end{lemma}
\begin{proof}
	(1). By definition \eqref{eq:phi} we have
	\[
	\begin{split}
	\varPhi_k(p,\rho_k)&=\varGamma_k(\rho_k)\rho_k (p-h_k(\rho_k))-
	\varGamma_k(\rho_k) \rho_k (p_k-h_k(\rho_k)) \\
	&= \varGamma_k(\rho_k)\rho_k(p-p_k) > 0.
	\end{split}
	\]
	
	(2). Since the compressive limit of the density $\rho_{\max}$ satisfies 
	the relation $(\rho_{\max}-\rho_k)\varGamma_k(\rho_{\max})=2\rho_k$,
we have
	\begin{equation}
	\begin{split}
	&\varPhi_k(p,\rho_{\max})\\
	=&\varGamma_k(\rho_k)\rho_k 
	(p-h_k(\rho_{\max}))-
	\varGamma_k(\rho_{\max}) \rho_{\max} (p_k-h_k(\rho_k)) \\
	 -&
	\dfrac{1}{2}\varGamma_k(\rho_{\max})(\rho_{\max}
	-\rho_k)\varGamma_k(\rho_k)(p+p_k)\\
	=&-\varGamma_k(\rho_k)\rho_k(p_k+h_k(\rho_{\max})) -
	(\varGamma_k(\rho_{\max})+2)\rho_k(p_k-h_k(\rho_k)).
	\end{split}
	\label{eq:phimax}
	\end{equation}
	Obviously $\varPhi_k(p,\rho_{\max})<0$ if $h_k(\rho_{\max})\ge 0$. On 
	the other hand, suppose that $h_k(\rho_{\max})< 0$. Rewriting
	\eqref{eq:phimax} as
	\[
	\begin{split}
	\varPhi_k(p,\rho_{\max})&=
	-(\varGamma_k(\rho_k)+\varGamma_k(\rho_{\max})+2)\rho_kp_k\\
	&-\rho_k(\varGamma_k(\rho_k)h_k(\rho_{\max})
	-(\varGamma_k(\rho_{\max})+2)h_k(\rho_k)),
	\end{split}
	\]
	and using the inequality resulting from the conditions \textbf{(C2)}
and \textbf{(C3)} 
	\[
	\varGamma_k(\rho_k)h_k(\rho_{\max})
	\ge \varGamma_k(\rho_k)h_k(\rho_k)\ge 
    (\varGamma_\infty+2)h_k(\rho_k)\ge
	(\varGamma_k(\rho_{\max})+2)h_k(\rho_k),
	\]
	we conclude that $\varPhi_k(p,\rho_{\max})<0$.
	
	(3). This is an obvious result from the expression \eqref{eq:phide}.
	
	(4). The second derivative of $\varPhi_k(p,\rho)$ with respect to the
density is
	\[
	\begin{split}
	\dfrac{\partial^2\varPhi_k}{\partial\rho^2}(p,\rho)&= -\varGamma_k(\rho_k)\rho_k
	\left(h_k''(\rho)-\dfrac{p+p_k}{2}\varGamma_k''(\rho)\right)\\
	&-\varGamma_k(\rho_k)(2\varGamma_k'(\rho)
	+\rho\varGamma_k''(\rho))
	\left(\rho_ke_k+\dfrac{p+p_k}{2}\right),
	\end{split}
	\]
	which is negative if $\varGamma''_k(\rho)=0$. This completes the proof
of the whole theorem.
\end{proof}
Instantly by Lemma \ref{thm:phi}, the density $\rho$ can be uniquely 
determined from the equation $\varPhi_k(p,\rho)=0$ on the interval
$(\rho_k,\rho_{\max})$ for any fixed $p$. Also the Hugoniot curve is
monotonic due to $\chi>0$. Since the equation \eqref{eq:eose} uniquely
defines the interface density $\rho_k^*$ for a given value of $p^*$,
the right hand side of \eqref{eq:diffushock} can be regarded as a
function of the interface pressure $p^*$ alone, formally written as
\[
f_k(p^*)=\left((p^*-p_k)\left(\dfrac{1}{\rho_k}-\dfrac{1}{\rho^*_k}
\right)\right)^{1/2},\quad p^*>p_k.
\]

\item[-] {\bf Solution through a rarefaction wave}

If, on the other hand, $p^*\le p_k$, then the corresponding nonlinear
wave 
is a rarefaction wave, and the interface state $\bm U_k^*$ is
connected to 
the adjacent initial state $\bm U_k$ through a rarefaction curve.
Since the 
Riemann invariant
\[
u\pm \int \dfrac{1}{\rho c}\mathrm d p,
\]
is constant along the right-facing (left-facing) rarefaction curve,
we have
\begin{equation}
\mp(u^*-u_k)=\int_{p_k}^{p^*} \dfrac{1}{\rho c}\mathrm d p,
\label{eq:diffu}
\end{equation}
where the density $\rho$ is expressed in terms of $p$ by solving the isentropic
relation
\begin{equation}
\dfrac{\partial p}{\partial \rho}=c^2(p,\rho).
\label{eq:isenr}
\end{equation}
Similarly, the right hand side of \eqref{eq:diffu} can be expressed as a function of $p^*$ alone. Formally we define
\[
f_k(p^*) = \int^{p^*}_{p_k} \dfrac{1}{\rho c}\mathrm d p,\quad p^*\le p_k.
\]

\end{itemize}

Collecting both cases above, we have that 
\[
u^*-u_l=-f_l(p^*) \quad \text{ and } \quad u^*-u_r=f_r(p^*).
\] 
Therefore, the interface pressure $p^*$ is the zero of the
following \emph{pressure function}
\[
f(p) := f_l(p)+f_r(p) + u_r - u_l.
\]
And the interface velocity $u^*$ can be determined by
\[
u^*=\dfrac{1}{2}(u_l+u_r+f_r(p^*)-f_l(p^*)).
\]
Recall that the formula of the function $f_k(p)$ is given by
\[
f_k(p)=
\begin{cases}
\displaystyle\int^{p}_{p_k} \dfrac{1}{\rho c}\mathrm d p, & p\le p_k,\\
\left((p-p_k)\left(\dfrac{1}{\rho_k}-\dfrac{1}{\rho}\right)
\right)^{1/2}, & p>p_k,
\end{cases}
\]
where $\rho$ is determined through either the Hugoniot relation \eqref{eq:eose}
or the isentropic relation \eqref{eq:isenr} for a given $p$. We claim on
$f_k(p)$ that
\begin{lemma}\label{thm:propf}
  Assume that the conditions \textbf{(C1), (C2)} and \textbf{(C3)}
  hold for $\varGamma_k(\rho)$ and $h_k(\rho)$, the function $f_k(p)$ is
  monotonically increasing and concave, i.e.
  \[
  f_k'(p)>0 \quad \text{ and } \quad f_k''(p)<0,
  \]
  if the Hugoniot function is concave with respect to the density,
  i.e. ${\partial^2 \varPhi_k(p,\rho)}/{\partial \rho^2}<0$.
\end{lemma}
\begin{proof}
  The first and second derivatives of $f_k(p)$ can be found to be
  \[
    f_k'(p)=
    \begin{cases}
      \dfrac{1}{\rho c}, & p\le p_k, \\
      \dfrac{1}{2f_k(p)} \left(\dfrac{1}{\rho_k}-\dfrac{1}{\rho}
        +\dfrac{p-p_k}{\rho^2\chi}\right), & p>p_k,
    \end{cases}
  \]
  and
  \[
    f_k''(p)=
    \begin{cases}
      -\dfrac{\mathscr G}{\rho^2c^3}, &p\le p_k, \\
      \begin{aligned}
      &-\dfrac{1}{4f_k^3(p)}\left(
        \dfrac{2(p-p_k)^2}{\rho^2\chi^2}\left(\dfrac{1}{\rho_k}-\dfrac{1}{\rho
          }\right)
        \left(\dfrac{2}{\rho}+\left.\dfrac{\partial\chi}{\partial p}\right|_{\varPhi_k}\right)
        \right.\\
        &+
        \left.
        \left(\dfrac{1}{\rho_k}-\dfrac{1}{\rho}-\dfrac{p-p_k}{\rho^2\chi}
        \right)^2
      \right)
      , 
      \end{aligned}
      & p>p_k,
    \end{cases}
  \]
  where $\mathscr G$ is the fundamental derivative \eqref{eq:fundd} and
  \[
  \left.\dfrac{\partial\chi}{\partial p}\right|_{\varPhi_k}
  =
  \left(\dfrac{\partial\varPhi_k}{\partial\rho}(p,\rho)\right)^{-1}
  \left(
    \dfrac{\partial^2\varPhi_k}{\partial\rho^2}(p,\rho)+\chi\varGamma_k(\rho_k)
    (\rho_k\varGamma_k'(\rho)-(\rho\varGamma_k(\rho))')
  \right).
  \]
  The result then follows by direct observation.
\end{proof}
The behavior of $f_k(p)$ is related to the existence and uniqueness of the
solution of the Riemann problem. The existence and uniqueness of the Riemann
solution for gas dynamics under appropriate conditions have been established by
Liu \cite{Liu1975} and by Smith \cite{Smith1979}. It is easy to show that the
conditions \textbf{(C1), (C2)} and \textbf{(C3)} imply Smith's ``strong''
condition $\partial e(p,\rho)/\partial \rho<0$. However, for completeness we
provide a short proof of the results for the case of Mie-Gr{\"u}neisen
EOS in the following theorem.
\begin{theorem}\label{thm:unique}
  Assume that the Mie-Gr{\"u}neisen EOS \eqref{eq:particulareos}
  satisfies the conditions \textbf{(C1), (C2)} and \textbf{(C3)}. The
  Riemann problem \eqref{system:oneriemann} admits a unique solution
  (in the class of admissible shocks, interfaces and rarefaction waves
  separating constant states) if and only if the initial states
  satisfy the constraint
  \begin{equation}
    u_r-u_l<\int_{0}^{p_l}\dfrac{\mathrm dp}{\rho c}+
    \int_{0}^{p_r}\dfrac{\mathrm dp}{\rho c}.
    \label{eq:vacuum}
  \end{equation}
\end{theorem}
\begin{proof}
  By definition and Lemma \ref{thm:propf} we know that the pressure
  function $f(p)$ is monotonically increasing.  Next we study the
  behavior of $f(p)$ when $p$ tends to infinity. Let $\tilde\rho$
  represents the density such that $\varPhi_k(2p_k,\tilde\rho)=0$.
  When the pressure $p>2p_k$, we have $\rho>\tilde{\rho}$, and thus
  \[
  f_k^2(p)=(p-p_k)\left(\dfrac{1}{\rho_k}-\dfrac{1}{\rho}\right)
  >(p-p_k)\left(\dfrac{1}{\rho_k}-\dfrac{1}{\tilde\rho}\right).
  \]
  As a result, $f_k(p)$ tends to positive infinity as
  $p\rightarrow +\infty$ and so does $f(p)$.
  
  Based on the behavior of the function $f(p)$, a necessary and
  sufficient condition for the interface pressure $p^*>0$ such that
  $f(p^*)=0$ to be uniquely defined is given by
  \[
  f(0)=f_l(0)+f_r(0)+u_r-u_l<0,
  \]
  or equivalently, the constraint given by \eqref{eq:vacuum}. This
  completes the proof.
\end{proof}

\begin{remark}
  When the initial states violate the constraint \eqref{eq:vacuum},
  the Riemann problem \eqref{system:oneriemann} has no solution in the
  above sense. One can yet define a solution by introducing a
  \emph{vacuum}. However, we are not going to address this issue
  since it is beyond the scope of our current interests.
\end{remark}


\section{Approximate Riemann Solver}
\label{sec:aps}

The solution of the Riemann problem \eqref{system:oneriemann} is obtained by
finding the unique zero $p^*$ of the function $f(p)$. A first try to this
problem is to use the Newton-Raphson iteration as
\begin{equation}
p_{n+1}=p_n-\dfrac{f(p_n)}{f'(p_n)}
=p_n-\dfrac{f_l(p_n)+f_r(p_n) + u_r -
u_l}{f_l'(p_n)+f_r'(p_n)},
\label{eq:iterp1}
\end{equation}
with a suitable initial estimate which we may choose as, for example, the
acoustic approximation \cite{Godunov1976}
\[
p_0 =
\dfrac{\rho_lc_lp_r+\rho_rc_rp_l+\rho_lc_l\rho_rc_r(u_l-u_r)}
{\rho_lc_l+\rho_rc_r}.
\]
The concavity of the pressure function $f(p)$ leads to the following global
convergence of the Newton-Raphson iteration
\begin{corollary}
  The Newton-Raphson iteration for \eqref{eq:iterp1} converges if
  $\varGamma_l''(\rho)=\varGamma_r''(\rho)=0$.
\end{corollary}
Unfortunately, there is no closed-form expression for the pressure function
$f(p)$ or its derivative $f'(p)$ for equations of state such as polynomial EOS,
JWL EOS or Cochran-Chan EOS. Therefore, we have to implement the iteration
\eqref{eq:iterp1} approximately. Here we adopt the \emph{inexact Newton method}
\cite{Dembo1982} instead, which is formulated as
\begin{equation}
\left\{
\begin{aligned}
p_{n+1}&=p_n-\dfrac{F_n}{F_n'}=p_n-\dfrac{F_{n,l}+F_{n,r}+u_r-u_l}
{F'_{n,l}+F'_{n,r}},\\
u_{n}&=\dfrac{1}{2}(u_l+u_r+F_{n,r}-F_{n,l}),
\end{aligned}\right.
\label{eq:iterp2}
\end{equation}
where $F_{n,k}$ and $F'_{n,k}$ approximate $f_k(p_{n})$ and $f_k'(p_{n})$,
respectively.

To specify the sequences $\{F_{n,k}\}$ and $\{F'_{n,k}\}$, we solve the Hugoniot
loci through the numerical iteration and the isentropic curves by the numerical
integration. It is natural to expect that the sequences $p_n$ and $u_n$ will
tend to $p^*$ and $u^*$ respectively whenever the evaluation errors
$|F_{n,k}-f_k(p_n)|$ and $|F'_{n,k}-f'_k(p_n)|$ are sufficiently small. As a
preliminary result, let us introduce the following Lemma \ref{the:conv_p}, which
states the local convergence of inexact Newton iterates
\begin{lemma}
  \label{the:conv_p}
  If the initial iterate $p_0$ is sufficiently close to $p^*$, and the 
  evaluation errors of $f(p_n)$ and $f'(p_n)$ satisfy the following 
  constraint
  \[
  2|F_n-f(p_n)|+|\Delta p_n|\cdot|F'_n-f'(p_n)|\le \eta |F_n|,
  \]
  where $\Delta p_n=p_{n+1}-p_n=-F_n/F'_n$ denotes the step increment
  and $\eta\in (0,1)$ is a fixed constant, then the sequence of
  inexact Newton iterates $p_n$ defined by
  \[
  p_{n+1}=p_n-\dfrac{F_n}{F'_n},
  \]
  converges to $p^*$. Moreover, the convergence is linear in the sense
  that $|p_{n+1}-p^*|\le \zeta |p_n-p^*|$ for $\zeta \in (\eta,1)$.
\end{lemma}
\begin{proof}
  Since $\eta<\zeta$, there exists $\gamma>0$ sufficiently small that
  $(1+\gamma)((\eta+2)\gamma+\eta)\le \zeta$.
  Choose $\epsilon>0$ sufficiently small that
  \begin{itemize}
  \item[(1).] $|f'(p)-f'(p^*)|\le \gamma |f'(p^*)|$;
  \item[(2).] $|f'(p)^{-1}-f'(p^*)^{-1}|\le \gamma |f'(p^*)|^{-1}$;
  \item[(3).] $|f(p)-f(p^*)-f'(p^*)(p-p^*)|\le \gamma |f'(p^*)||p-p^*|$,
  \end{itemize}
  whenever $|p-p^*|\le \epsilon$. Now we prove the convergence rate by
  induction. Let the initial solution satisfy $|p_0-p^*|\le
  \epsilon$. Suppose that $|p_n-p^*|\le \epsilon$, then
  \[
  \begin{split}
  |f(p_n)|&=|f(p_n)-f(p^*)-f'(p^*)(p_n-p^*)+f'(p^*)(p_n-p^*)| \\
  &\le (1+\gamma)|f'(p^*)||p_n-p^*|.
  \end{split}
  \]
  The error of $(n+1)$-th iterate can be written as
  \[
  \begin{split}
  p_{n+1}-p^*&= f'(p_n)^{-1}
  (r_n+(f'(p_n)-f'(p^*))(p_n-p^*)\\
  &-(f(p_n)-f(p^*)-f'(p^*)(p_n-p^*))),
  \end{split}
  \]
  where the residual $r_n=f'(p_n)\Delta p_n+f(p_n)$ satisfies 
  \[
  \begin{split}
    |r_n|&=|(f'(p_n)-F_n')\Delta p_n+f(p_n)-F_n|\\
    &=|(f'(p_n)-F_n')\Delta p_n+(1+\eta)(f(p_n)-F_n)
    -\eta(f(p_n)-F_n)|
    \le \eta |f(p_n)|.
  \end{split}
  \]
  Therefore
  \[
  \begin{split}
    |p_{n+1}-p^*|&\le  (1+\gamma)|f'(p^*)|^{-1}
    (\eta|f(p_n)|+\gamma |f'(p^*)||p_n-p^*| \\
    & + \gamma
    |f'(p^*)||p_n-p^*|)\\
    &\le (1+\gamma)(\eta(1+\gamma)+2\gamma)|p_n-p^*|\\
    &\le \zeta |p_n-p^*|,
  \end{split}
  \]
  and hence $|p_{n+1}-p^*|\le \epsilon$.
  It follows that $p_n$ converges linearly to $p^*$.
\end{proof}

Recall that $f(p)$ and $f'(p)$ can be decomposed as
\[
f(p)=f_l(p)+f_r(p)+u_r-u_l
\an 
f'(p)=f_l'(p)+f_r'(p).
\]
The following Theorem \ref{the:conv_pu} tells us that if $f_k$ and
$f_k'$ are evaluated accurately enough, the convergences of iterates
$p_n$ and $u_n$ are ensured.
\begin{theorem}\label{the:conv_pu}
  Suppose that the initial estimate $p_0$ is sufficiently close to
  $p^*$.  If the evaluation errors of $f_k(p_n)$ and $f'_k(p_n)$
  satisfy
  \[
  |F_{n,k}-f_k(p_n)|\le \dfrac{1}{6}\eta |F_n|
  \an 
  |F'_{n,k}-f_k'(p_n)|\le \dfrac{1}{6}\eta |F_n'|,
  \]
  where $\eta\in (0,1)$ is a constant.  Moreover, assume that the
  sequence $\{F_n'\}$ is bounded. Then the sequences of pressure and
  velocity defined in \eqref{eq:iterp2} converge, namely
  \[
  p_n\rightarrow p^*\an u_n\rightarrow u^*.
  \]
\end{theorem}
\begin{proof}
  Since 
  \begin{align*}
    |F_n-f(p_n)|\le |F_{n,l}-f_l(p_n)|+ |F_{n,r}-f_r(p_n)|
    \le \dfrac{1}{3}\eta |F_n|,\\
    |F'_n-f'(p_n)|\le |F'_{n,l}-f'_l(p_n)|+
    |F'_{n,r}-f'_r(p_n)|
    \le \dfrac{1}{3}\eta |F'_n|,
  \end{align*}
  we have
  \[
  2|F_n-f(p_n)|+|\Delta p_n|\cdot|F'_n-f'(p_n)|\le \dfrac{2}{3}\eta
  |F_n|+ \left|\dfrac{F_n}{F'_n}\right|\cdot
  \dfrac{1}{3}\eta|F'_n|=\eta |F_n|.
  \]
  From Lemma \ref{the:conv_p} we conclude that $p_n\rightarrow p^*$.
  Due to the continuity of $f_k$ we have
  $f_k(p_n)\rightarrow f_k(p^*)$.  From $F_n=(p_n-p_{n+1})F_n'$ and
  the boundness of $\{F'_n\}$ we know that $F_n\rightarrow 0$.  It
  follows that
  \[
  |F_{n,k}-f_k(p^*)|\le 
  |F_{n,k}-f_k(p_n)|+|f_k(p_n)-f_k(p^*)|\rightarrow 0.
  \]
  Finally, by the definition of $u_n$ and $u^*$ we have
  \[
  |u_n-u^*|=\dfrac{1}{2}\left| 
    (F_{n,r}-f_r(p^*)) - (F_{n,l} - f_l(p^*))
  \right|\rightarrow 0,
  \]
  or $u_n\rightarrow u^*$.
\end{proof}

By Theorem \ref{the:conv_pu}, the convergence is guaranteed by \textit{a
posteriori} control on the evaluation errors of $f_k(p_n)$ and $f'_k(p_n)$. The
evaluation errors of $f_k(p_n)$ and $f'_k(p_n)$ depend on the residual of the
algebraic equation as well as the truncation error of the ordinary differential
equation. Therefore, to achieve better accuracy one may effectively reduce the
residual term and the truncation error. To achieve a trade-off between the
accuracy and efficiency in a practical implementation, we apply the
Newton-Raphson method to solve the Hugoniot loci, and the Runge-Kutta method to
solve the isentropic curves.

Precisely, for the given $n$-th iterator $p_n$, if $p_n>p_k$, then we
solve the following algebraic equation
\begin{equation}
\varPhi_k(p_n,\tilde\rho_{n,k})=0,
\label{eq:algphi}
\end{equation}
to obtain $\tilde\rho_{n,k}$ to a prescribed tolerance with the
Newton-Raphson method
\[
\rho_{n,k,m+1}=\rho_{n,k,m}- \dfrac{\varPhi_k(p_n,\rho_{n,k,m})}
{\partial\varPhi_k(p_n,\rho_{n,k,m})/\partial\rho}.
\]
By Lemma \ref{thm:phi}, we arrive at the following convergence result
\begin{corollary}
  The Newton-Raphson iteration for \eqref{eq:algphi} must converge if
  $\varGamma''_k(\rho)=0$.
\end{corollary}
The values of $F_{n,k}$ and $F'_{n,k}$ for the shock branch are thus
taken as
\begin{align}
\label{eq:shockf}
  F_{n,k}&=\left((p_n-p_k)\left(\dfrac{1}{\rho_k}-\dfrac{1}
           {\tilde\rho_{n,k}}\right)\right)^{1/2},\\ 
\label{eq:shockdf}
  F'_{n,k}&=\dfrac{1}{2F_{n,k}}\left(
            \dfrac{1}{\rho_k}-\dfrac{1}{\tilde\rho_{n,k}}
            +\dfrac{p_n-p_k}{\rho_{n,k}^2\chi(p_n,\tilde\rho_{n,k})}
            \right).
\end{align}

If, on the other hand, $p_n\le p_k$, then we solve the following
initial value problem
 \begin{align*}
   \dfrac{\mathrm d f_k}{\mathrm d p}=\dfrac{1}{\rho c}, \\
   f_k|_{p=p_k}=0,
\end{align*}
coupled with the initial value problem of the isentropic relations
 \begin{align*}
 \dfrac{\mathrm d\rho}{\mathrm d p}=\dfrac{1}{c^2},\\
 \rho|_{p=p_k}=\rho_k,
 \end{align*}
backwards until $p=p_n$ using fourth-order Runge-Kutta method. The
values of $F_{n,k}$ and $F'_{n,k}$ are then taken as
\begin{align}
  \label{eq:raref}
  F_{n,k} &= -\dfrac{1}{6}(p_k-p_n)
            \left(\dfrac{1}{Z_1}+\dfrac{2}{Z_2}
            +\dfrac{2}{Z_3}+\dfrac{1}{Z_4}
            \right), \\ 
  \label{eq:raredf}
  F'_{n,k} &=\dfrac{1}{\tilde\rho_{n,k} c(p_n,\tilde\rho_{n,k})},
\end{align}
where  
\[
  \tilde\rho_{n,k} = \rho_k-\dfrac{1}{6}(p_k-p_n)                
\left(\dfrac{1}{c_1^2}+\dfrac{2}{c_2^2}+\dfrac{2}{c_3^2}
                     +\dfrac{1}{c_4^2}\right),\\ 
\]
and the intermediate states are given by
\[
\begin{aligned}
  c_1&=c(p_k,\rho_k),& Z_1&=\rho_kc_1,\\
  c_2&=c\left(\dfrac{p_k+p_n}{2},\rho_k-\dfrac{p_k-p_n}{2c_1^2}\right),
  & Z_2&= \left(\rho_k-\dfrac{p_k-p_n}{2c_1^2}\right)c_2,\\
  c_3&=c\left(\dfrac{p_k+p_n}{2},\rho_k-\dfrac{p_k-p_n}{2c_2^2}\right),
  & Z_3&= \left(\rho_k-\dfrac{p_k-p_n}{2c_2^2}\right)c_3,  \\
  c_4&=c\left(p_n,\rho_k-\dfrac{p_k-p_n}{c_3^2}\right), & Z_4&=
  \left(\rho_k-\dfrac{p_k-p_n}{c_3^2}\right)c_4.
\end{aligned}
\]
Finally the procedure of the approximate Riemann solver for
\eqref{system:oneriemann} is listed below.
\begin{mdframed}
  \setlength{\parindent}{0pt}
  \textbf{Step 1} Provide initial estimates of the interface
  pressure
  \[
  p_0 =
  \dfrac{\rho_lc_lp_r+\rho_rc_rp_l+\rho_lc_l\rho_rc_r(u_l-u_r)}
  {\rho_lc_l+\rho_rc_r}.
  \]
  
  \textbf{Step 2} Assume that the $n$-th iterator $p_n$ is
  obtained. Determine the types of left and right nonlinear
  waves.
  
  (i) If $p_n>\max\{p_l,p_r\}$, then both nonlinear waves
  are shocks.
  
  (ii) If $\min\{p_l,p_r\}\le p_n \le \max\{p_l,p_r\}$, then
  one of the two nonlinear waves is a shock, and the other is a
  rarefaction wave.
  
  (iii) If $p_n<\min\{p_l,p_r\}$, then both nonlinear waves
  are
  rarefaction waves.
  
  \textbf{Step 3} Evaluate $F_{n,k}$ and $F'_{n,k}$ through
  \eqref{eq:shockf}~\eqref{eq:shockdf} or
  \eqref{eq:raref}~\eqref{eq:raredf}
  according to the wave structures and update the interface
  pressure 
  through
  \[
  p_{n+1}=p_n-\dfrac{F_{n,l}+F_{n,r} + 
    u_r - u_l}{F'_{n,l}+F'_{n,r}}.
  \]
  
  \textbf{Step 4} Terminate whenever the relative change of the
  pressure reaches the prescribed tolerance. The sufficiently accurate estimate
  $p_n$ is then taken as the approximate interface pressure $p^*$.
  Otherwise return to Step 2.
  
  \textbf{Step 5} Compute the interface velocity
  $u^*$ through
  \[
  u^*=\dfrac{1}{2}\left(u_l+u_r+F_{n,r} - F_{n,l}\right).
  \]
\end{mdframed}
\medskip

In the above algorithm the tolerances of the outer iteration for
pressure function and the inner iteration for Hugoniot function are
both set as $10^{-8}$. Although the evaluation errors of $f_k(p_n)$
and $f'_k(p_n)$ for the isentropic branch is not effectively
controlled in our implementation without applying a more accurate
numerical quadrature rule, we have never encountered any failure in
the convergence of inexact Newton iteration in the numerical tests.


\section{Application on Two-medium Flows}
\label{sec:model} 
We are considering the two-medium fluid flow described by an immiscible model in
the domain $\Omega$. Two fluids are separated by a sharp interface $\Gamma(t)$
characterized by the zero of the level set function $\phi(\bm x,t)$ which
satisfies
\begin{equation}
\dfrac{\partial \phi}{\partial t} +
\tilde u|\nabla\phi| = 0.
\label{eq:levelset:eov}
\end{equation}
Here $\tilde u$ denotes the normal velocity of the interface, which is
determined by the solution of multi-medium Riemann problem. And the normal direction
is chosen as the gradient of the level set function.
The region occupied by each fluid can be expressed in terms of the level set
function $\phi(\bm x,t)$
\[
\Omega^+(t):=\{\bm x\in\Omega~|~\phi(\bm x,t)> 0\}
\an
\Omega^-(t):=\{\bm x\in\Omega~|~\phi(\bm x,t)< 0\}.
\]
And the fluid in each region is governed by the Euler equations
\begin{equation} 
  \dfrac{\partial \bm{U}}{\partial t} + 
  \nabla\cdot \bm F(\bm U)= 
  \bm 0,\quad \bm x\in \Omega^\pm (t),
  \label{eq:euler} 
\end{equation}
where
\[
\bm{U}= 
\begin{bmatrix} 
\rho \\ 
\rho \bm u \\
E  
\end{bmatrix}\an
\bm{F(\bm U)}=
\begin{bmatrix}
\rho \bm u^\top \\ 
\rho \bm u\otimes \bm u+p\mathbf I\\
(E+p)\bm u^\top
\end{bmatrix}.
\]
Here $\bm u$ stands for the velocity vector, and other variables represent the
same as that in \eqref{system:oneriemann}. To provide closure for the above
Euler equations, we need to specify the equation of state for each fluid, i.e.
\[
p=
\begin{cases}
P^+(\rho,e), & \phi>0,\\
P^-(\rho,e), & \phi<0.
\end{cases}
\]
The specific forms of the corresponding equations of state are the
Mie-Gr\"uneisen EOS discussed in Section \ref{sec:riemann}.

The level set equation \eqref{eq:levelset:eov} and the Euler equations
\eqref{eq:euler} are coupled in the sense that the level set equation provides
the boundary geometry for the Euler equations, whereas the Euler equations
provide the interface motion for the level set equation. To solve the coupled
system, we use the numerical scheme following Guo \textit{et al.}
\cite{Guo2016}, which is implemented on an Eulerian grid. For completeness,
however, we briefly sketch the main steps of numerical scheme for the two-medium
flow therein. The approximate Riemann solver we proposed is applied to calculate
the numerical flux at the interface in the numerical scheme.

As a numerical scheme on an Eulerian grid, the whole domain $\Omega$ is divided
into a conforming mesh with simplex cells, and may be refined or coarsened
during the computation based on the $h$-adaptive mesh refinement strategy of Li
and Wu \cite{Li2013}. The whole scheme is divided into three steps:

\begin{itemize}
\item[(1).] {\bf Evolution of interface}

The level set function is approximated by a continuous piecewise-linear
function. The discretized level set function \eqref{eq:levelset:eov} is updated
through the characteristic line tracking method once the motion of the interface
is given. Due to the nature of level set equation, it remains to specify the
normal velocity $\tilde u$ within a narrow band near the interface. This can be
achieved by firstly solving a multi-medium Riemann problem across the interface
and then extending the velocity field to the nearby region using the harmonic
extension technique of Di \textit{et al.} \cite{Di2007}. The solution of the
multi-medium Riemann problem has been elaborated in Section \ref{sec:riemann}.

In order to keep the property of signed distance function, we solve the
following reinitialization equation 
\[
\begin{cases}
  \dfrac{\partial \psi}{\partial \tau} =
  \sgn(\psi_0)\cdot\left(1-\left|\nabla \psi \right|\right), \\
  \psi(\bm x,0)=\psi_0=\phi(\bm x, t),
\end{cases}
\]
until steady state using the explicitly positive coefficient scheme
\cite{Di2007}.

Once the level set function is updated until $n$-th time level, we can obtain
the discretized phase interface $\Gamma_{h,n}$. A cell $K_{i,n}$ is called an
\emph{interface cell} if the intersection of $K_{i,n}$ and $\Gamma_{h,n}$,
denoted as $\Gamma_{K_{i,n}}$, is nonempty. Since the level set function is
piecewise-linear and the cell is simplex, $\Gamma_{K_{i,n}}$ must be a linear
manifold in $K_{i,n}$. The interface $\Gamma_{h,n}$ further cuts the cell
$K_{i,n}$ and one of its boundaries $S_{ij,n}$ into two parts, which are
represented as $K_{i,n}^\pm$ and $S_{ij,n}^\pm$ respectively (may be an empty
set). The unit normal of $\Gamma_{K_{i,n}}$, pointing from $K_{i,n}^-$ to
$K_{i,n}^+$, is denoted as $\bm n_{K_{i,n}}$. These quantities can be readily
computed from the geometries of the interface and cells. See Fig.
\ref{fig:twophasemodel} for an illustration.

\begin{figure}[htb]
\centering
\includegraphics[width=.7\textwidth]{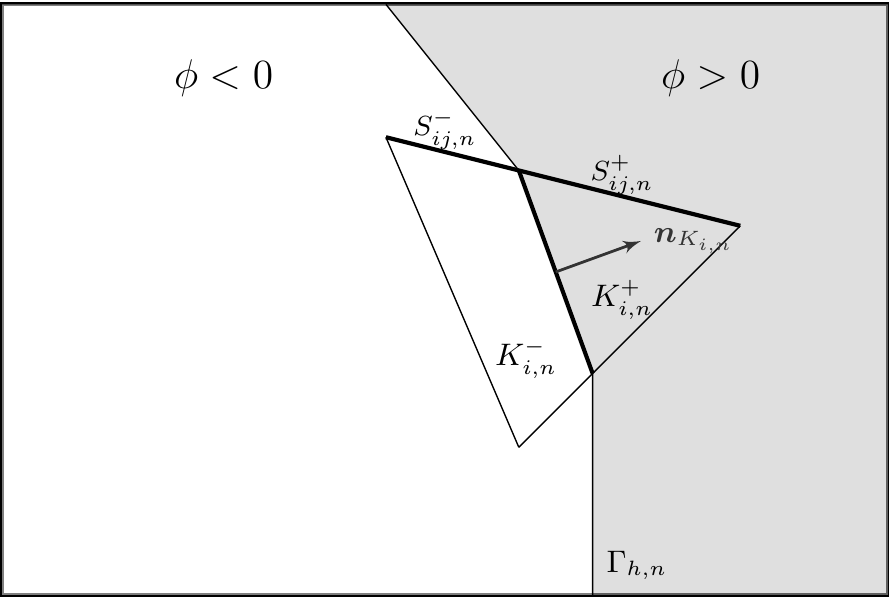}
\caption{Illustration of the two-medium fluid model.}
\label{fig:twophasemodel}
\end{figure}

\item[(2).] {\bf Numerical flux}

The numerical flux for the two-medium flow is composed of two parts: the cell
edge flux and the interface flux. Below we explain the flux contribution towards
any given cell $K_{i,n}$. We introduce two sets of flow variables at $n$-th time
level
\[
\bm U_{K_{i,n}}^\pm = \left[\begin{array}{c}
\rho_{K_{i,n}}^\pm \\ \rho_{K_{i,n}}^\pm\bm u_{K_{i,n}}^\pm \\
E_{K_{i,n}}^\pm 
\end{array}\right],
\]
which refer to the constant states in the cell $K_{i,n}^\pm$. Note that the flow
variables vanish if there is no corresponding fluid in a given cell. 

\begin{itemize}
\item {\bf Cell edge flux}

The cell edge flux is the exchange of flux between the same fluid across the
cell boundary. For any edge $S_{ij}$ of the cell $K_{i,n}$, let $\bm n_{ij,n}$
be the unit normal pointing from $K_{i,n}$ into the adjacent cell $K_{j,n}$. The
cell edge flux across $S_{ij,n}^\pm$ is calculated as
\begin{equation}\label{eq:cell_bound_flux}
  \hat{\bm F}_{ij,n}^\pm = \Delta t_n
  \left|S_{ij,n}^\pm \right| 
  \hat{\bm F} \left(
    \bm U_{K_{i,n}}^\pm, \bm U_{K_{j,n}}^\pm; 
    \bm n_{ij,n} \right),
\end{equation}
where $\Delta t_n$ denotes the current time step length, and $\hat{\bm F}(\bm
U_l, \bm U_r; \bm n)$ is a consistent monotonic numerical flux along $\bm n$.
Here we adopt the local Lax-Friedrich flux
\[
\hat{\bm F}(\bm U_l, \bm U_r; \bm n) = \dfrac{1}{2} \left(\bm F(\bm U_l) +
  \bm F(\bm U_r)\right) \cdot \bm n - \dfrac{1}{2}\lambda (\bm U_r - \bm U_l),
\]
where $\lambda$ is the maximal signal speed over $\bm U_l$ and $\bm U_r$.

\item {\bf Interface flux}
The interface flux is the exchange of the flux between two fluids due to the
interaction of fluids at the interface. If $K_{i,n}$ is an interface cell, then
the flux across the interface can be approximated by
\begin{equation}
\hat{\bm F}_{K_{i,n}}^{\pm}=
\Delta t_n\left|\Gamma_{K_{i,n}}\right|
\begin{bmatrix}
0 \\ p_{K_{i,n}}^*\bm n_{K_{i,n}} \\
p_{K_{i,n}}^*u_{K_{i,n}}^*
\end{bmatrix}.
\label{eq:phaseflux}
\end{equation}
Here $p_{K_{i,n}}^*$ and $u_{K_{i,n}}^*$ are the interface pressure and normal
velocity, which are obtained by applying the approximate solver we proposed in
Section \ref{sec:aps} to a local one-dimensional Riemann problem in the normal
direction of the interface with initial states
\[
\begin{split}
[\rho_l,u_l,p_l]^\top=
[\rho_{K_{i,n}}^-,
 \bm u_{K_{i,n}}^-\cdot \bm n_{K_{i,n}},
  p_{K_{i,n}}^-]^\top,\\
[\rho_r,u_r,p_r]^\top=
[\rho_{K_{i,n}}^+,
 \bm u_{K_{i,n}}^+\cdot \bm n_{K_{i,n}},
  p_{K_{i,n}}^+]^\top.
 \end{split}
\]
Here the pressure $p_{K_{i,n}}^\pm$ in the initial states are given
through the corresponding equations of state, namely
\[
p_{K_{i,n}}^\pm=P^\pm (\rho_{K_{i,n}}^\pm,e_{K_{i,n}}^\pm),\quad
e_{K_{i,n}}^\pm=\dfrac{E_{K_{i,n}}^\pm}{\rho_{K_{i,n}}^\pm}
-\dfrac{1}{2}\|\bm u_{K_{i,n}}^\pm\|^2.
\]
\end{itemize}

\item[(3).] {\bf Update of conservative variables}

Once the edge flux \eqref{eq:cell_bound_flux} and interface flux
\eqref{eq:phaseflux} are computed, the conservative variables at $(n+1)$-th 
time level is thus assigned as
\[
\bm U^{\pm}_{K_{i,n+1}} \!=\! \left\{\begin{array}{ll}
\bm 0, & K^\pm_{i,n+1} \!=\! \varnothing, \\
\dfrac{1}{|K^\pm_{i,n+1}|} \!
\left(|K^\pm_{i,n}|
\bm U^\pm_{K_{i,n}} \!+\! 
\displaystyle\sum_{S_{ij}^\pm\subseteq \partial K_{i,n}^\pm}
\hat{\bm F}_{ij,n}^\pm \!+\!
\hat{\bm F}^\pm_{K_{i,n}}\right), 
& K^\pm_{i,n+1} \!\neq\! \varnothing.
\end{array}\right.
\]
\end{itemize}

Basically, the steps we presented above may close the numerical scheme, while
there are more details in the practical implementation to gurantee the stability
of the scheme. Please see \cite{Guo2016} for those details.


\section{Numerical Examples}\label{sec:num}

In this section we present several numerical examples to validate our schemes,
including one-dimensional Riemann problems, spherically symmetric problems and
multi-dimensional shock impact problems. One-dimensional simulations are carried
out on uniform interval meshes, while two and three-dimensional simulations are
carried out on unstructured triangular and tetrahedral meshes respectively.

\subsection{One-dimensional problems}

In this part, we present some numerical examples of one-dimensional Riemann
problems. The computational domain is $[0,1]$ with $400$ cells. And the
reference solution, if mentioned, is computed on a very fine mesh with $10^4$
cells.

\subsubsection{Shyue shock tube problem}

This is a single-medium JWL Riemann problem used by Shyue \cite{Shyue2001}. The
parameters of JWL EOS \eqref{eq:jwl} therein are given by
$A_1=\unit[8.545\times10^{11}]{Pa}, A_2=\unit[2.050\times10^{10}]{Pa}$,
$\omega=0.25$, $R_1=4.6$, $R_2=1.35$ and $\rho_0=\unit[1840]{kg/m^3}$.
The initial conditions are
\[
 [\rho, u, p]^\top = \left\{
    \begin{array}{ll}
      [1700, ~0, ~10^{12}]^\top, &x<0.5,\\ [2mm]
      [1000, ~0, ~5.0\times 10^{10}]^\top, & x > 0.5.
    \end{array}
  \right.
\] 
The simulation terminates at $t=1.2\times 10^{-5}$. In Fig. \ref{res:rm_shyue},
the top panel shows the results of our numerical methods, while the bottom panel
contains the results of Shyue \cite{Shyue2001}. Our results agree well with the
highly resolved solution shown in a solid line given by Shyue.

\begin{figure}[htbp]
\centering
\subfloat
{\includegraphics[width=0.3\textwidth]{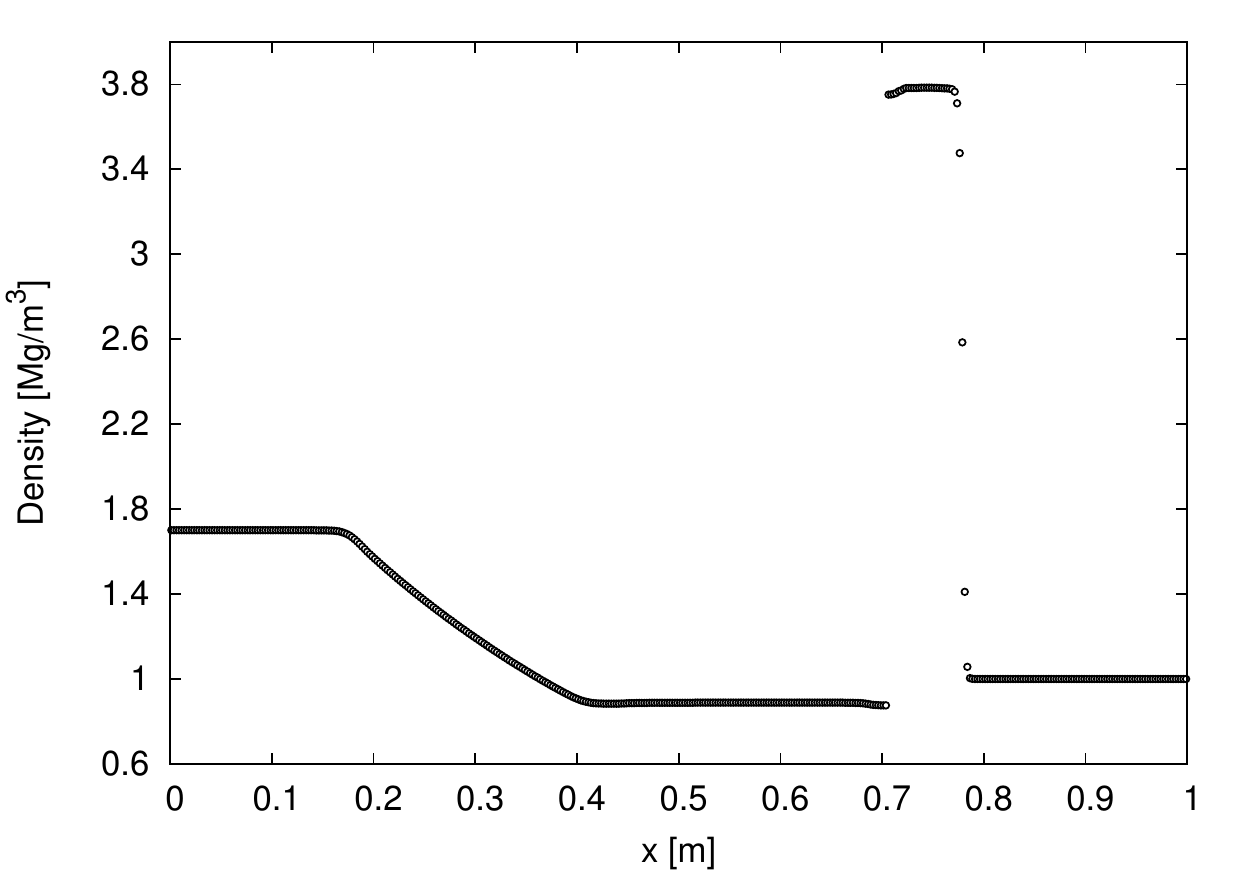}}
\subfloat
{\includegraphics[width=0.3\textwidth] {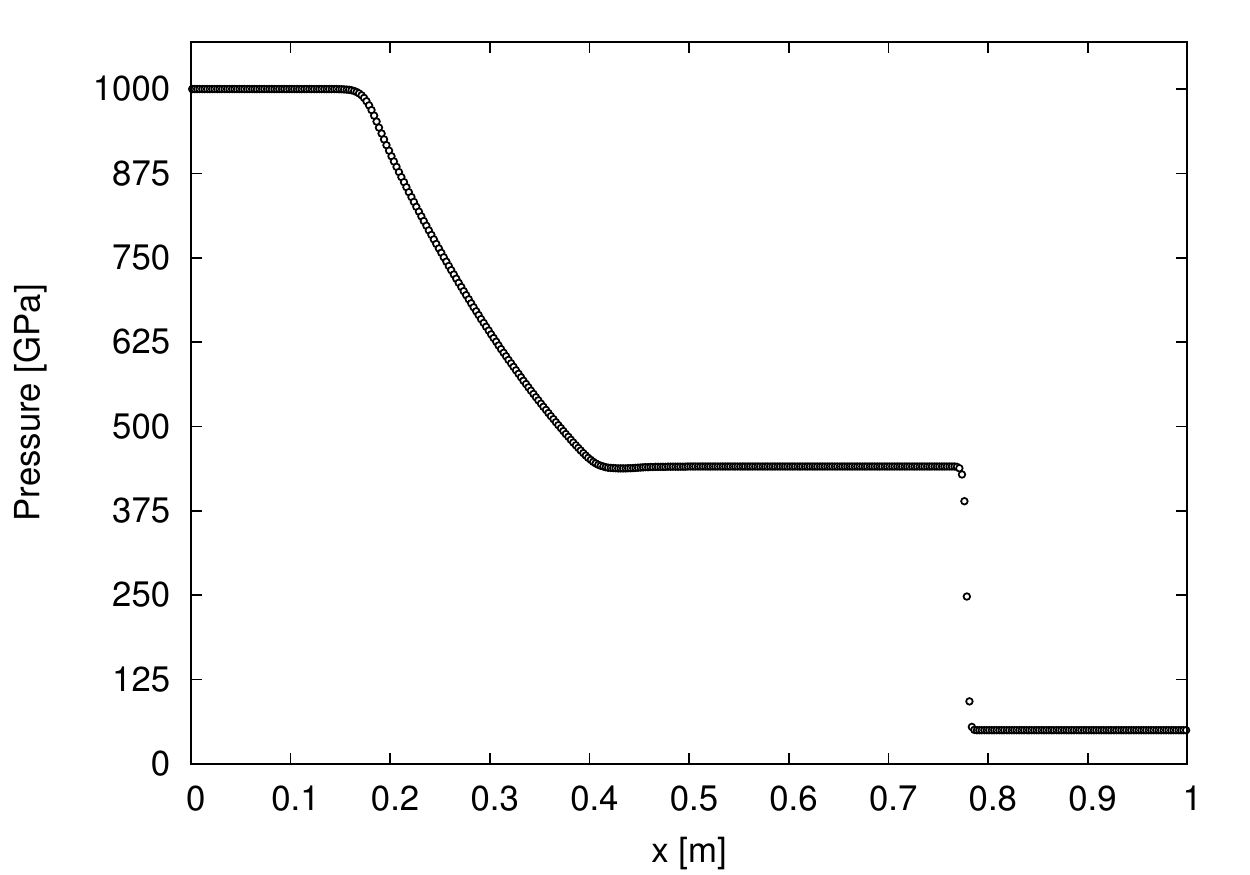}}
\subfloat
{\includegraphics[width=0.3\textwidth]{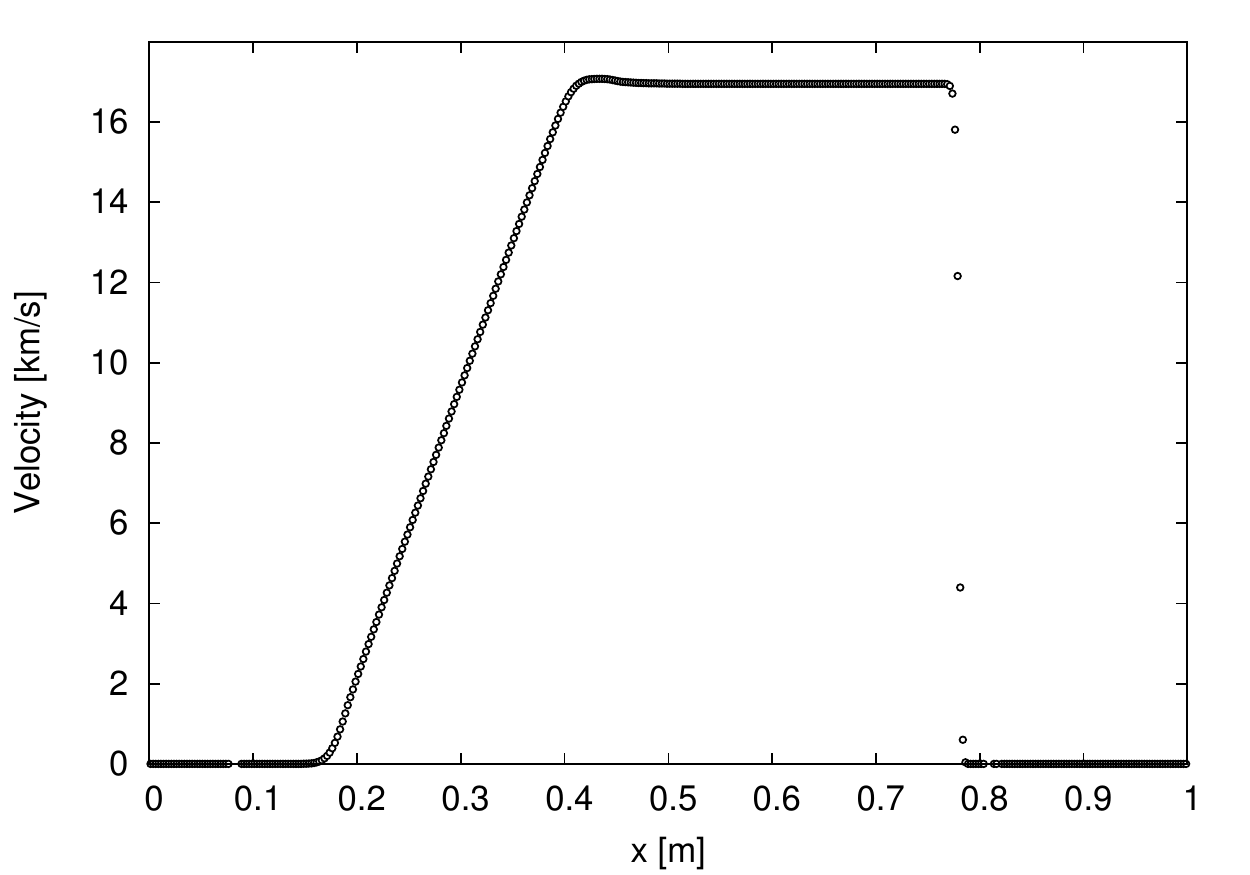}} \\
\addtocounter{subfigure}{-3}
\subfloat[Density]
{\includegraphics[width=0.3\textwidth]{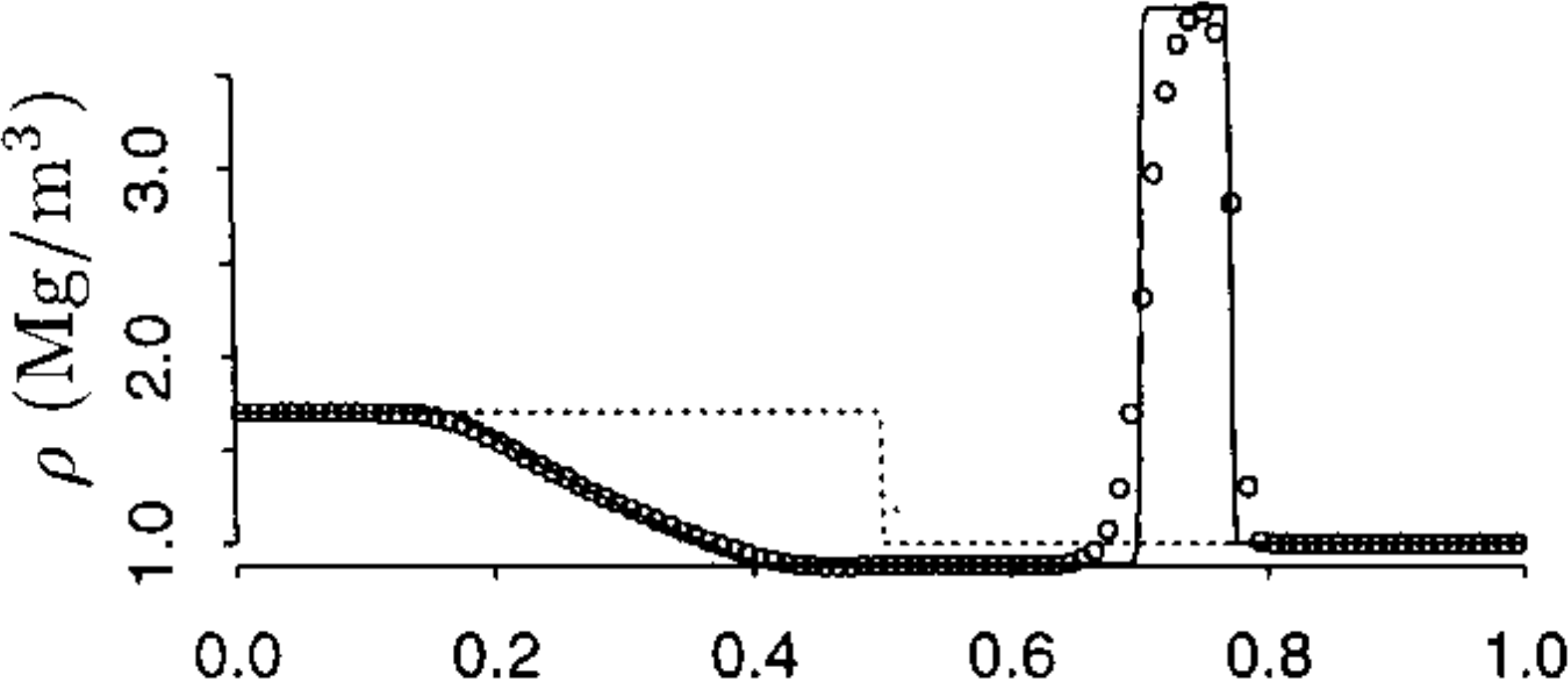}}
\quad
\subfloat[Pressure]
{\includegraphics[width=0.255\textwidth]{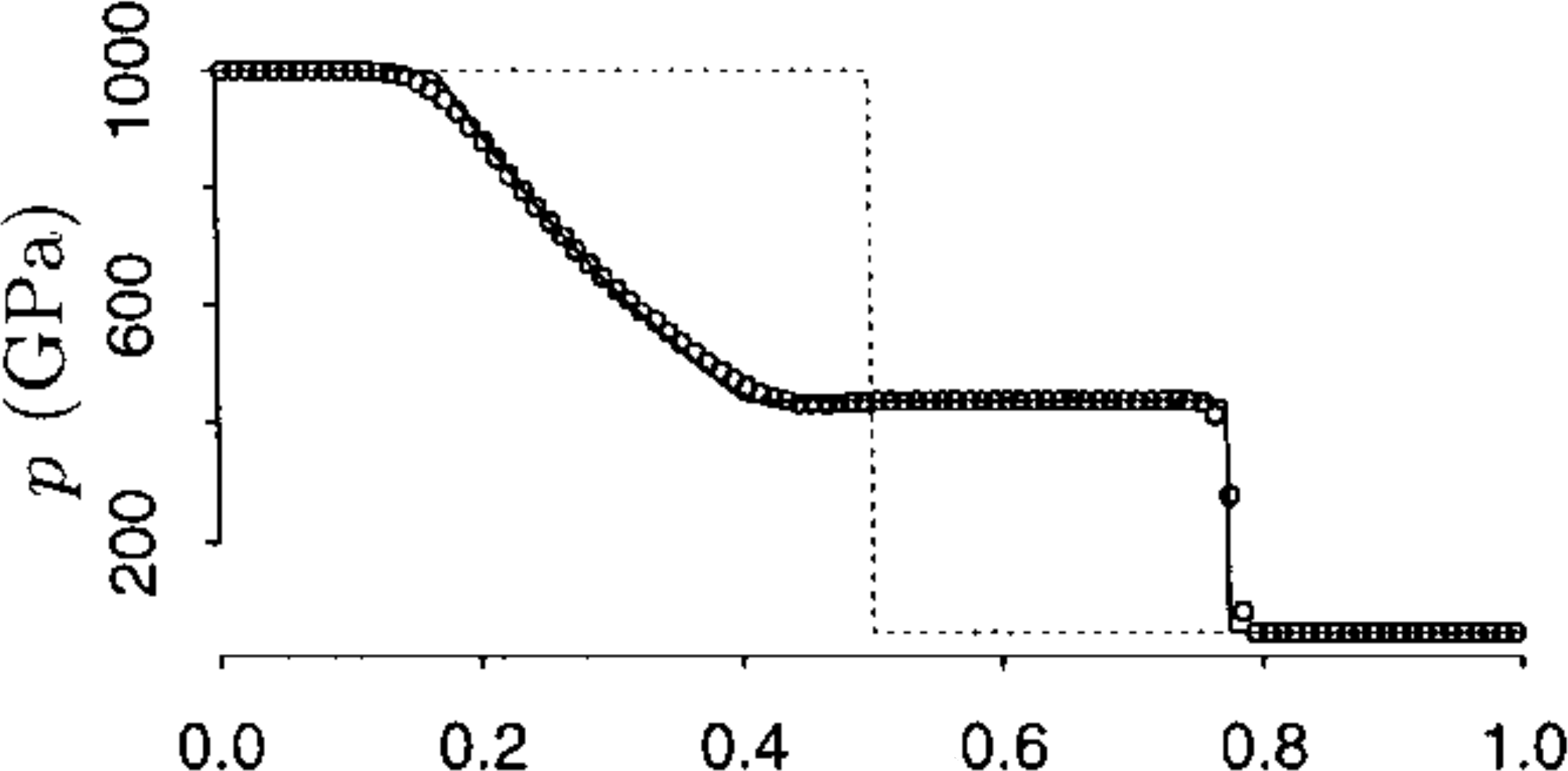}}
\quad
\subfloat[Velocity]
{\includegraphics[width=0.28\textwidth]{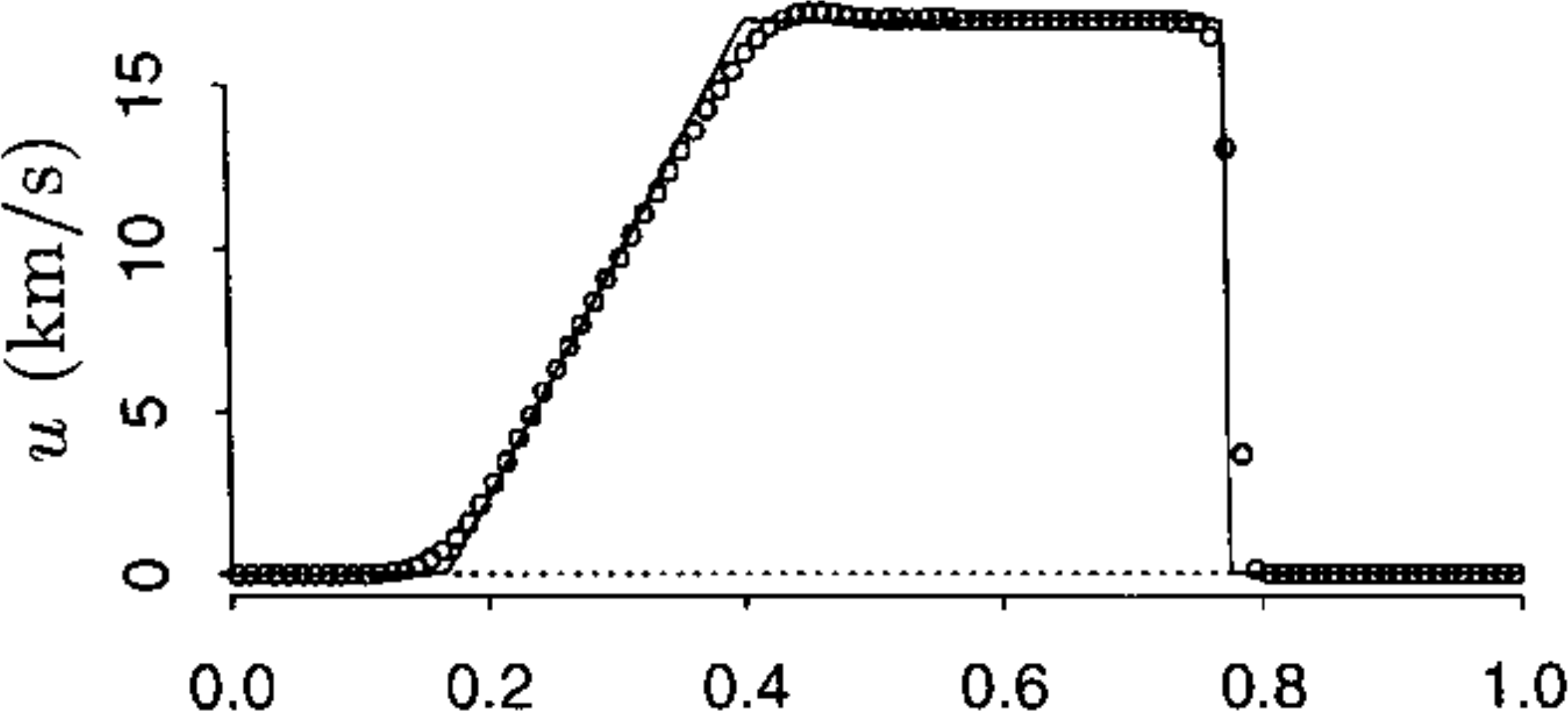}}
\caption{Shyue shock tube problem (top panel: our results, bottom panel:  
            results from Shyue \cite{Shyue2001}).}
\label{res:rm_shyue}
\end{figure}

\subsubsection{Saurel shock tube problem}

In this problem, we consider the advection of a density discontinuity of the
liquid nitromethane described by Cochran-Chan EOS \eqref{eq:CC} in a uniform
flow \cite{Saurel2007,Lee2013}. The parameters are given by
$A_1=\unit[8.192\times10^8]{Pa}$, $A_2=\unit[1.508\times10^9]{Pa},
\omega=1.19$, $R_1=4.53$, $R_2=1.42$ and $\rho_0=\unit[1134]{kg/m^3}$.  The
initial value is
\[
 [\rho, u, p]^\top = \left\{
    \begin{array}{ll}
      [1134, ~1000, ~2.0\times 10^{10}]^\top, &x<0.5,\\ [2mm]
      [500, ~1000, ~2.0\times 10^{10}]^\top, & x>0.5.
    \end{array}
  \right.
\]
The simulation terminates at $t=4.0\times 10^{-5}$. We use this Riemann problem
to  assess the performance of our methods against highly nonlinear equations of
state. Fig. \ref{res:rm_saurel} displays the results of our numerical scheme and
that of Saurel \textit{et al.} \cite{Saurel2007}, where we can see that there is
no non-physical pressure and velocity across the contact discontinuity in our
numerical scheme.

\begin{figure}[htbp]
\centering
\subfloat
{\includegraphics[width=0.3\textwidth]{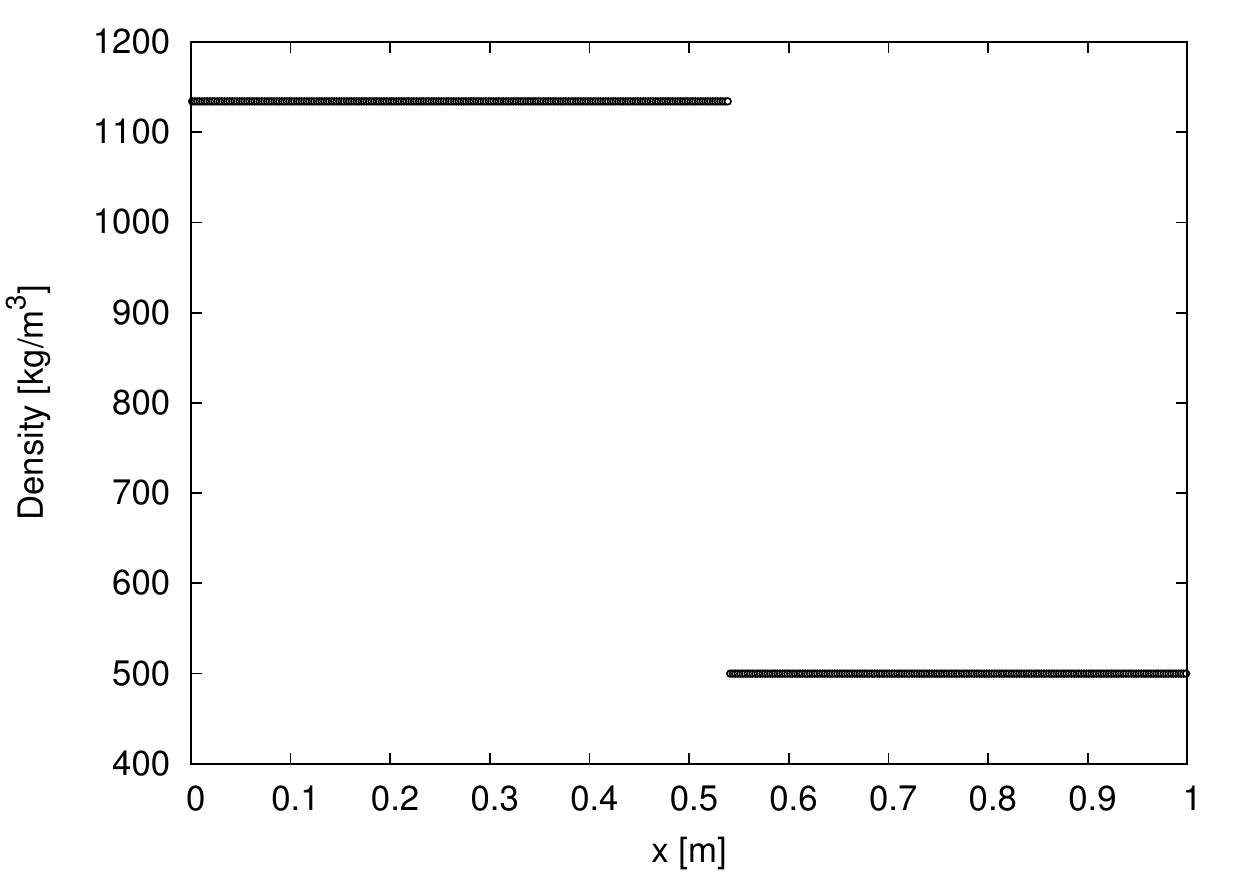}}
\subfloat
{\includegraphics[width=0.3\textwidth]{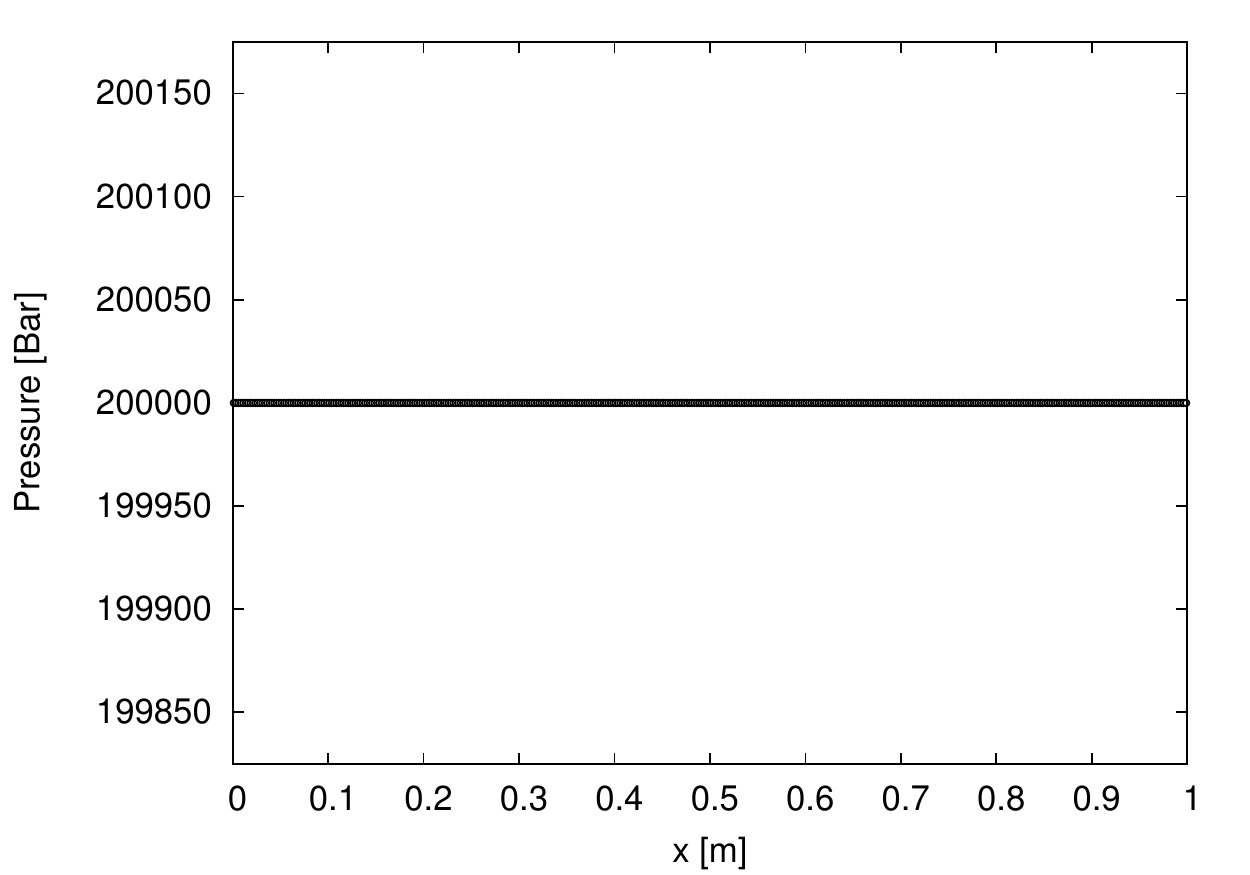}}
\subfloat
{\includegraphics[width=0.3\textwidth]{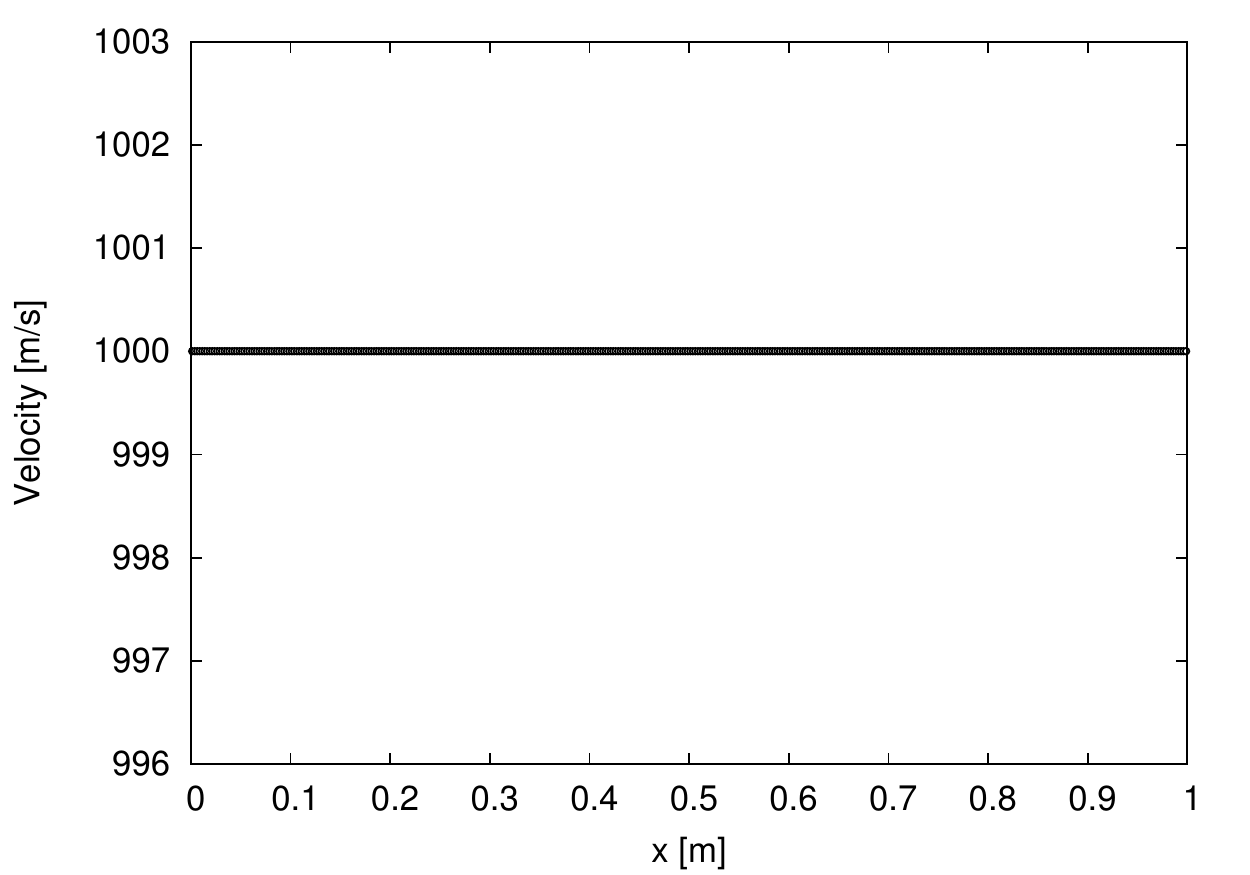}} \\
\addtocounter{subfigure}{-3}
\null\quad
\subfloat[Density]
{\includegraphics[width=0.26\textwidth]{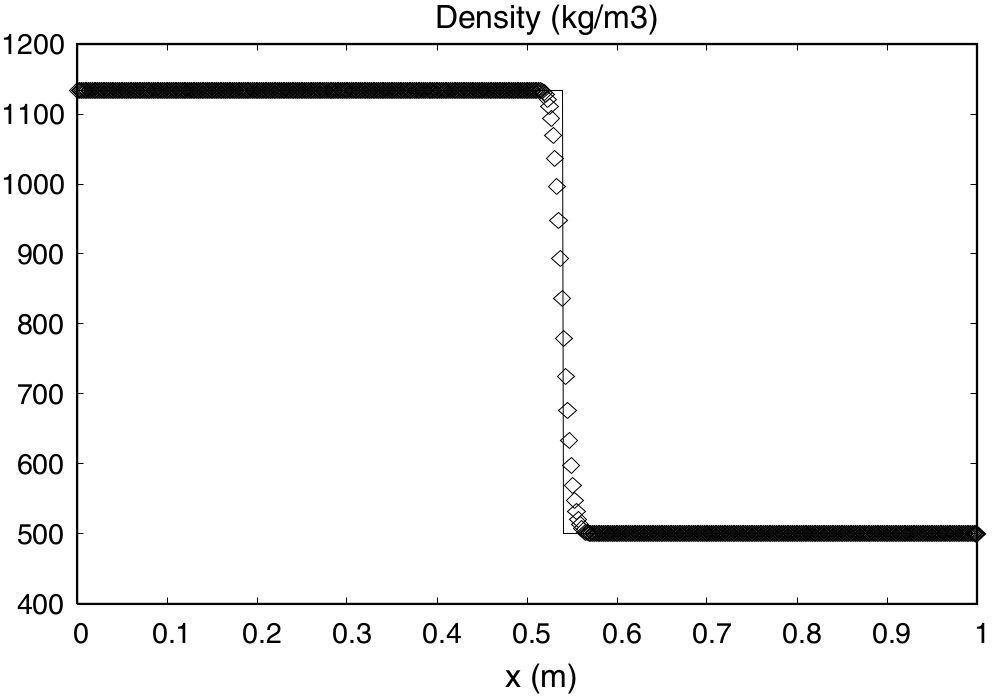}}
\quad
\subfloat[Pressure]
{\includegraphics[width=0.26\textwidth]{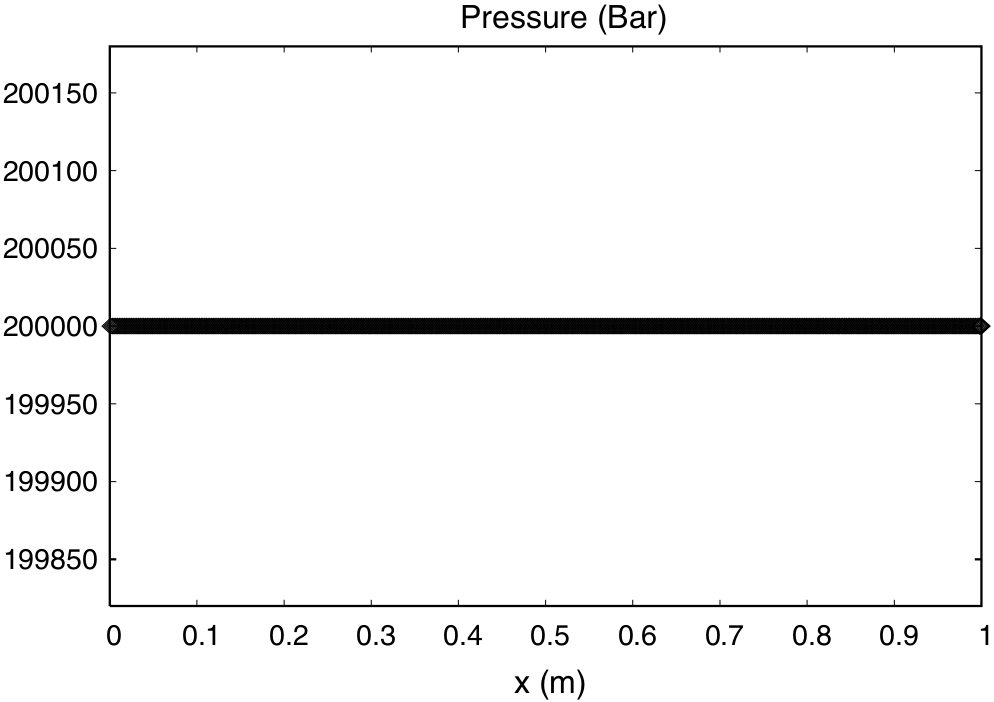}}
\quad 
\subfloat[Velocity]
{\includegraphics[width=0.26\textwidth]{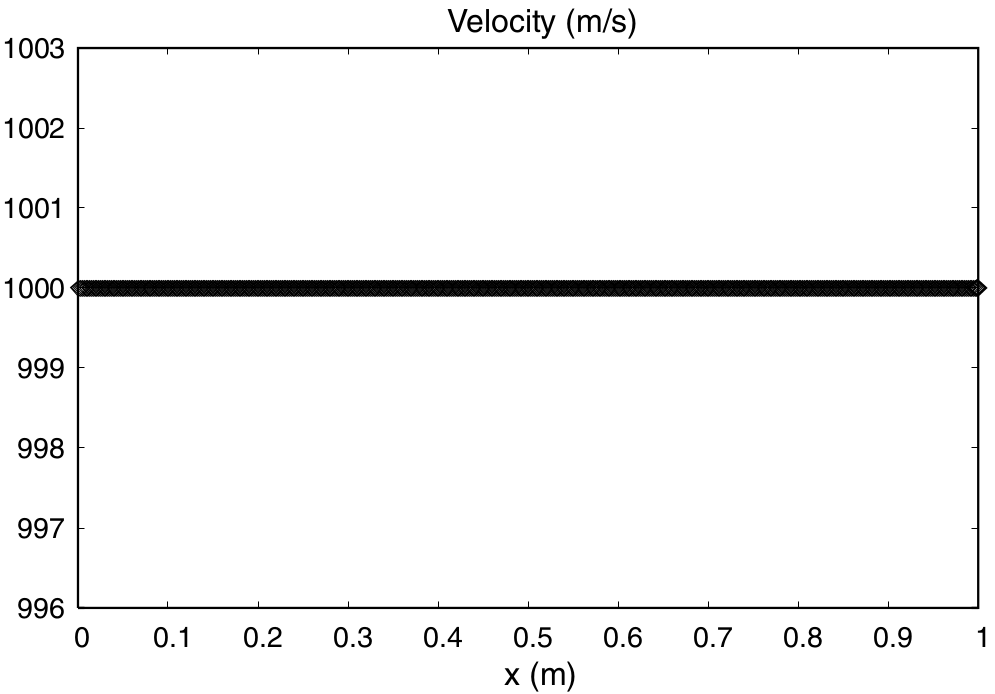}}
\caption{Saurel shock tube problem (top row: our results, bottom row: results
from Saurel \textit{et al.} \cite{Saurel2007}).} \label{res:rm_saurel}
\end{figure}

\subsubsection{Ideal gas-water Riemann problem}

In this example we simulate the ideal gas-water Riemann problems. The water is
modeled by either the stiffened gas EOS \eqref{eq:stiffened} or the polynomial
EOS \eqref{eq:poly}. The initial density, velocity and pressure are both
assigned with
\[
[\rho, u, p]^\top = \left\{
\begin{array}{ll}
[1630, ~0, ~7.0\times 10^9]^\top, &x<0.5,\\ [2mm]
[1000, ~0, ~1.0\times 10^5]^\top, & x > 0.5.
\end{array}
\right.
\]

The adiabatic exponent is $\gamma=2.0$ for the ideal gas EOS. The parameters of
water are $\gamma=7.15$ and $p_\infty=\unit[3.31\times10^8]{Pa}$ for the
stiffened gas EOS \cite{Rallu2009, Wang2008}, and $\rho_0=\unit[1000]{kg/m^3}$,
$A_1=\unit[2.20\times10^9]{Pa}$, $A_2=\unit[9.54\times10^9]{Pa}$,
$A_3=\unit[1.45\times10^{10}]{Pa}$, $B_0=B_1=0.28$, 
$T_1=\unit[2.20\times10^9]{Pa}$ and $T_2=0$ for the
polynomial EOS \cite{Jha2014}, respectively. The same parameters 
are chosen in the remaining numerical examples for the water.

The results with distinct equations of state at $t=8.0\times 10^{-5}$ are shown
in Fig. \ref{res:rm_gm_water}, where we can observe that the numerical results
agree well with the corresponding reference solutions. The comparison between
these two graphs also shows the discrepancies due to the choices of the
equations of state. It is observed that the shock wave in the stiffened gas EOS
propagates faster than that in the polynomial EOS.

\subsubsection{JWL-polynomial Riemann problem}

This example concerns the
JWL-polynomial Riemann problem. The initial states are 
\[
[\rho, u, p]^\top = \left\{
\begin{array}{ll}
[1630, ~0, ~8.3\times 10^9]^\top, &x<0.5,\\ [2mm]
[1000, ~0, ~1.0\times 10^5]^\top, & x > 0.5.
\end{array}
\right.
\]
We use the following values to describe the TNT \cite{Smith1999}:
$A_1=\unit[3.712\times10^{11}]{Pa}$,
$A_2=\unit[3.230\times10^9]{Pa}$,
$\omega=0.30$, 
$R_1=4.15$, 
$R_2=0.95$ and $\rho_0=\unit[1630]{kg/m^3}$.
The result at $t=8.0\times 10^{-5}$ is shown in Fig.
\ref{res:rm_jwl_pol}, where we can observe that both the interface and
shock are captured well without spurious oscillation. 

\begin{figure}[htbp]
	\centering
	\subfloat
	{\includegraphics[width=0.3\textwidth]{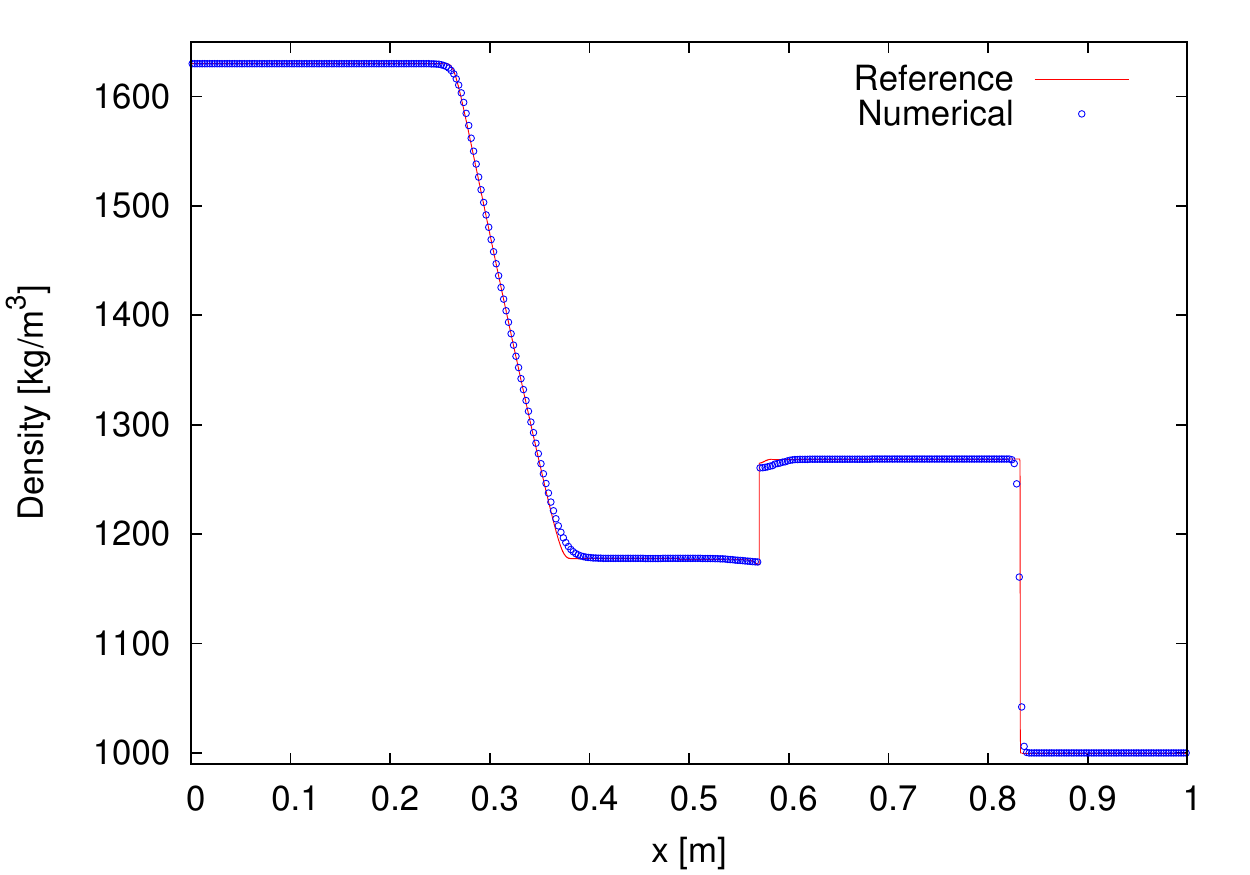}}
	\subfloat
	{\includegraphics[width=0.3\textwidth]{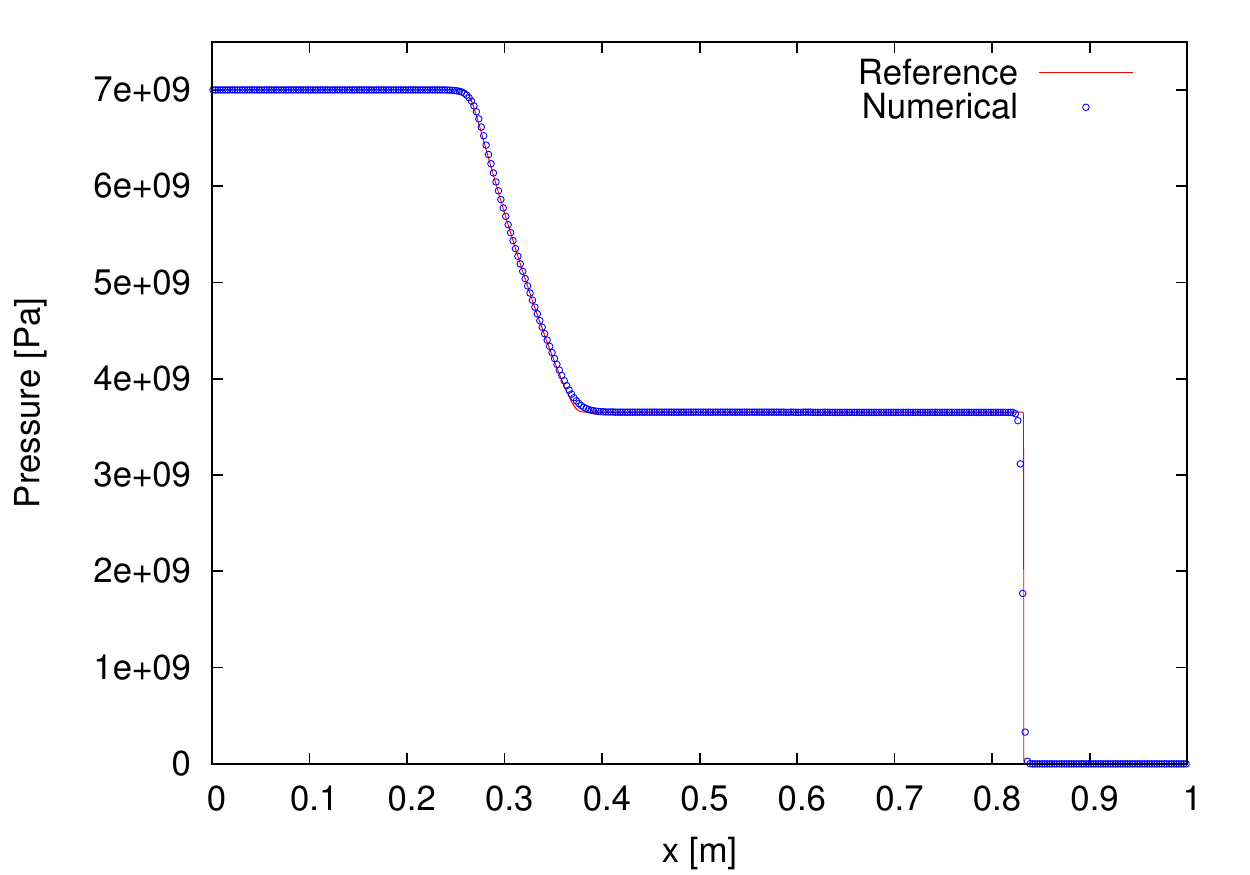}}
	\subfloat
	{\includegraphics[width=0.3\textwidth]{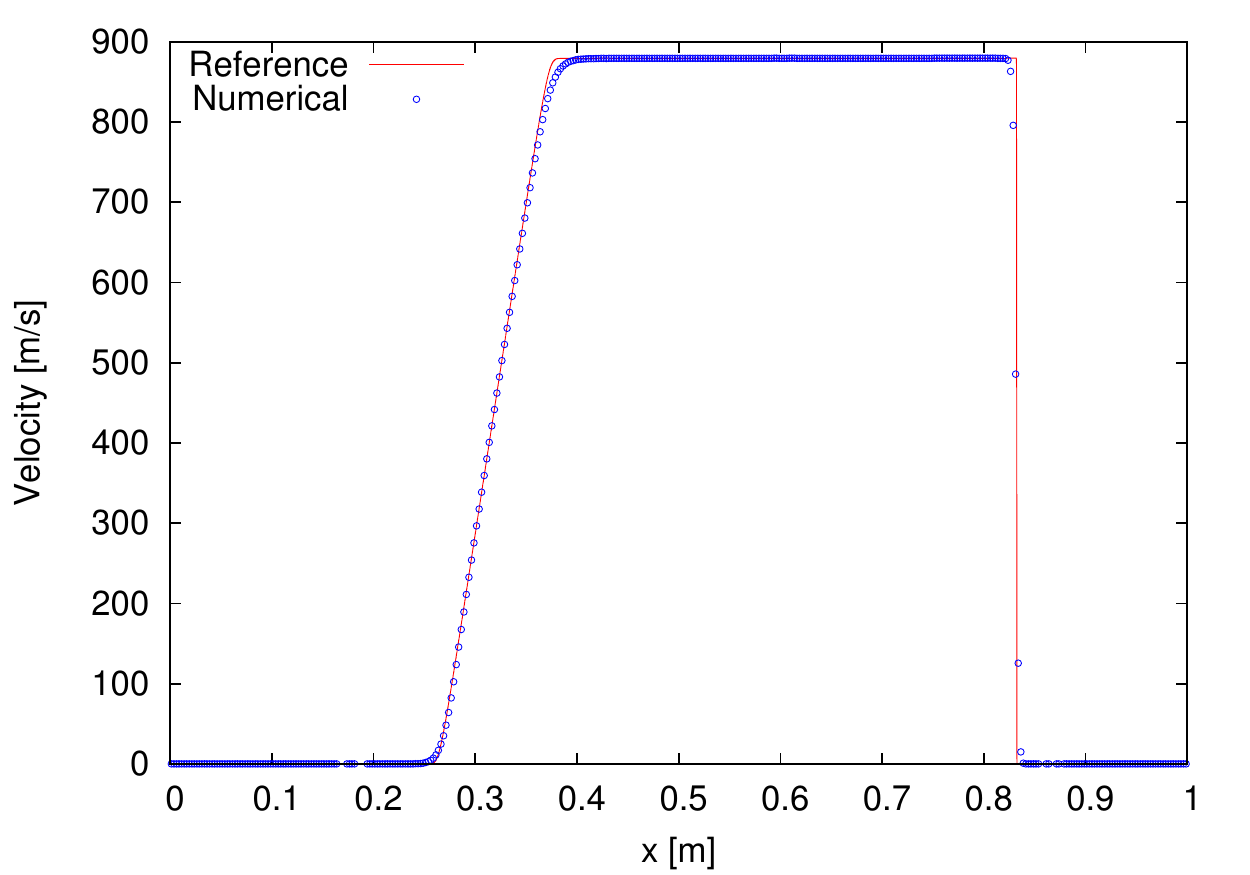}}
	\addtocounter{subfigure}{-3}
	\subfloat[Density]
	{\includegraphics[width=0.3\textwidth]{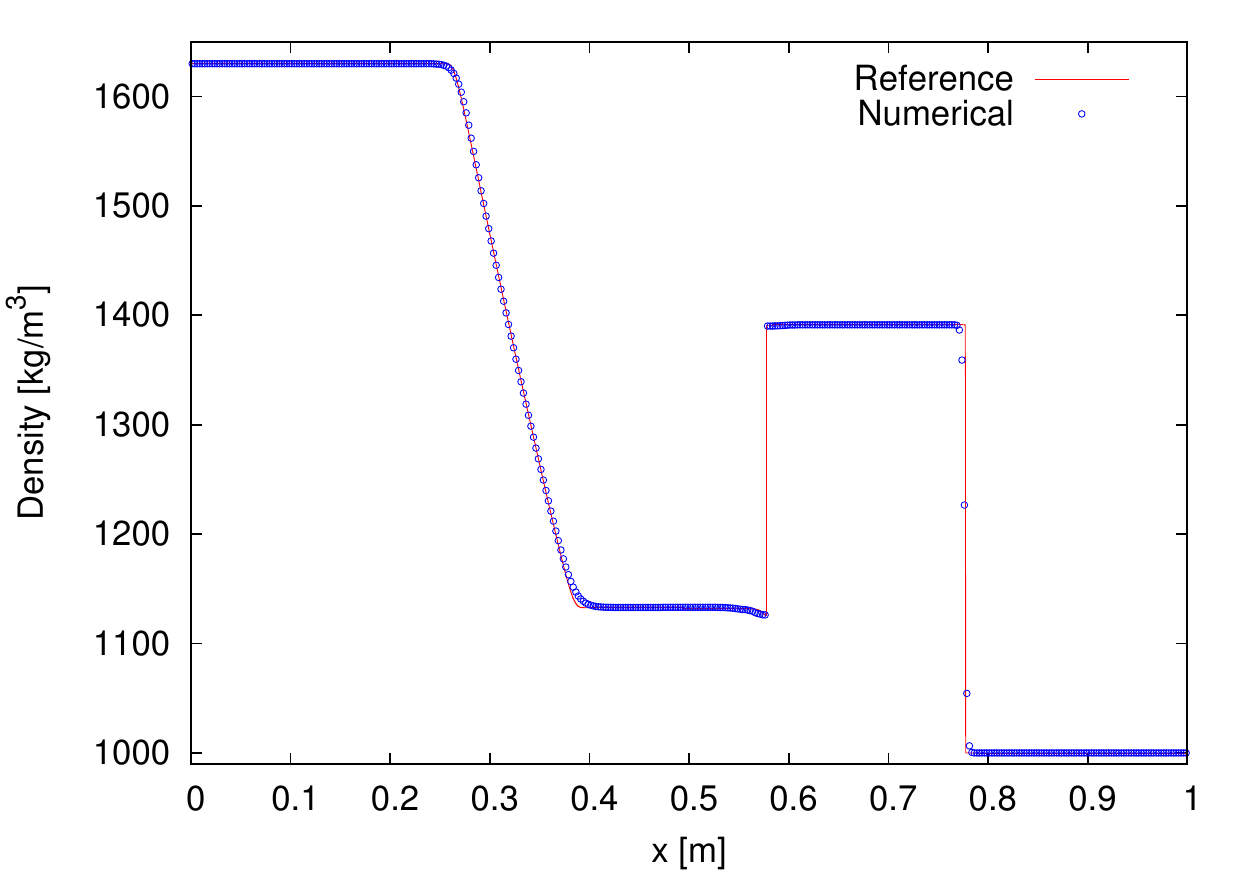}}
	\subfloat[Pressure]
	{\includegraphics[width=0.3\textwidth]{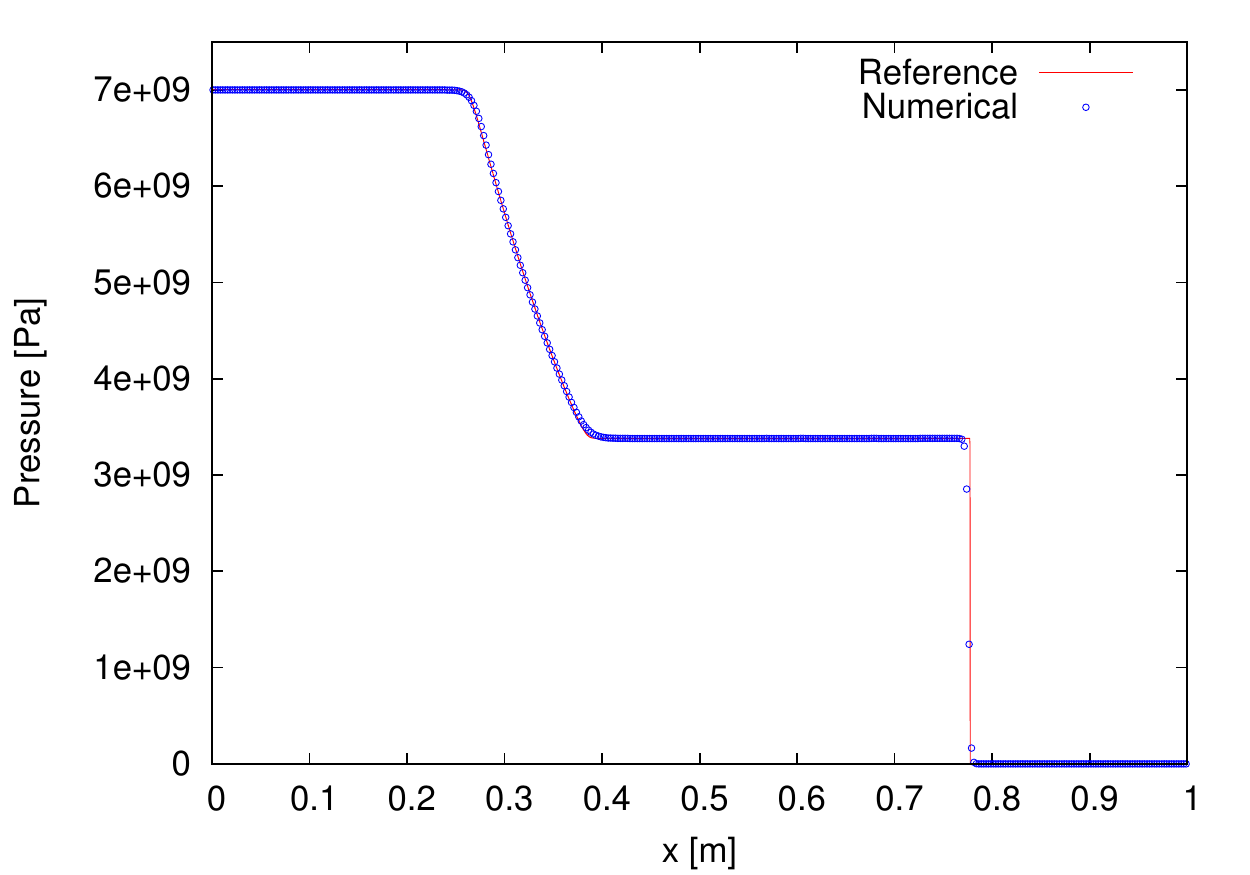}}
	\subfloat[Velocity]
	{\includegraphics[width=0.3\textwidth]{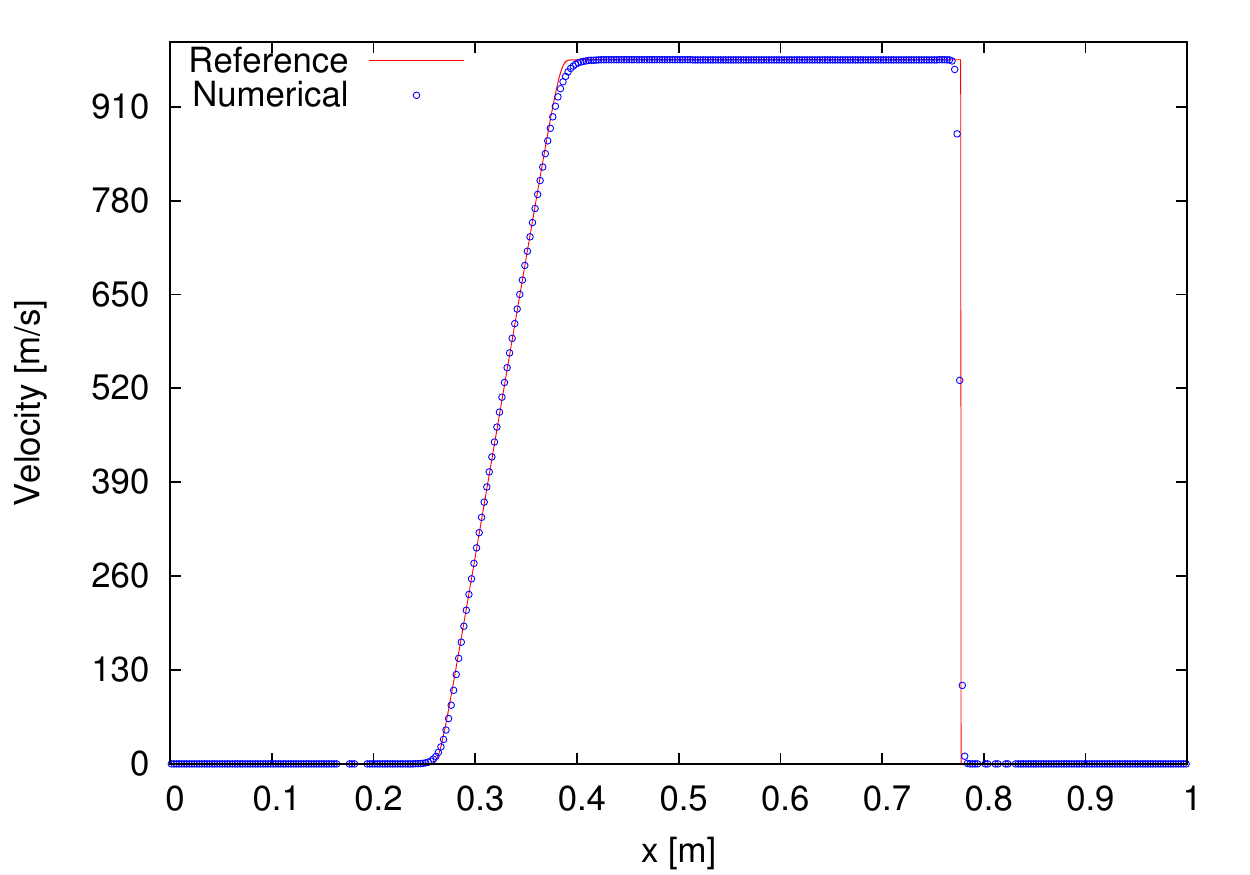}}
	\caption{Ideal gas-water Riemann problem (top row: stiffened
gas, bottom row: polynomial).}
	\label{res:rm_gm_water}
	\subfloat[Density]
	{\includegraphics[width=0.3\textwidth]{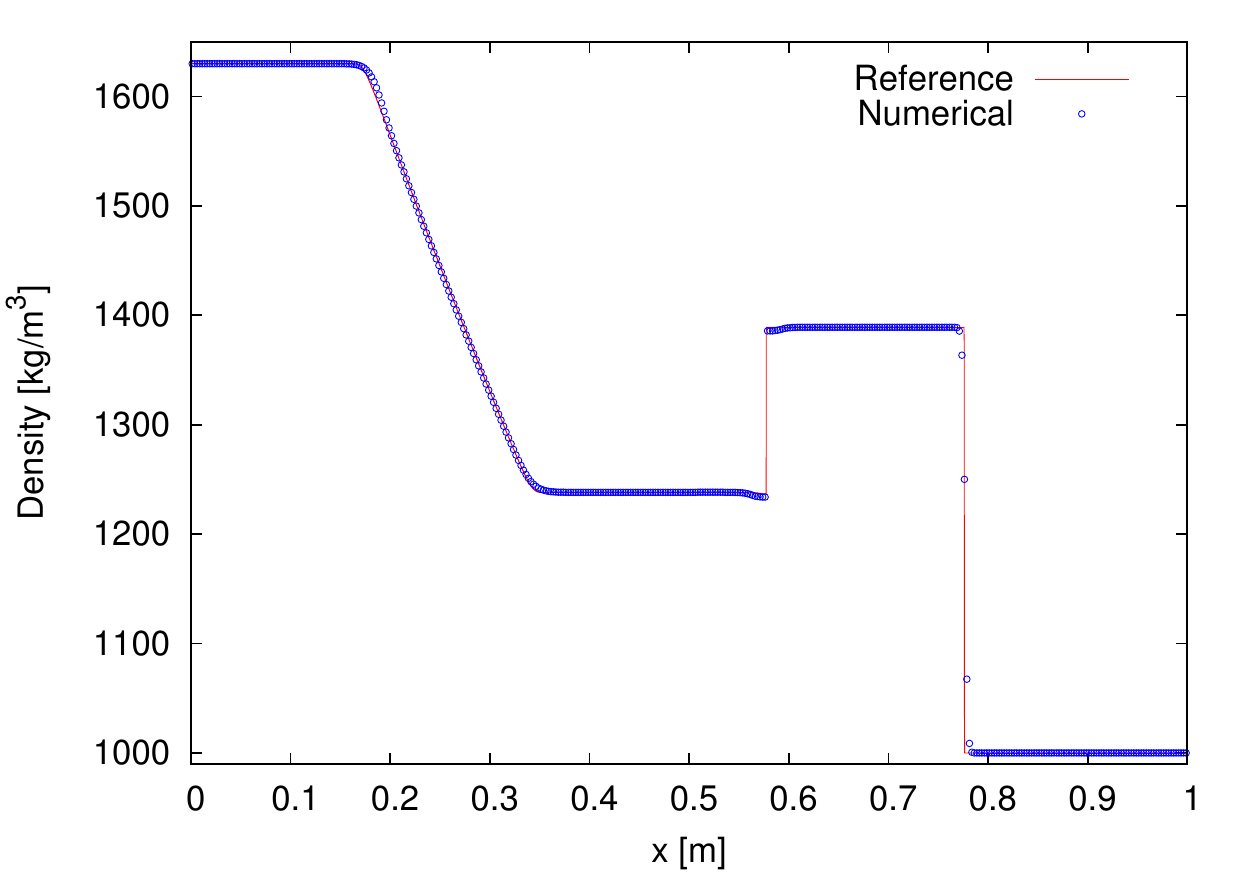}}
	\subfloat[Pressure]
	{\includegraphics[width=0.3\textwidth]{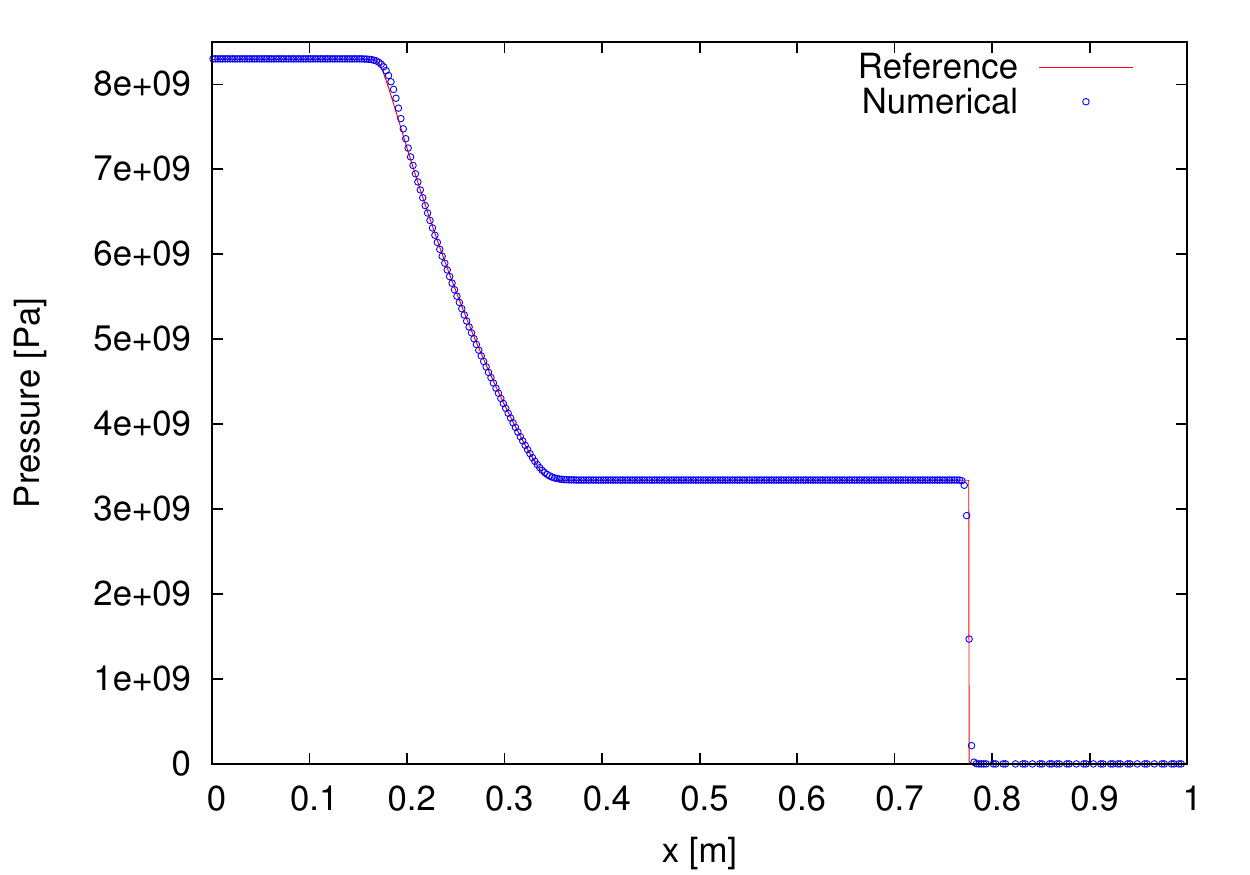}}
	\subfloat[Velocity]
	{\includegraphics[width=0.3\textwidth]{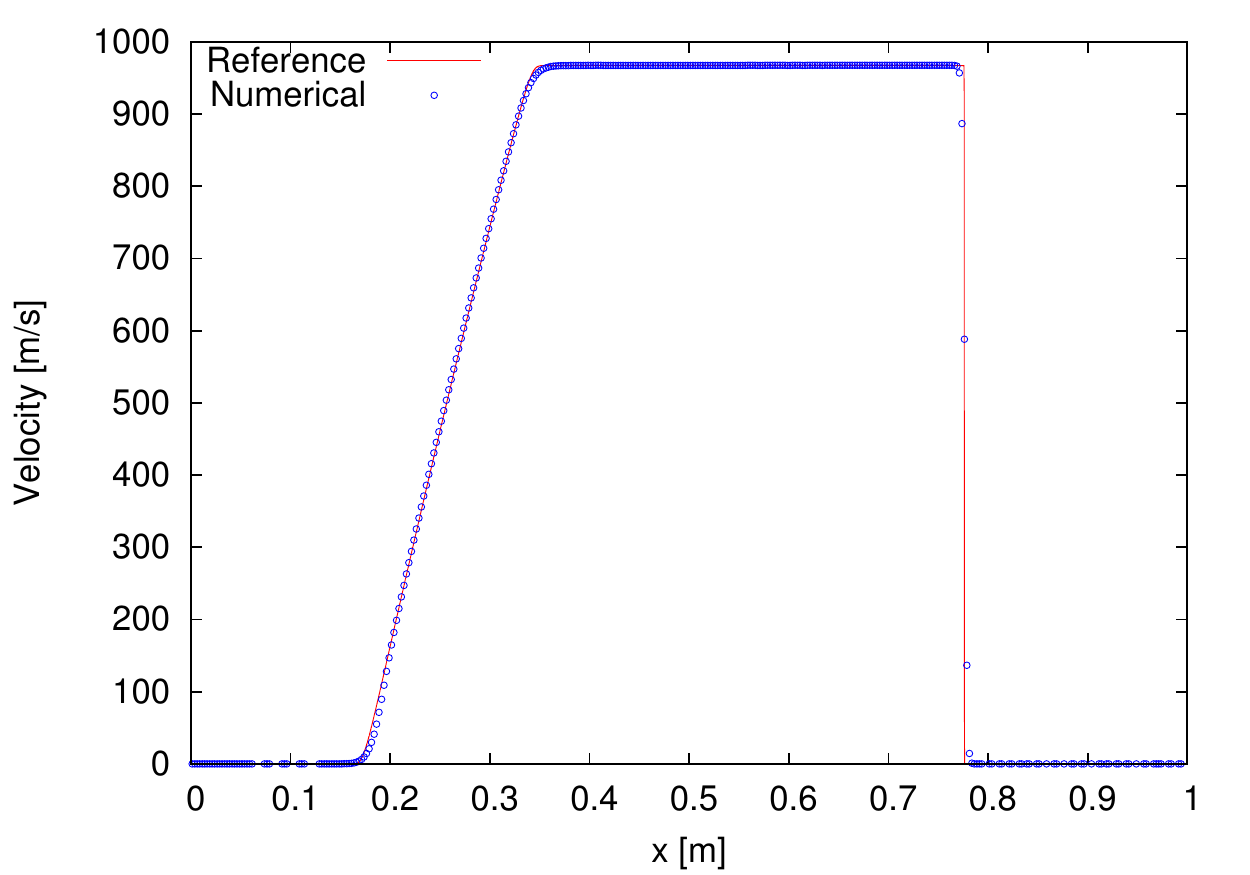}}
	\caption{JWL-polynomial Riemann
problem.}\label{res:rm_jwl_pol}
\end{figure}

\subsection{Spherically symmetric problems}

In this part we present two spherically symmetric problems, where the
governing equations are formulated as follows
\begin{equation}
\dfrac{\partial}{\partial t}
\begin{bmatrix}
r^2\rho \\ r^2\rho u \\ r^2E
\end{bmatrix}
+\dfrac{\partial}{\partial r}
\begin{bmatrix}
r^2\rho u \\ r^2(\rho u^2+p) \\ r^2(E+p)u
\end{bmatrix}
=\begin{bmatrix}
0 \\ 2rp \\ 0
\end{bmatrix}.
\label{eq:spheq}
\end{equation}
The source term in \eqref{eq:spheq} is discretized using an explicit Euler 
method. 

\subsubsection{Air blast problem}\label{problem:blast}

The shock wave that propagates through the air as a consequence of the nuclear
explosion is commonly referred to as the \emph{blast wave}. In this example we
simulate the blast wave from one kiloton nuclear charge. The explosion products
and air are modeled by the ideal gas EOS with adiabatic exponents $\gamma=1.2$
and $1.4$ respectively. The initial density and pressure are
$\unit[618.935]{kg/m^3}$ and $\unit[6.314 \times 10^{12}]{Pa}$ for the explosion
products, and $\unit[1.29]{kg/m^3}$ and $\unit[1.013\times 10^5]{Pa}$ for the
air.  The initial interface is located at $r=\unit[0.3]{m}$ initially. To
effectively capture the wave propagation we use a computational domain of radius
$\unit[5000]{m}$. 

It is known that the destructive effects of the blast wave can be measured by
its \emph{overpressure}, i.e., the amount by which the static pressure in the
blast wave exceeds the ambient pressure ($\unit[1.013\times 10^5]{Pa}$). The
overpressure increases rapidly to a peak value when the blast wave arrives,
followed by a roughly exponential decay. The integration of the overpressure
over time is called \emph{impulse}. See Fig.  \ref{rm::air_blast} (a) for an
illustration of these terminologies.  Fig. \ref{rm::air_blast} (b) -- (d) show 
the peak overpressure, impulse and shock arrival time at different radii. The
results are compared with the point explosion solutions in Qiao \cite{Qiao2003},
which confirm the accuracy of our methods in the air blast applications.

\subsubsection{Underwater explosion problem}

We use this example to simulate the underwater explosion problem where a TNT of
one hundred kilograms explodes in the water. The high explosives and water are
characterized by the JWL EOS and polynomial EOS, respectively. The radii of the
computational domain and the initial interface are $\unit[15]{m}$ and
$\unit[0.245]{m}$ respectively.  The initial pressure of the high explosives is
$\unit[9.5\times 10^9]{Pa}$.  The same problem has been simulated in
\cite{Jia2007} using ANSYS/AUTODYN.  Fig. \ref{rm::udex} shows the computed peak
overpressure and impulse at different radii. The results agree well with the
empirical law provided in \cite{Cole1948}.

\begin{figure}
\centering
\subfloat[Typical pressure-time relation]
  {\includegraphics[width=0.5\textwidth]{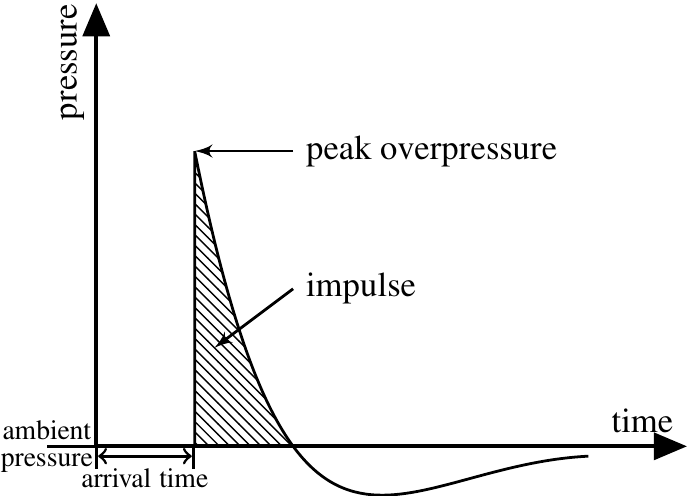}}
\subfloat[Peak overpressure]
  {\includegraphics[width=0.5\textwidth]{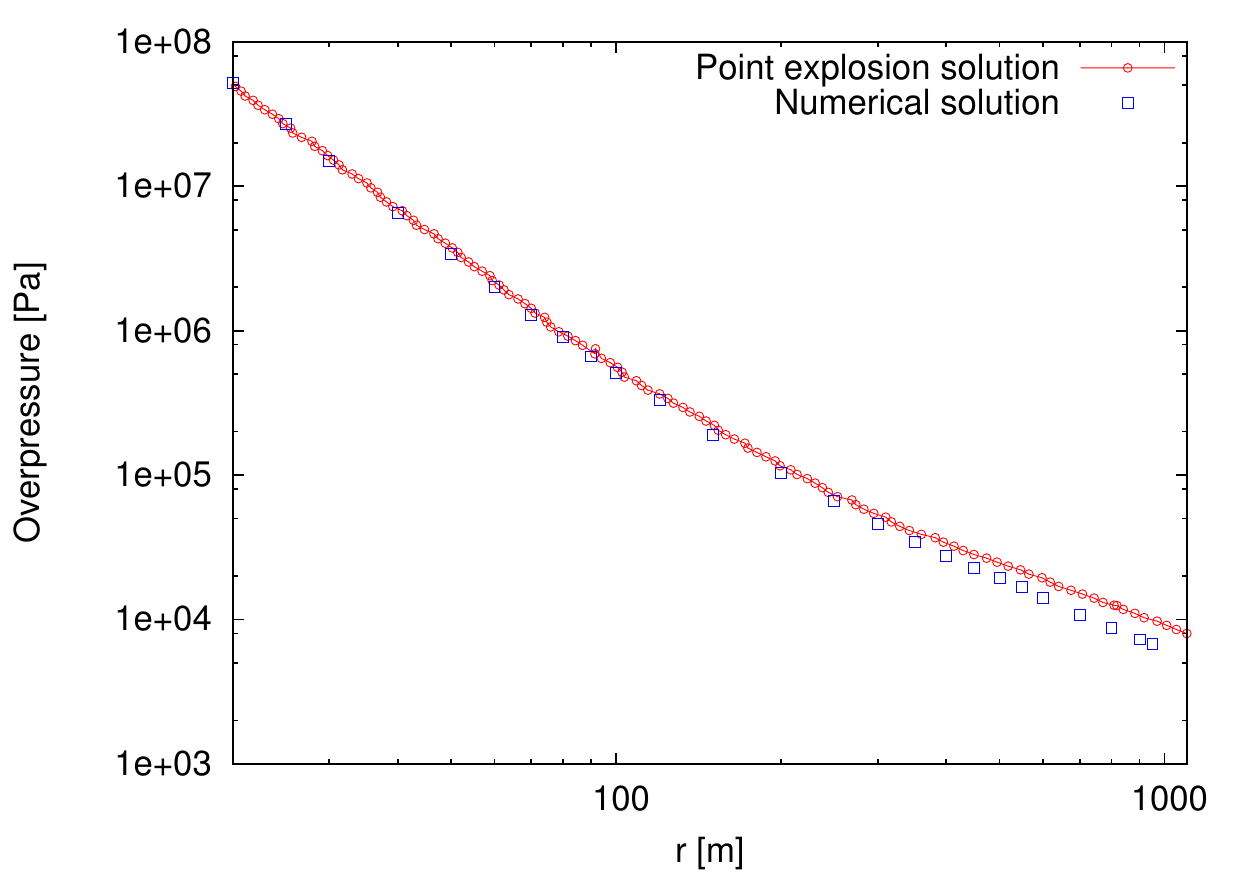}} \\
\subfloat[Impulse]
  {\includegraphics[width=0.5\textwidth]{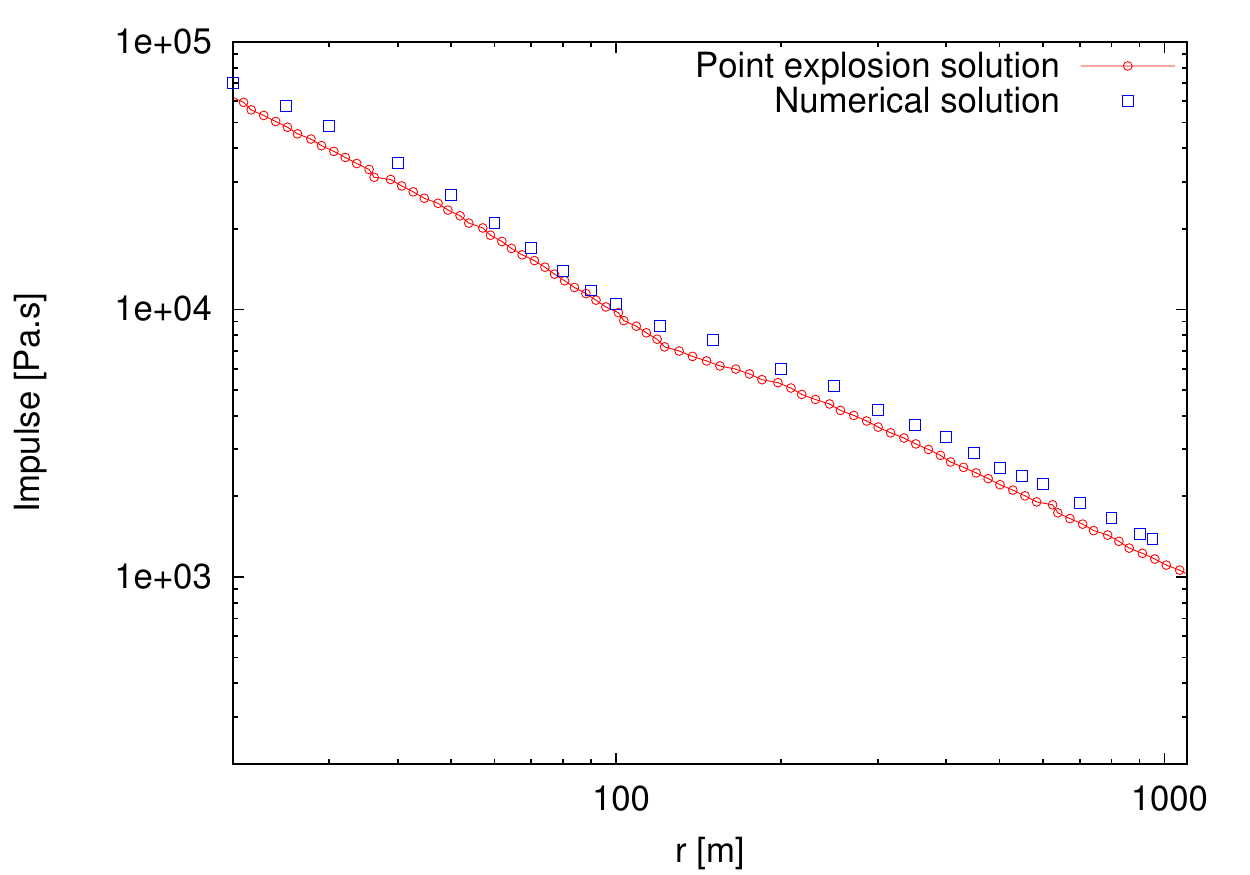}}
\subfloat[Arrival time]
  {\includegraphics[width=0.5\textwidth]{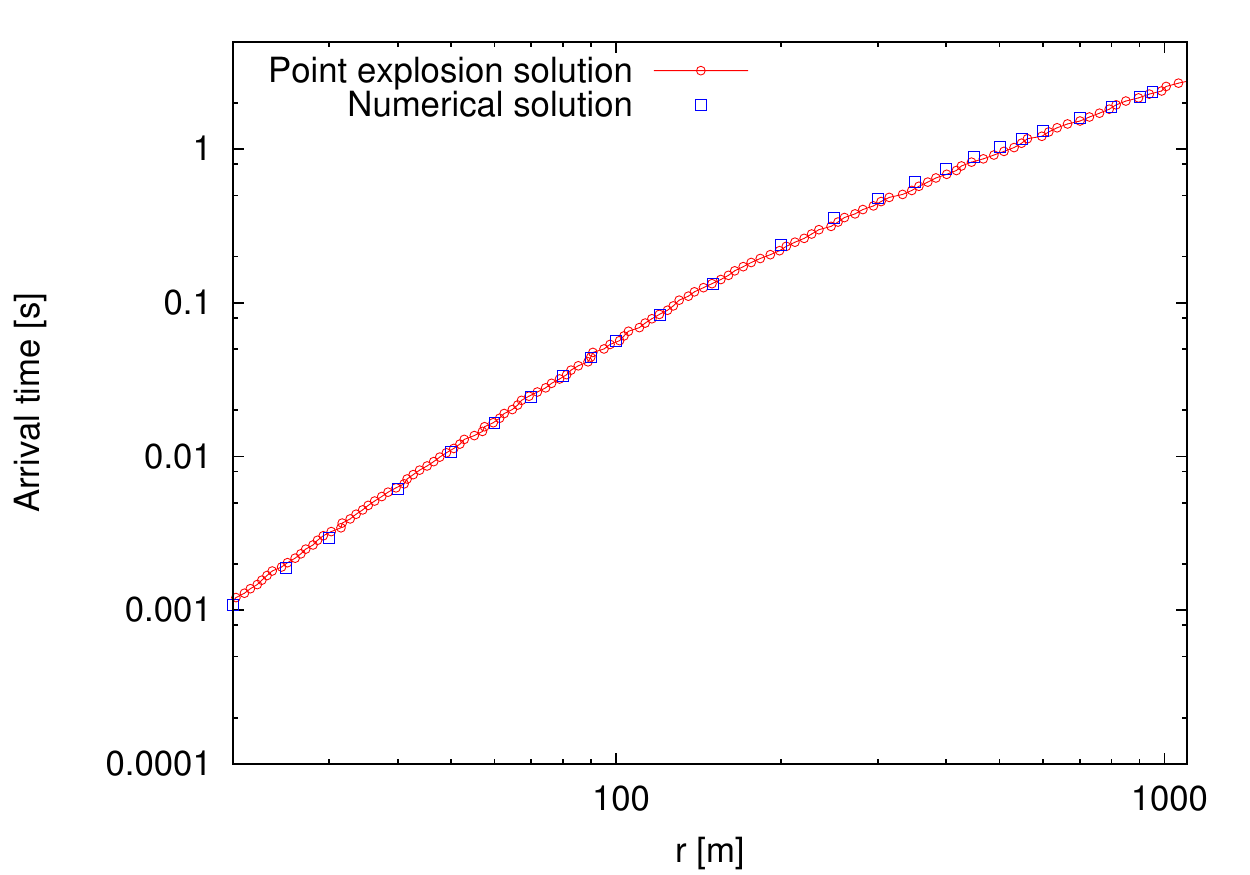}}
\caption{Shock wave parameters for air blast problem.}
\label{rm::air_blast}
\end{figure}

\begin{figure}[htbp]
\centering
\subfloat[Peak overpressure]
{\includegraphics[width=0.5\textwidth]{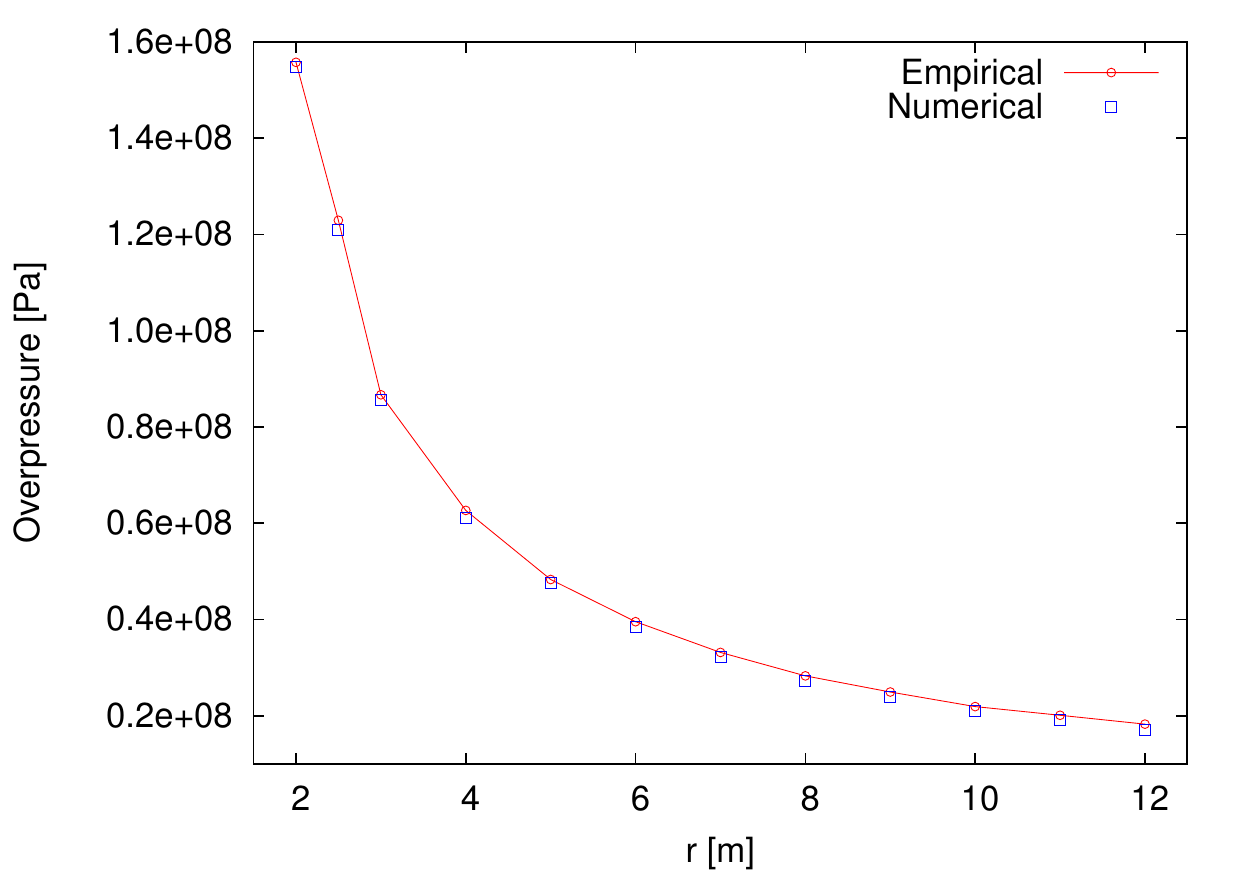}}
\subfloat[Impulse]
{\includegraphics[width=0.5\textwidth]{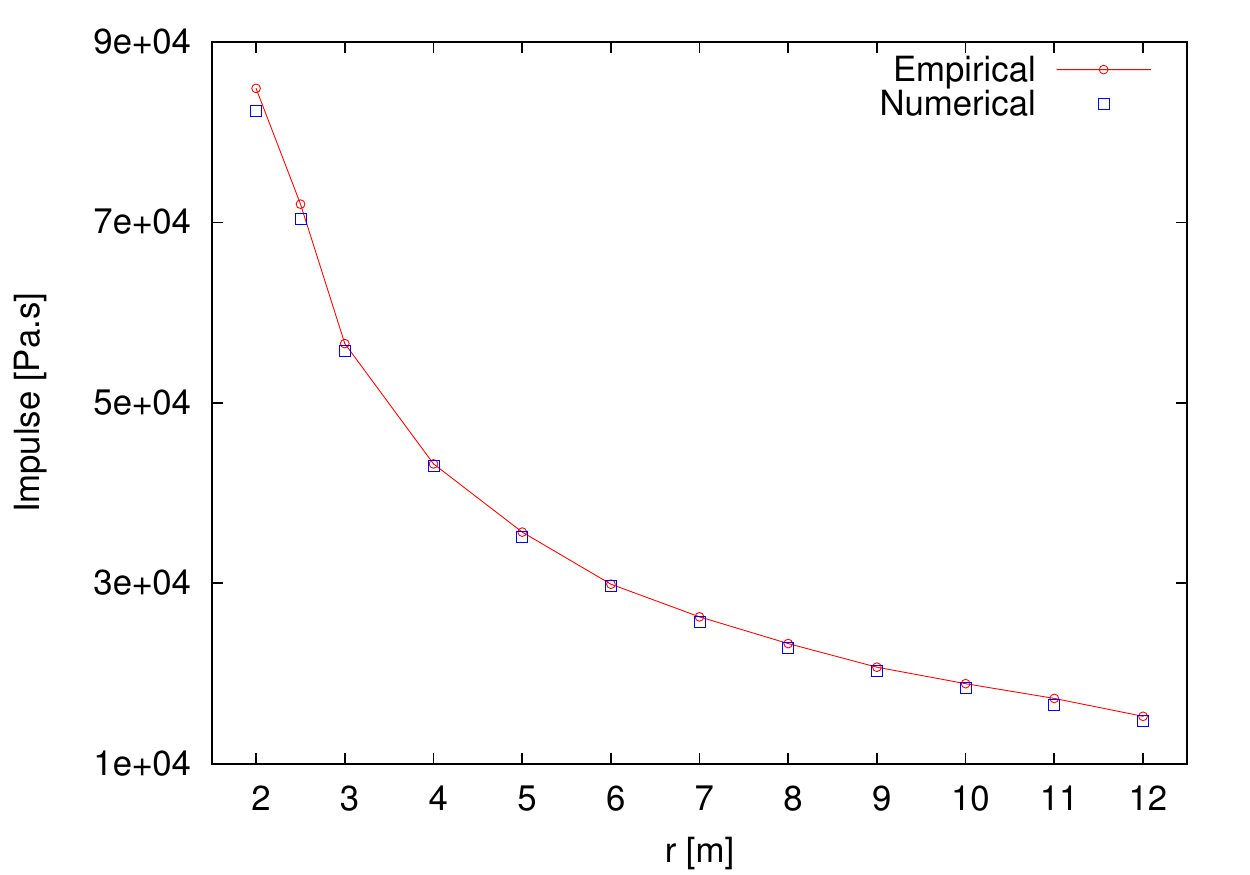}}
\caption{Shock wave parameters for underwater explosion problem.}
\label{rm::udex}
\end{figure}

\subsection{Two-dimensional problems}

In this part, we present a few two-dimensional cylindrically symmetric flows in
engineering applications. The Euler equations for this configuration are
formulated as a cylindrical form
\[
\dfrac{\partial}{\partial t}
\begin{bmatrix}
r\rho \\ r\rho u \\ r\rho v\\ rE
\end{bmatrix}
+\dfrac{\partial}{\partial r}
\begin{bmatrix}
r\rho u \\ r(\rho u^2+p) \\ r\rho uv \\ r(E+p)u
\end{bmatrix}
+\dfrac{\partial}{\partial z}
\begin{bmatrix}
r\rho v \\ r\rho uv \\ r(\rho v^2+p)\\ r(E+p)v
\end{bmatrix}
=\begin{bmatrix}
0 \\ p \\ 0 \\ 0
\end{bmatrix}.
\]
To improve the efficiency of the simulation, the $h$-adaptive mesh method is
adopted here \cite{Li2013}. Roughly speaking, more elements will be distributed
in the region where the jump of pressure is sufficiently large.

\subsubsection{Nuclear air blast problem}

In this example we simulate the nuclear air blast in the computational domain
$0\le r,z\le \unit[2000]{m}$. The initial states of the explosion products and
air in this example are the same as that in Section \ref{problem:blast}, except
that the bottom edge $z=0$ is now a rigid ground. The explosive center is
located at the height $z=\unit[100]{m}$.  And the radius of the initial
interface is $\unit[0.3]{m}$ at $t=0$. Fig. \ref{rm::air_blast2} shows the
pressure contours and adaptive meshes at $t=\unit[0.09]{s}$ and $\unit[0.3]{s}$.
When the blast wave produced by the nuclear explosion arrives at the ground, it
will be reflected firstly and propagate along the rigid ground simultaneously.
When the incident angle exceeds the limit, the reflective wave switches from
regular to irregular, and a Mach blast wave occurs. The peak overpressure and
impulse at different radii are shown in Fig.  \ref{rm::air_blast_ref2}. Our
numerical results agree well the the reference data interpolated from the given
standard data in \cite{Glasstone1977}.

\begin{figure}[htbp]
\centering
\subfloat[Pressure contour at $t=0.09$s]
{\includegraphics[width=0.45\textwidth]{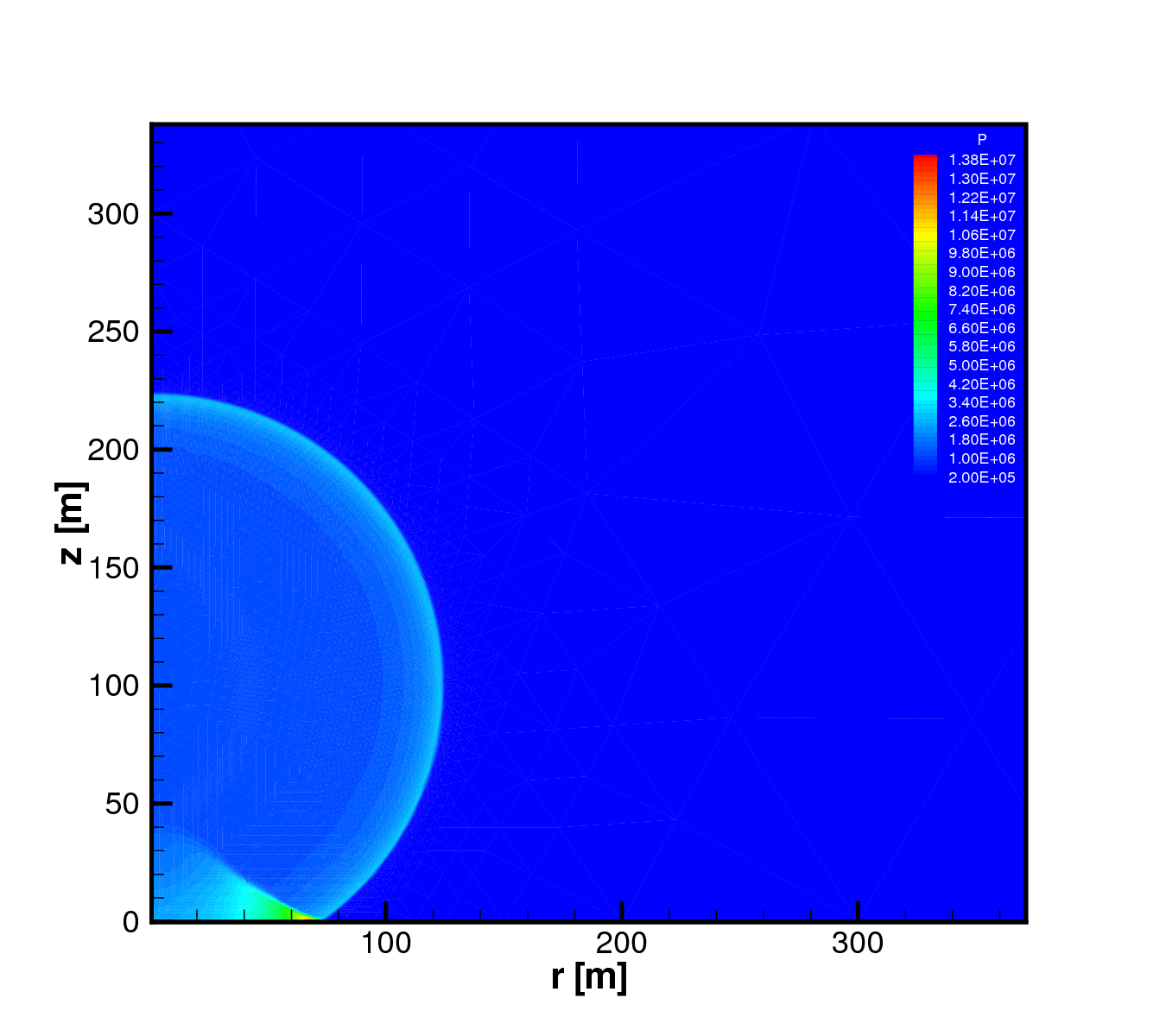}}
\subfloat[Adaptive mesh at $t=0.09$s]
{\includegraphics[width=0.45\textwidth]{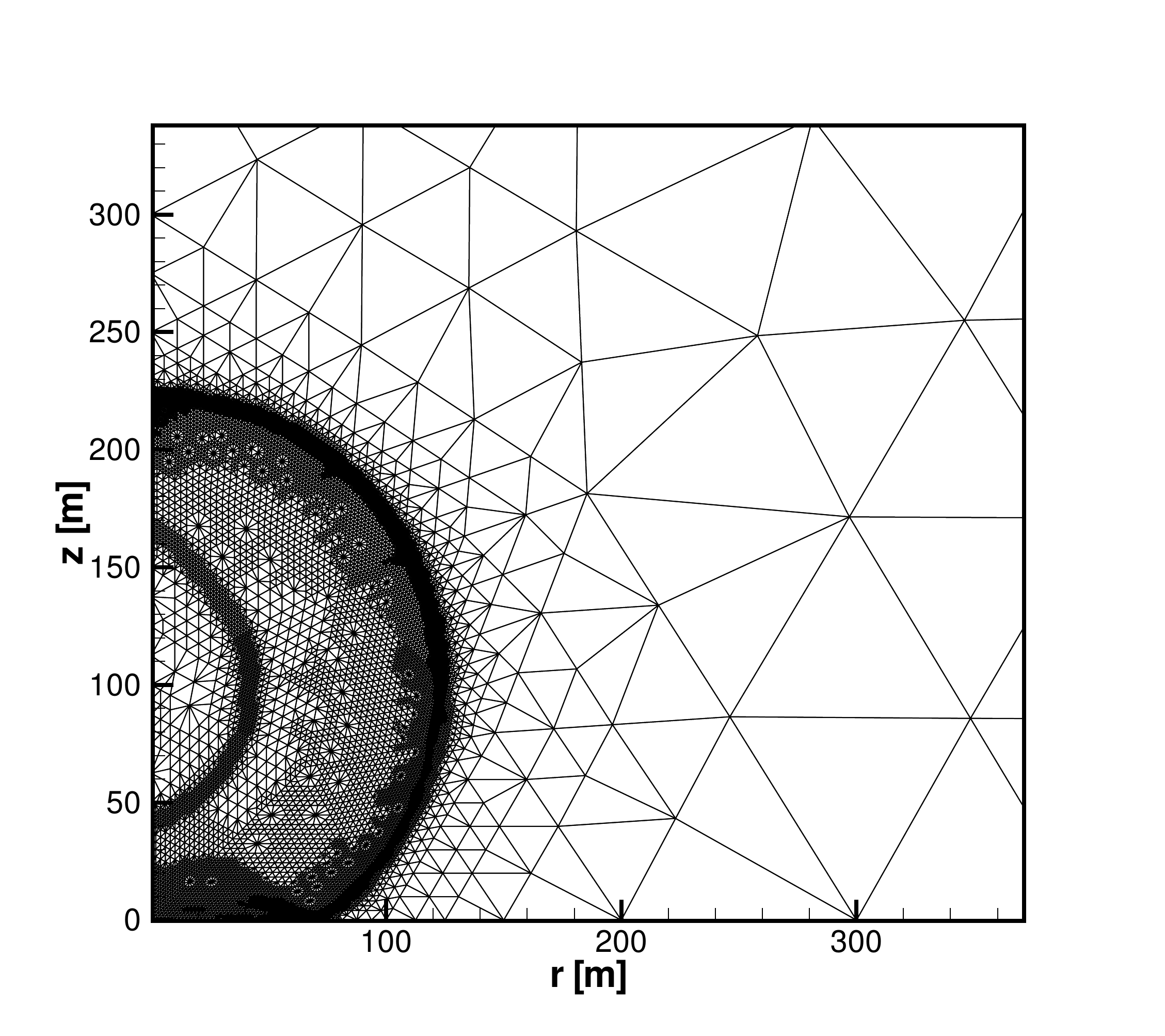}}\\
\subfloat[Pressure contour at $t=0.3$s]
{\includegraphics[width=0.45\textwidth]{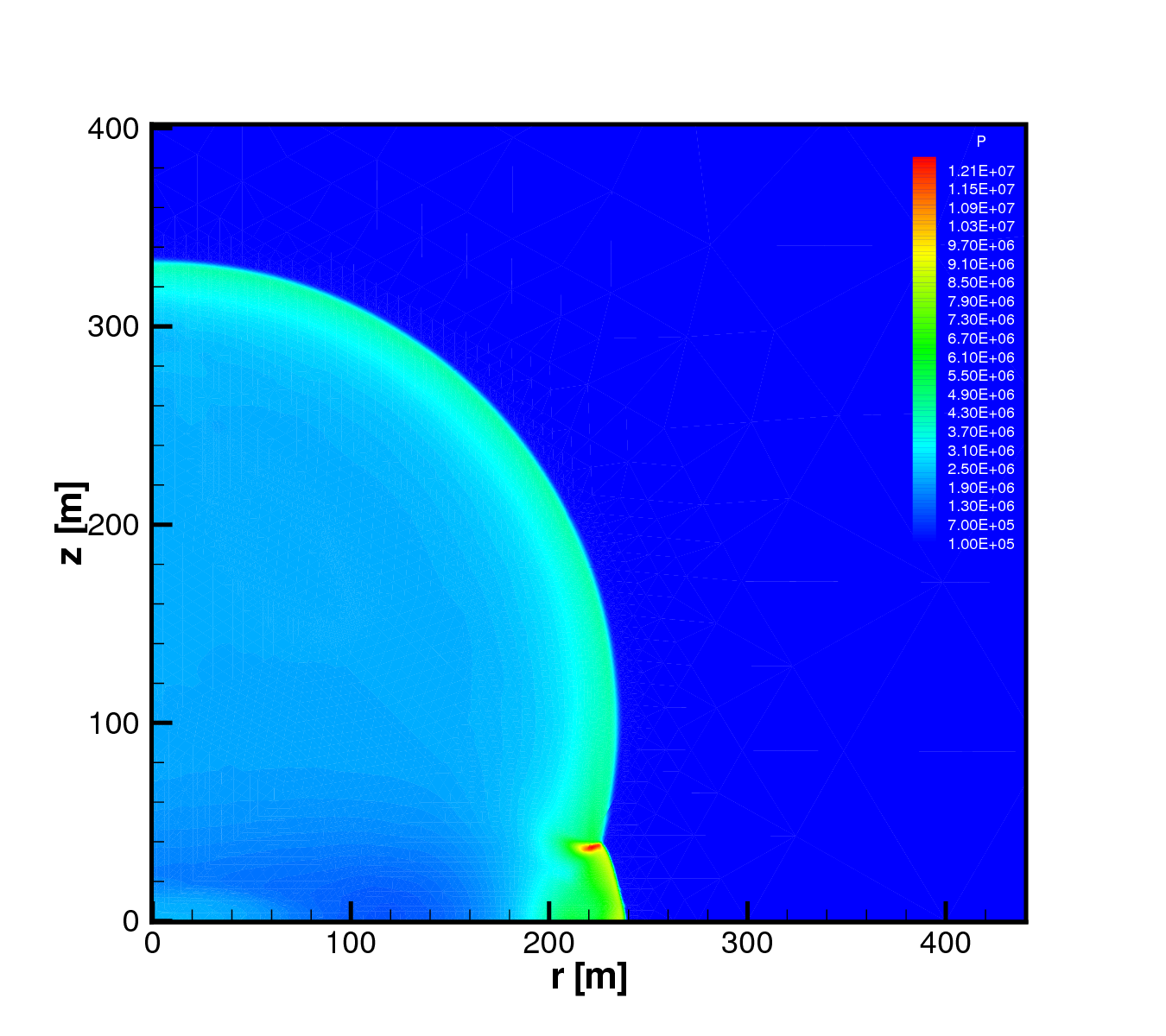}}
\subfloat[Adaptive mesh at $t=0.3$s]
{\includegraphics[width=0.45\textwidth]{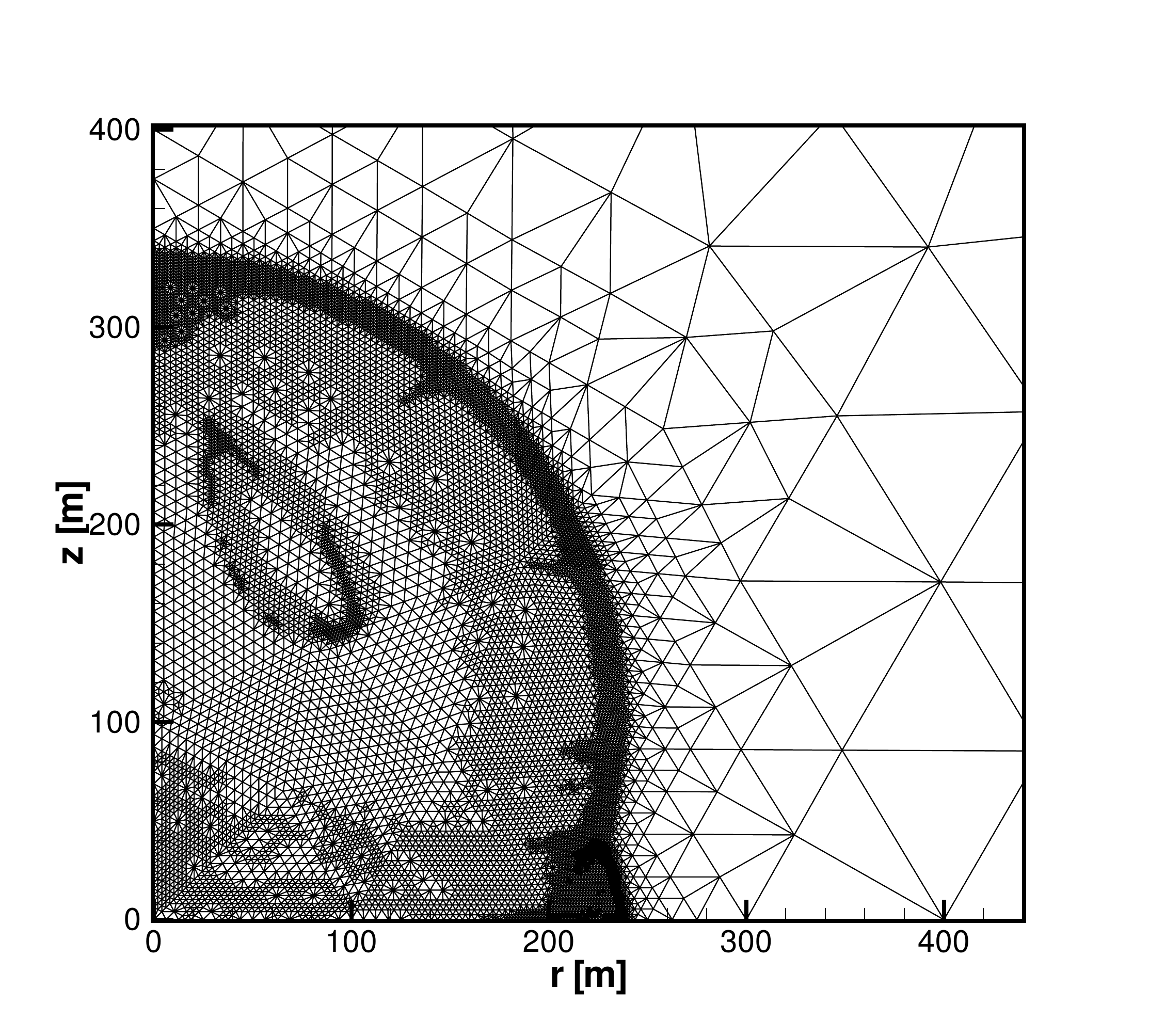}}
\caption{Pressure contours and adaptive meshes for nuclear 
	        air blast problem.}
\label{rm::air_blast2}
\subfloat[Peak overpressure]
{\includegraphics[width=0.45\textwidth]{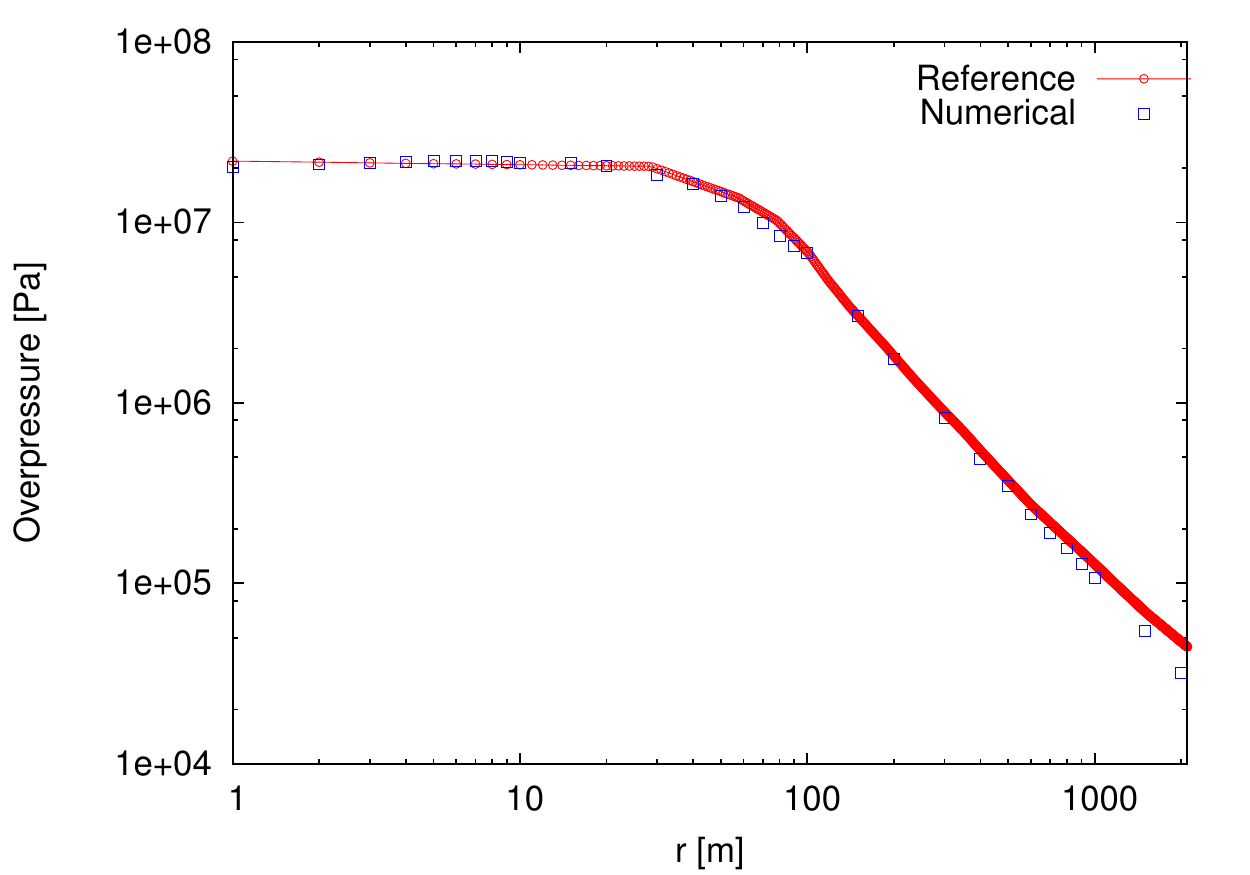}}
\subfloat[Impulse]
{\includegraphics[width=0.45\textwidth]{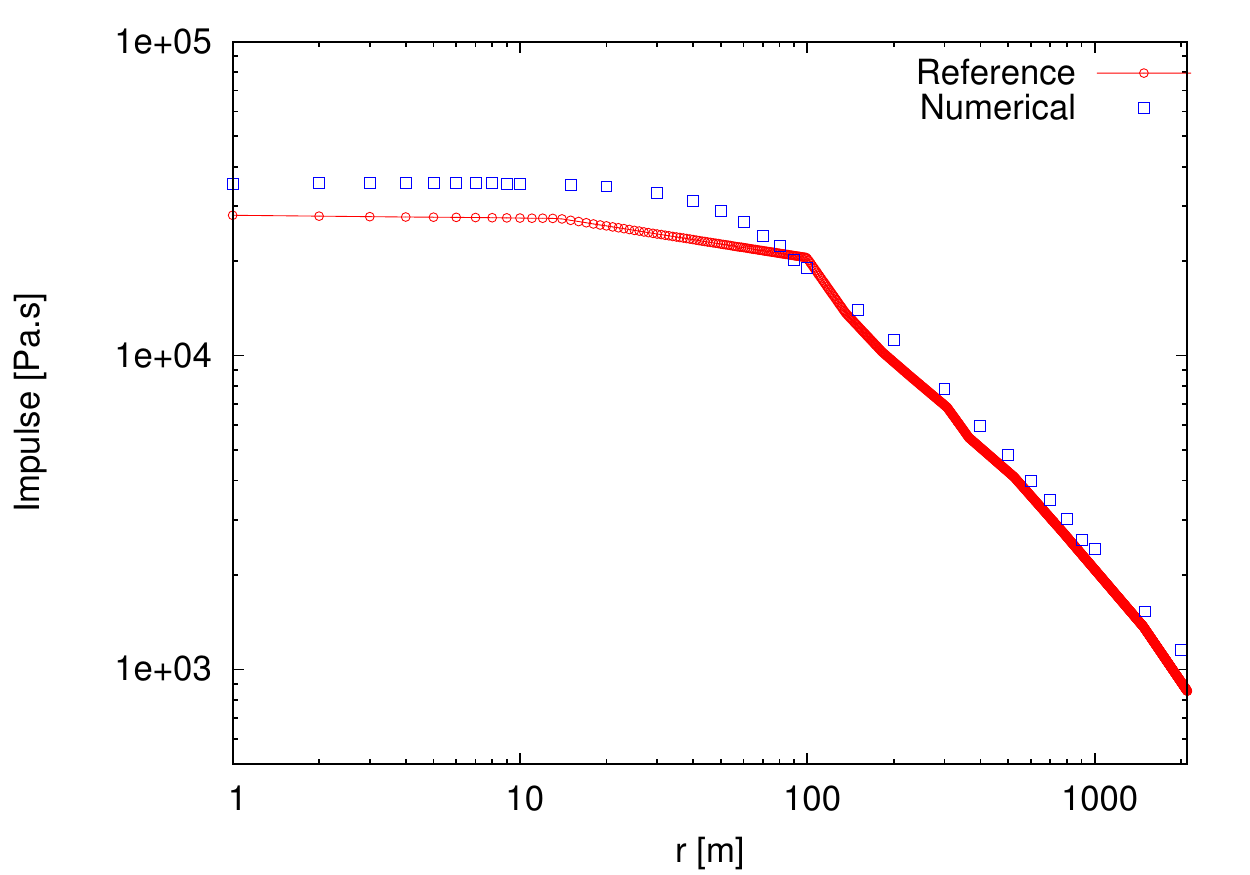}}
\caption{Shock wave parameters for nuclear air blast problem.}
\label{rm::air_blast_ref2}
\end{figure}

\subsubsection{TNT explosion in air}\label{problem:tntair}

We use this example to assess the isotropic behavior of TNT explosion in a
computational domain $0\le r\le \unit[20]{m},0\le z\le \unit[40]{m}$. The
initial density and pressure are $\unit[1630]{kg/m^3}$ and $\unit[9.5\times
10^9]{Pa}$ for high explosives, and $\unit[1.29]{kg/m^3}$ and $\unit[1.013\times
10^5]{Pa}$ for the air. The initial interface is a sphere of radius
$\unit[0.0527]{m}$ centered at the height $z=\unit[10]{m}$.  The results of
shock produced by the high explosives are shown in Fig.  \ref{rm::tnt1}. The
shock parameters are shown in Fig. \ref{rm::tnt2}, in comparison with the
experimental data in \cite{Baker1973} and \cite{Crowl1969}.

\begin{figure}[htbp]
\centering
\subfloat[Pressure contour at $t=1.8\times 10^{-4}$s]
{\includegraphics[width=0.45\textwidth]{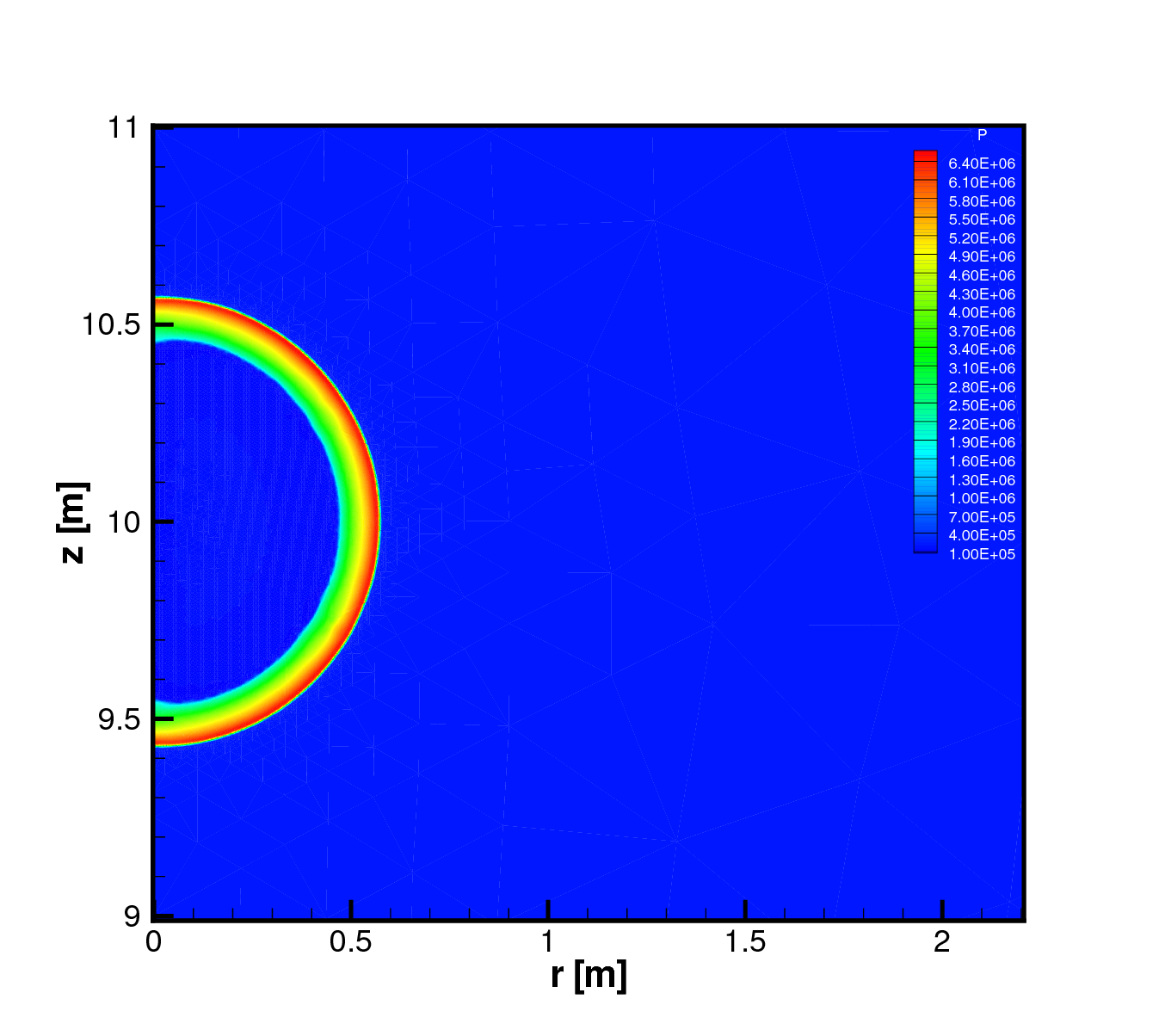}}
\subfloat[Adaptive mesh at $t=1.8\times 10^{-4}$s]
{\includegraphics[width=0.45\textwidth]{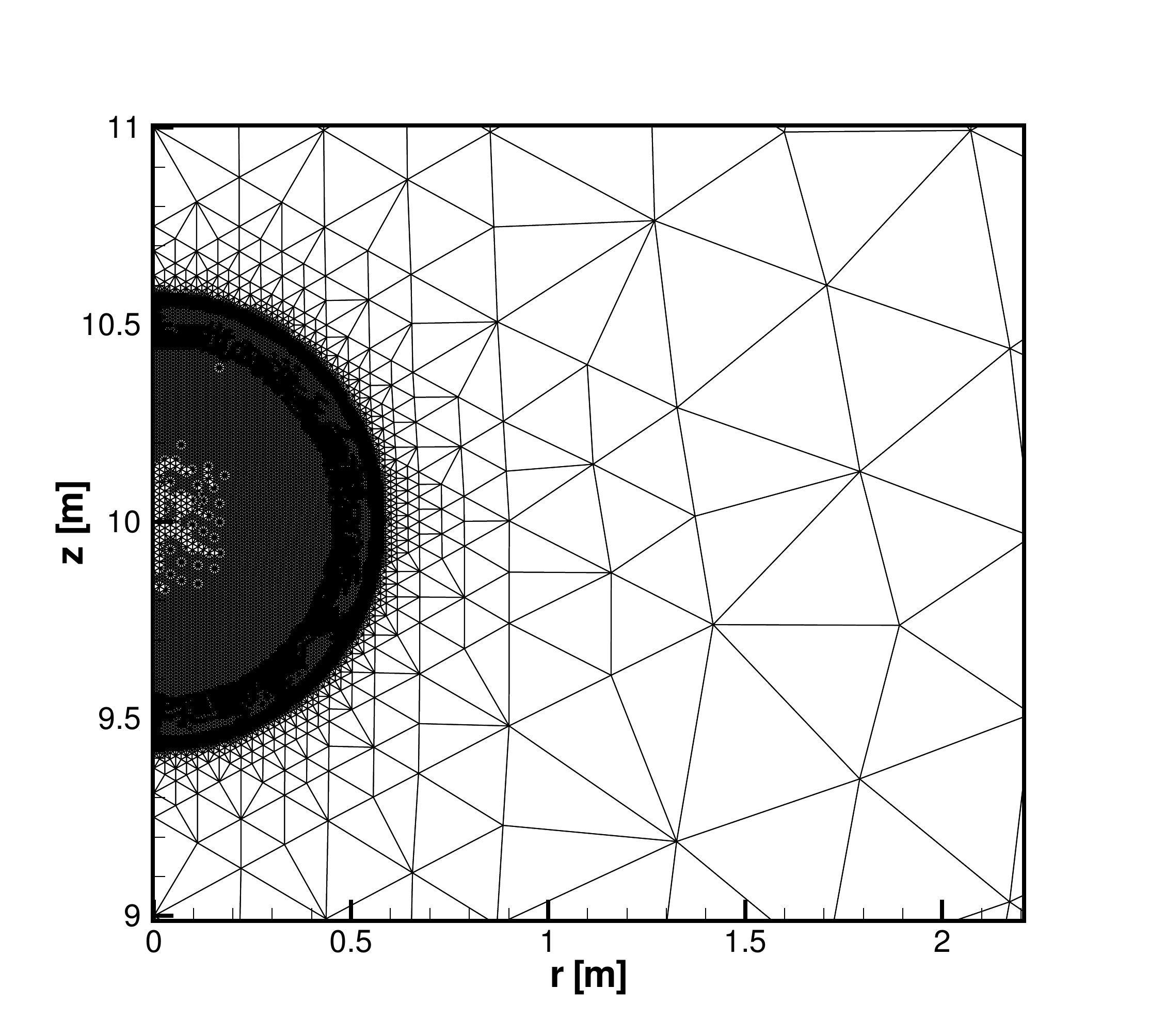}}\\
\subfloat[Pressure contour at $t=1.2\times 10^{-3}$s]
{\includegraphics[width=0.45\textwidth]{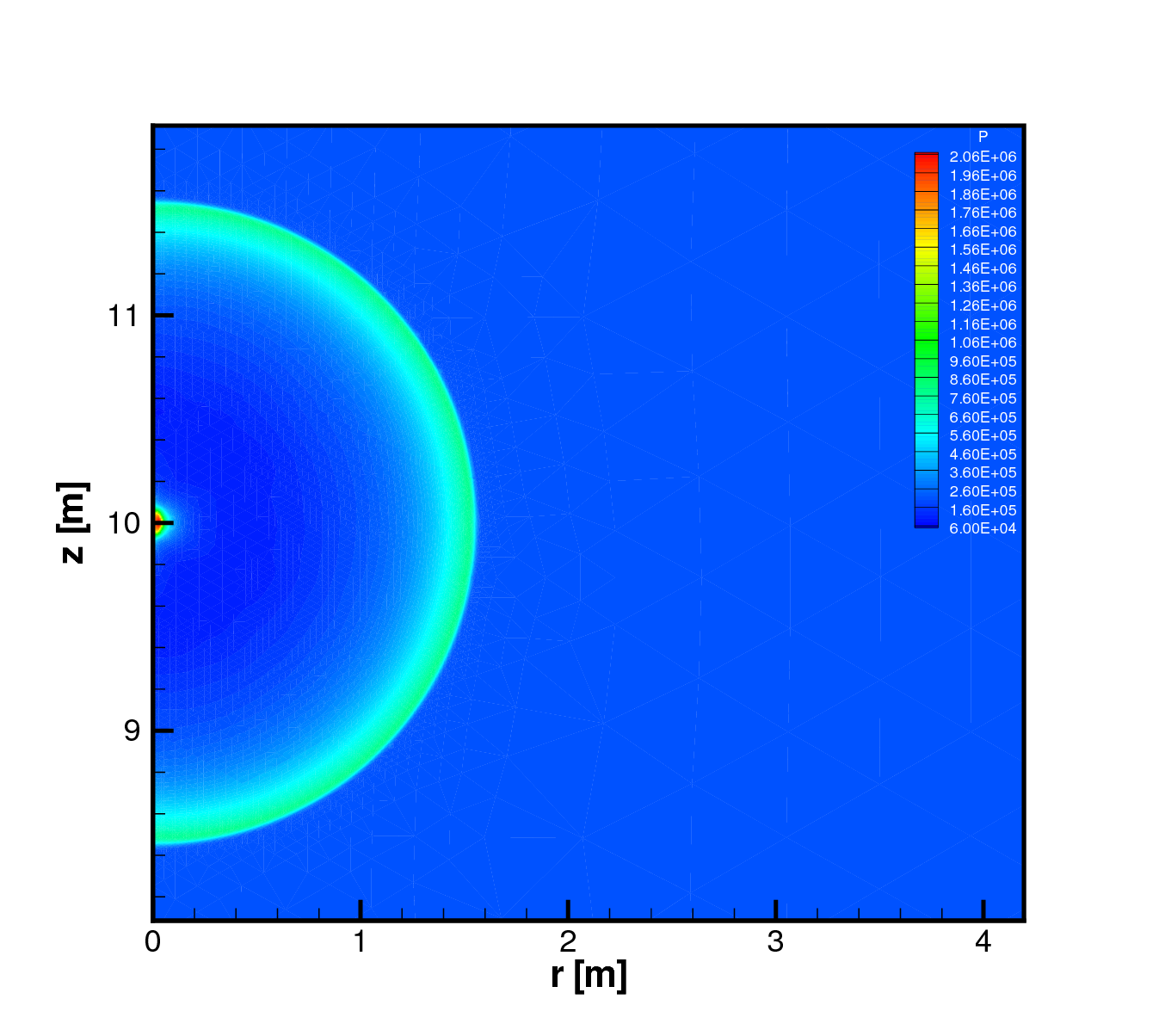}}
\subfloat[Adaptive mesh at $t=1.2\times 10^{-3}$s]
{\includegraphics[width=0.45\textwidth]{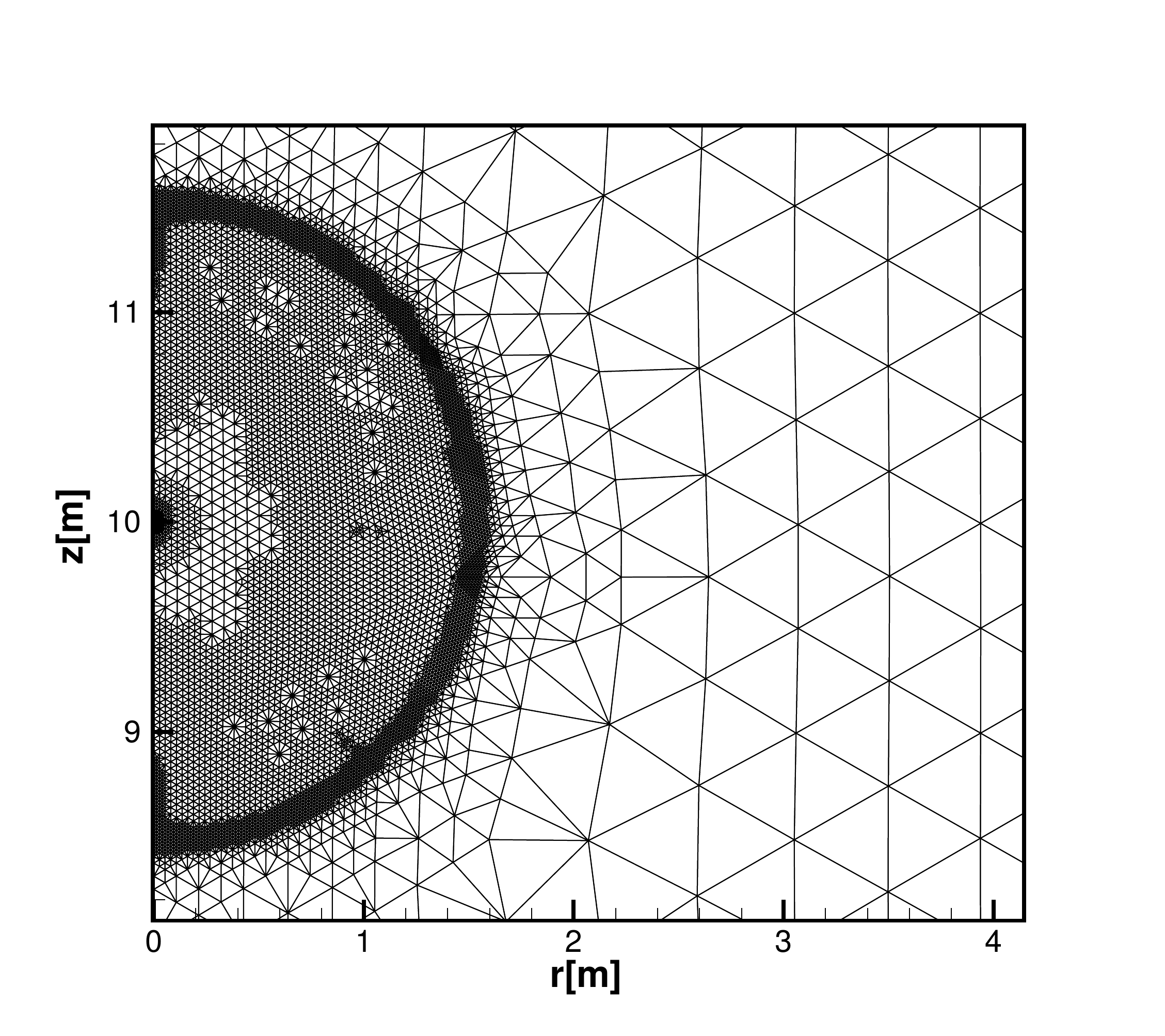}}
\caption{Pressure contours and adaptive meshes
	         for TNT explosion in air.}
\label{rm::tnt1}
\subfloat[Peak overpressure]
{\includegraphics[width=0.42\textwidth]{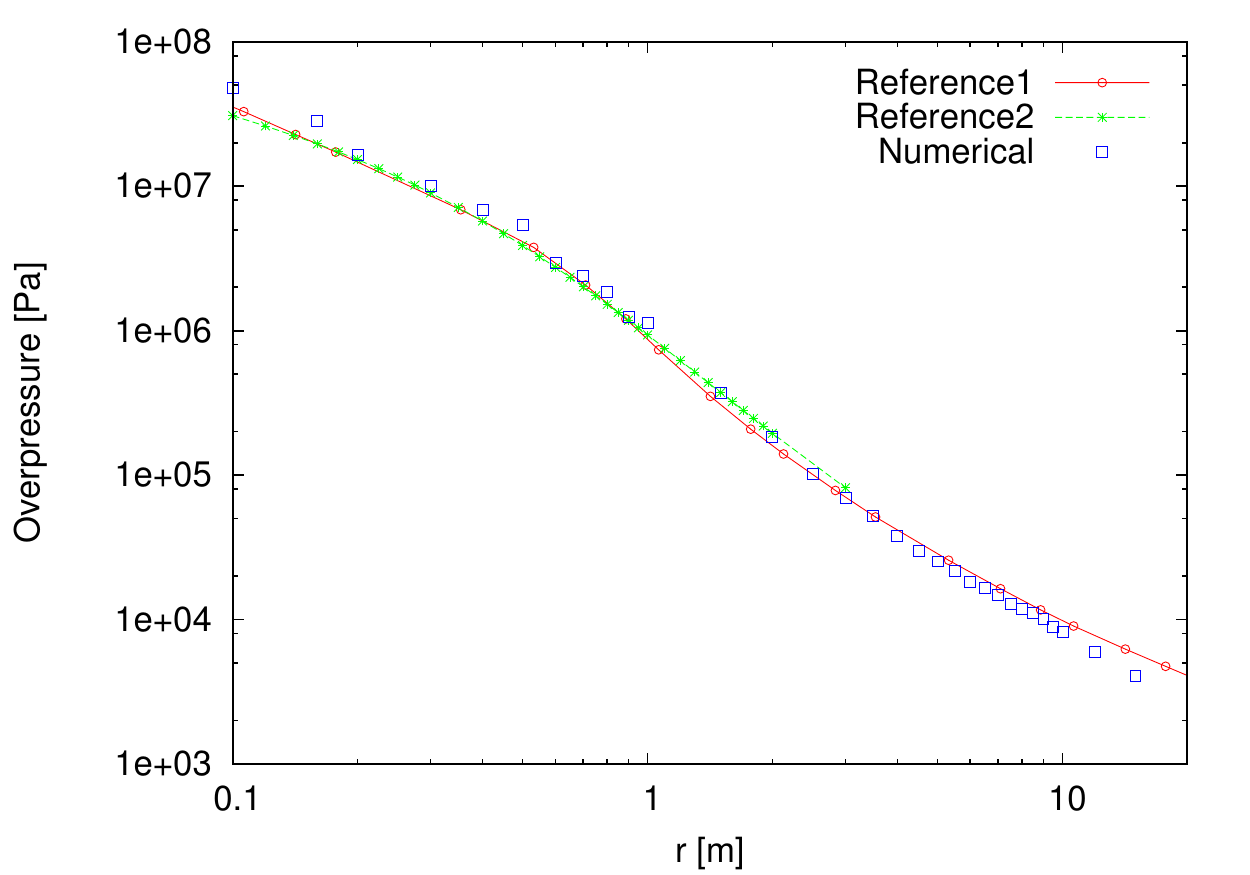}}
\subfloat[Impulse]
{\includegraphics[width=0.42\textwidth]{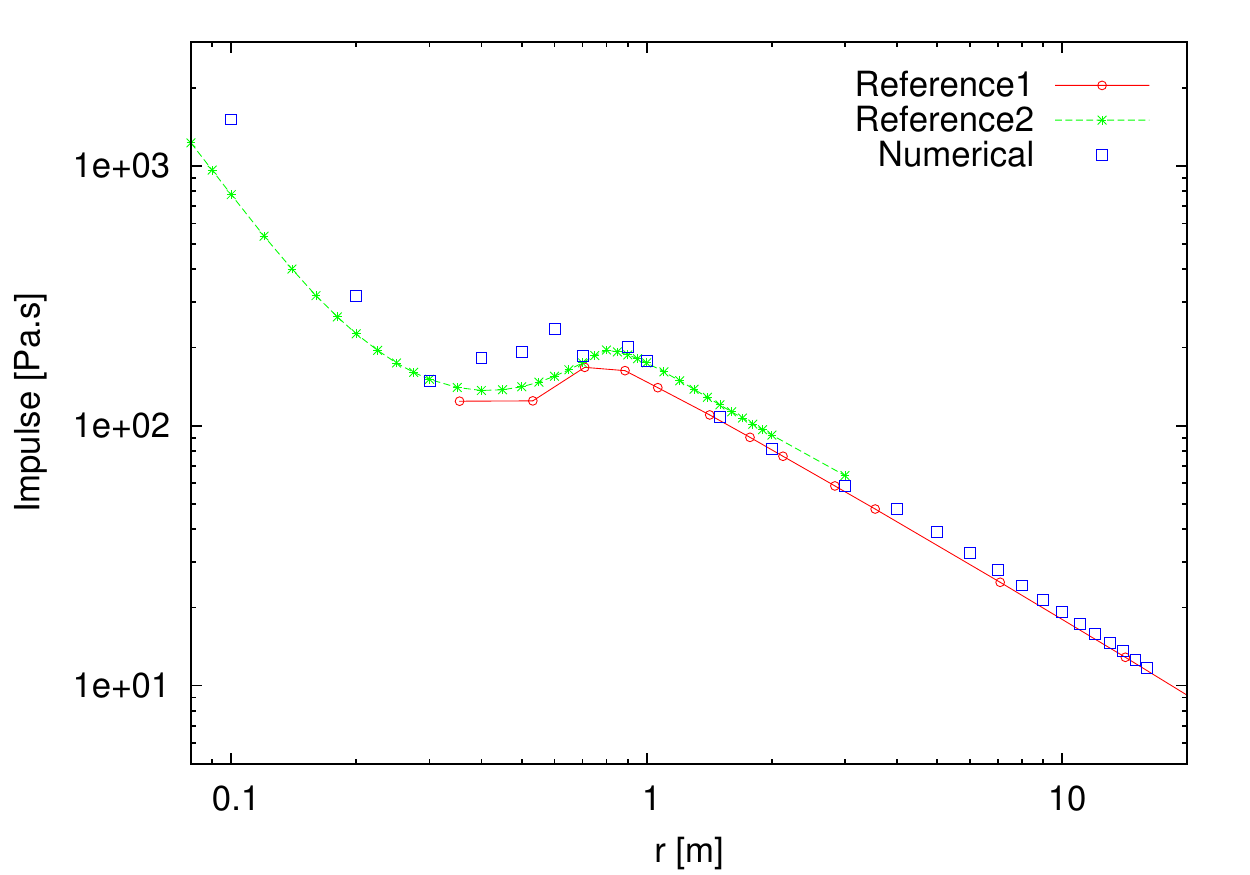}}
\caption{Shock wave parameters for TNT explosion in air. The reference
             solution 1 is taken from Baker \cite{Baker1973}, while the
             reference solution 2 is taken from Crowl \cite{Crowl1969}.}
\label{rm::tnt2}
\end{figure} 

\subsection{Three-dimensional shock impact problem}

In the last example, we present a three-dimensional shock impact problem in
practical applications. The computational domain is a cross-shaped confined
space $E_1\bigcup E_2$ where $E_1=[-2,2]\times [-2,2]\times [-6,6]$ and
$E_2=[-6,6]\times [-2,2]\times[-2,2]$ in meters. The sphere of radius
$\unit[0.4]{m}$ centered at the origin is filled with high explosives, while the
remaining region is filled with air. The initial states are exactly the same as
that of the problem in Section \ref{problem:tntair}.  Due to the symmetry we
only compute the problem in the first octant, namely an L-shaped region. All the
boundaries are reflective walls. Again we use the $h$-adaptive technique to
capture the shock front. The total number of cells is about $0.8$ -- $1$
million. To accelerate the computation we perform the parallel computing with
eight processors based on the classical domain decomposition methods.

The numerical results of shock wave produced by the high explosives at
$t=\unit[2.4\times 10^{-4}]{s}$ and $t=\unit[9.2\times 10^{-4}]{s}$ are
displayed in Fig. \ref{rm::3d1}. From here we can see that the shock wave
propagates as an expansive spherical surface in earlier period. When the
spherical shock wave impinges on the rigid surface, shock strength increases and
shock reflection occurs. It is also observed in the slice plots of Fig.
\ref{rm::3d2} that the wave structures are much more complicated in the
three-dimensional confined space. The numerical results here also show the
capability of our methods in resolving fully three-dimensional flows. 

\begin{figure}[htbp]
\centering
\subfloat[Pressure contour at $t=2.4\times 10^{-4}$s]
{\includegraphics[width=0.4\textwidth]{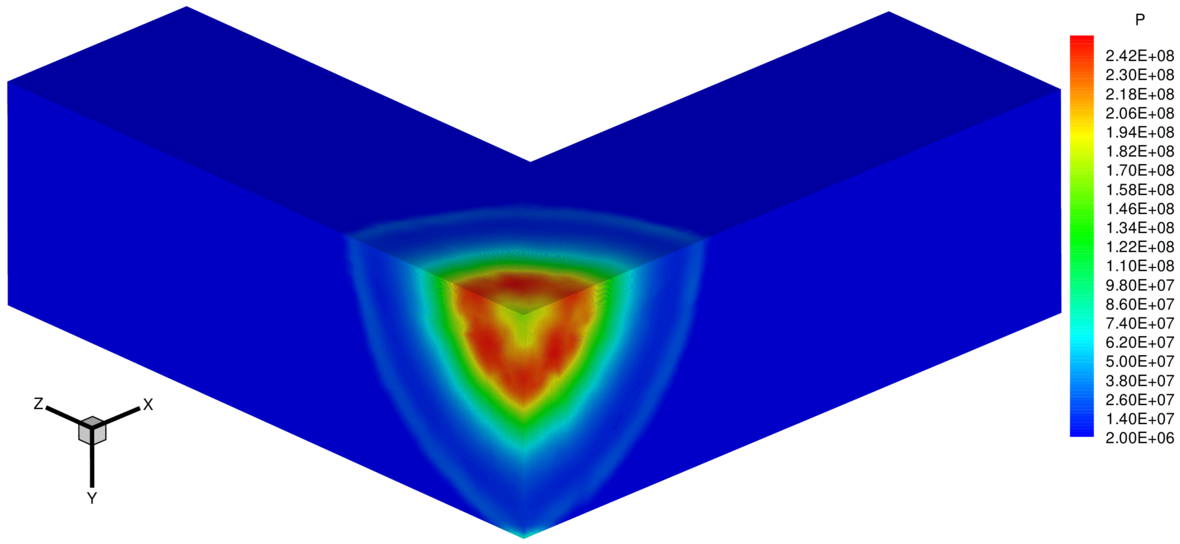}}
\qquad
\subfloat[Adaptive mesh at $t=2.4\times 10^{-4}$s]
{\includegraphics[width=0.36\textwidth]{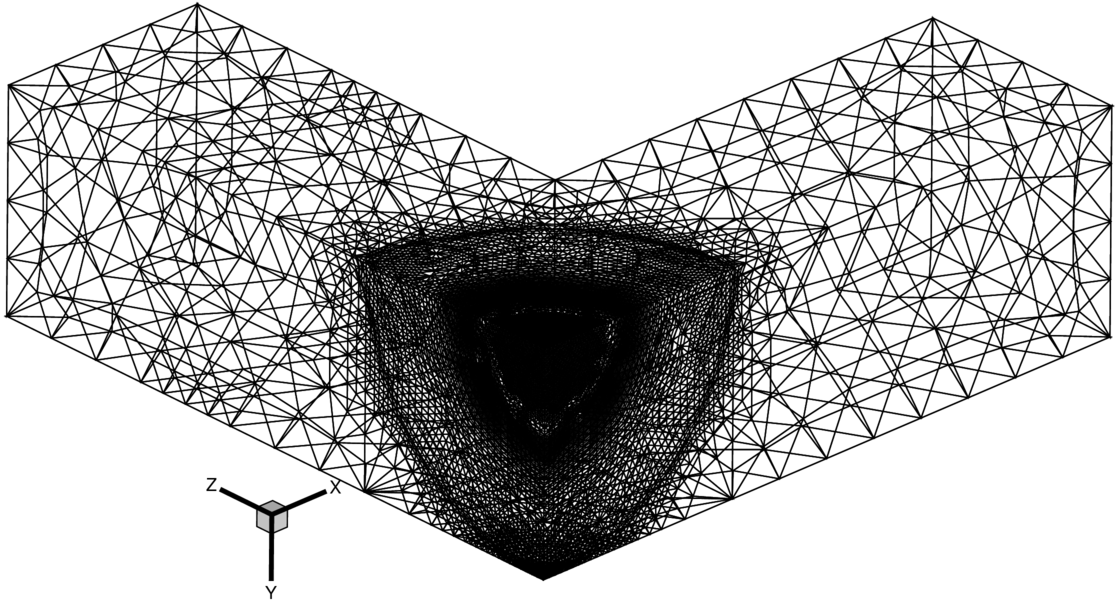}}\\
\subfloat[Pressure contour at $t=9.2\times 10^{-4}$s]
{\includegraphics[width=0.4\textwidth]{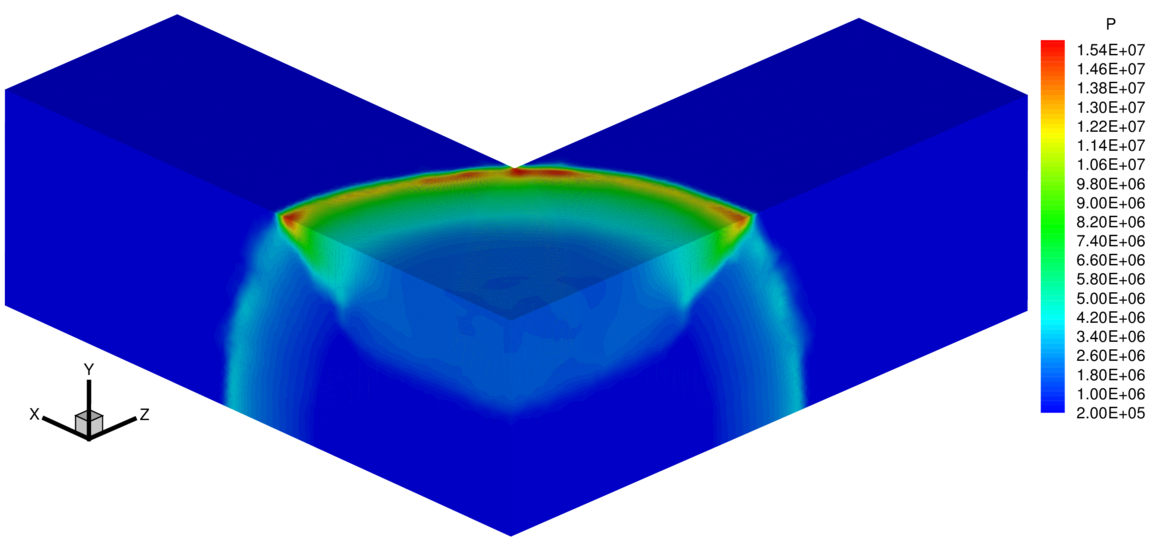}}
\qquad
\subfloat[Adaptive mesh at $t=9.2\times 10^{-4}$s]
{\includegraphics[width=0.36\textwidth]{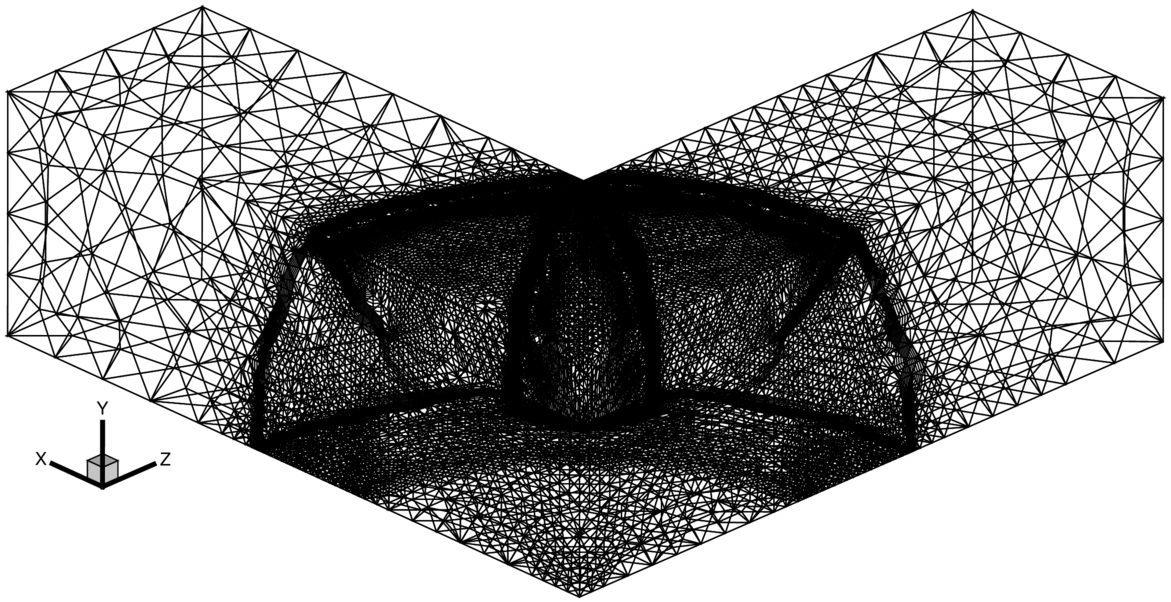}}
\caption{Pressure contours and adaptive meshes for three-dimensional
            shock impact problem. The plots (c) and (d) are turned over in the 
            $y$-direction in order to display the shock reflection.}
\label{rm::3d1}
\subfloat[Pressure slice at $x=2$m]
{\includegraphics[width=0.35\textwidth]{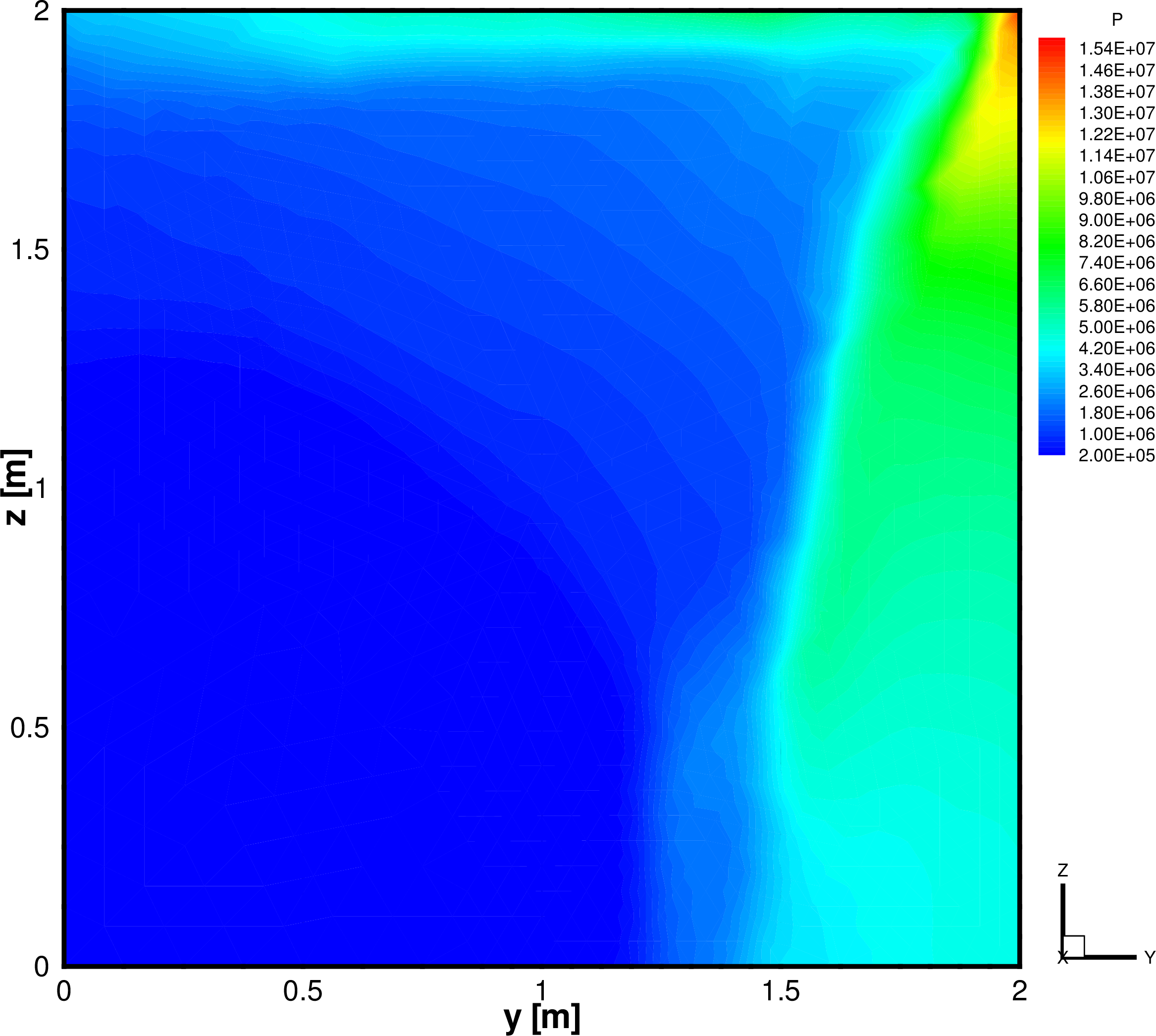}}
\qquad
\subfloat[Pressure slice at $x=2.2$m]
{\includegraphics[width=0.35\textwidth]{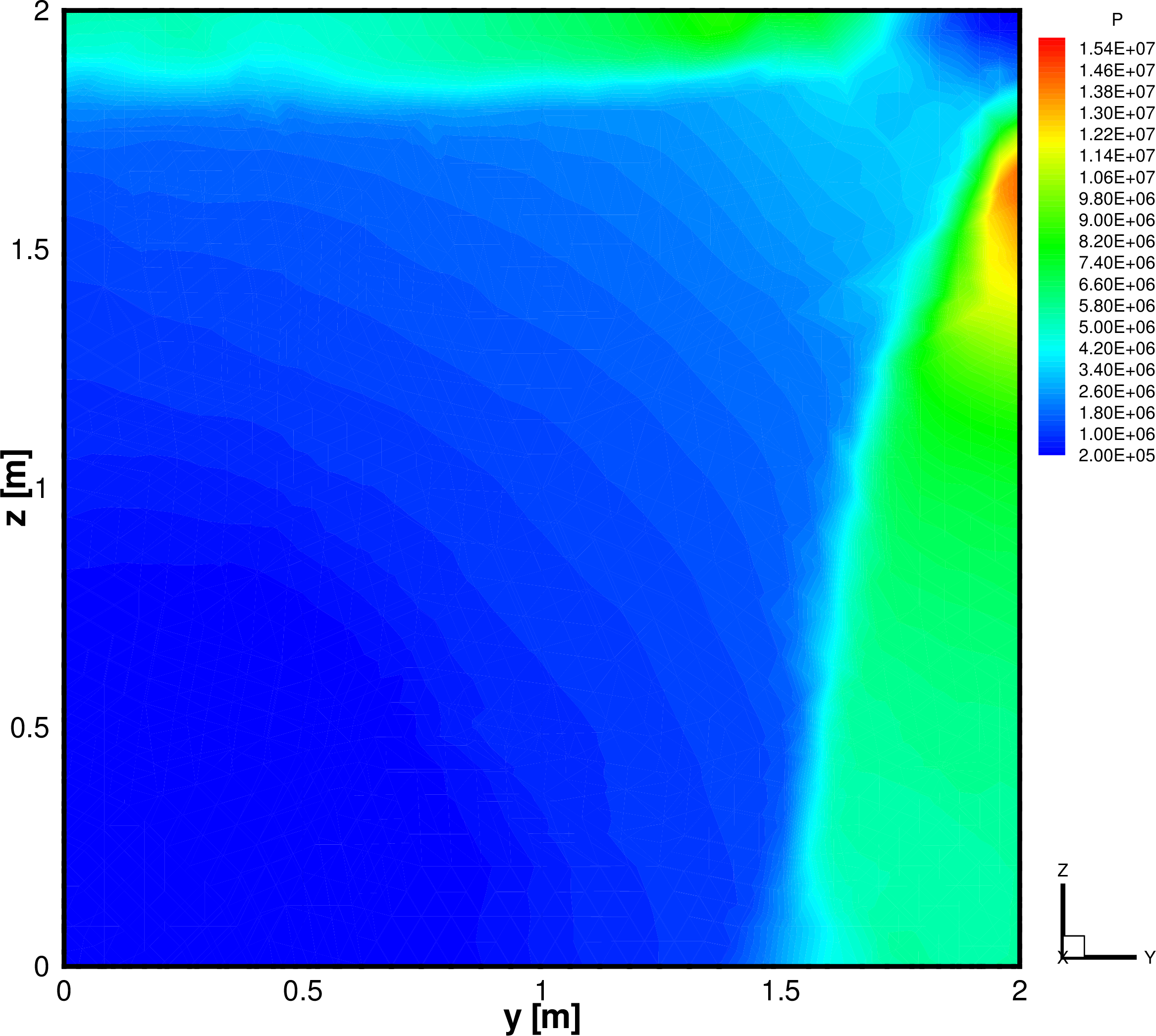}}\\
\subfloat[Pressure slice at $y=1.8$m]
{\includegraphics[width=0.35\textwidth]{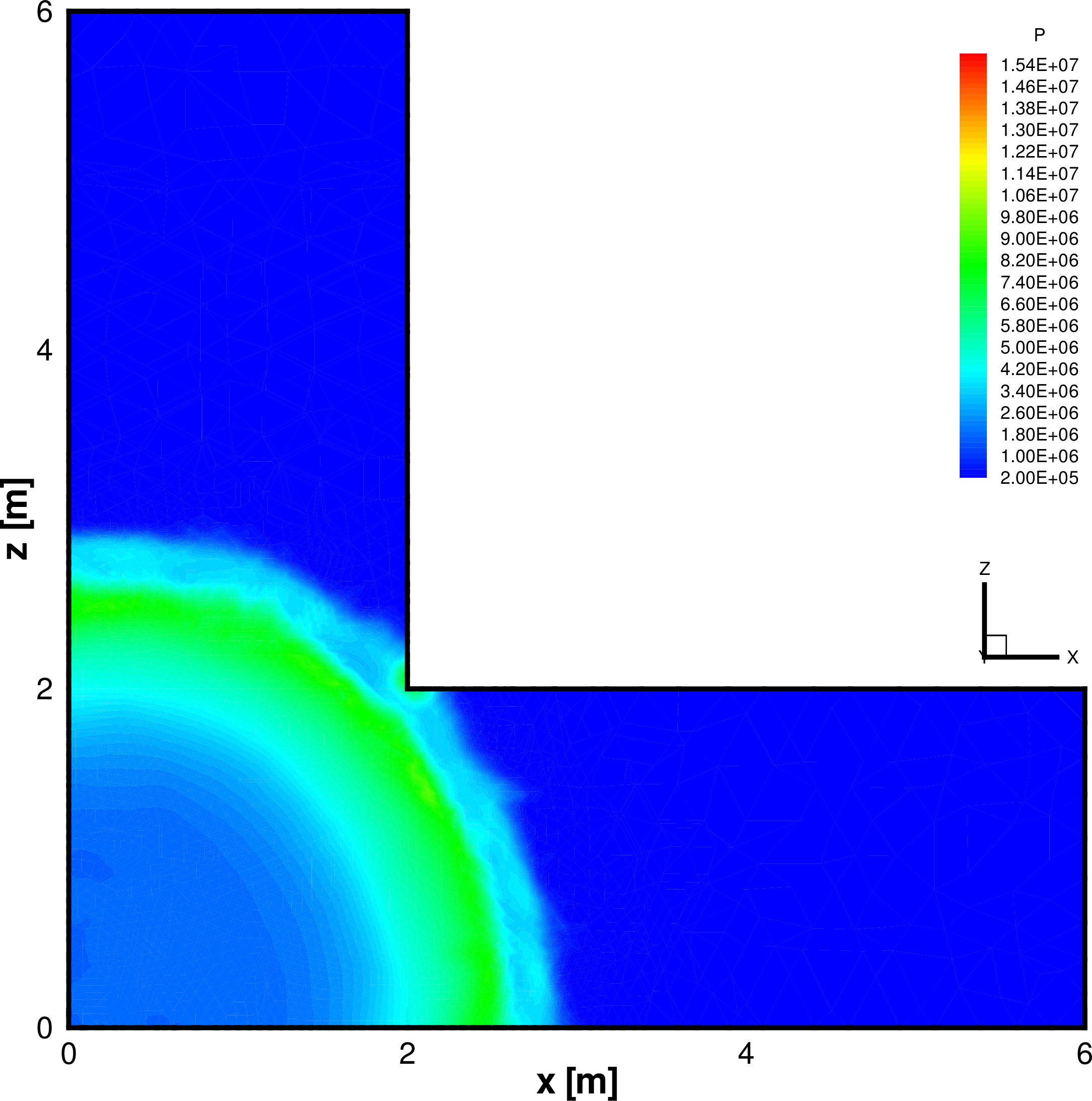}}
\qquad
\subfloat[Pressure slice at $y=2$m]
{\includegraphics[width=0.35\textwidth]{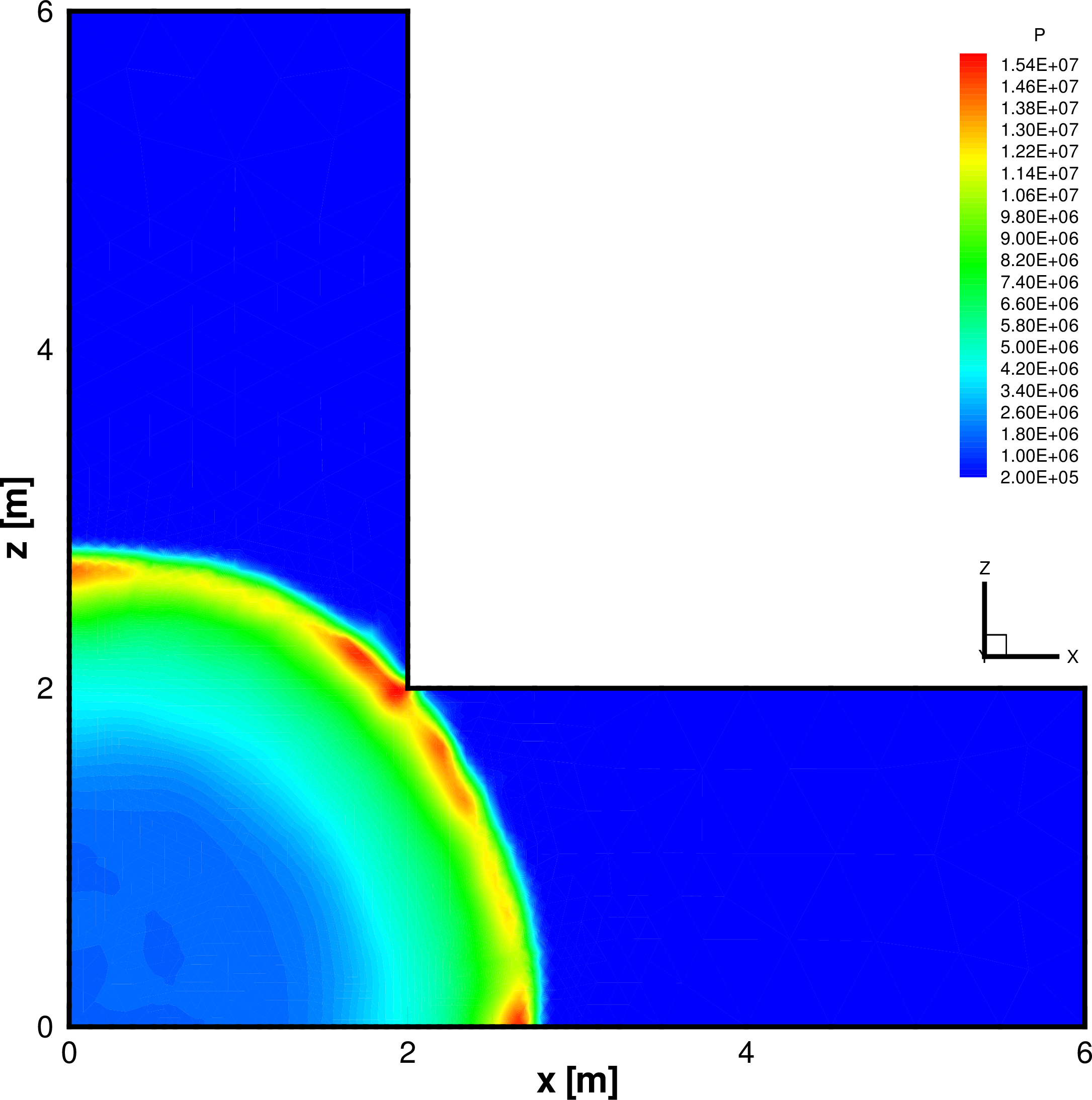}}
\caption{Pressure slices at $t=9.2\times 10^{-4}$s.}
\label{rm::3d2}
\end{figure} 


\section{Conclusions}\label{sec:conclusion}

We extend the numerical scheme in Guo \textit{et al.} \cite{Guo2016} to the
fluids that obey a general Mie-Gr{\"u}neisen equations of state. The algorithm
of the multi-medium Riemann problem is elaborated, which is a fundamental
element of the two-medium fluid flow. A variety of preliminary numerical
examples and application problems validate the effectiveness and robustness of
our methods. In our future work, we will generalize the framework to fluid-solid
coupling problems.

\clearpage
\appendix

\section{Collections of equations of state}

In this appendix we elaborate the equations of state that are mentioned in the
numerical results. For convenience, we also collect the expression of coefficients
and their derivatives at the end of this part.

\subsubsection*{\bf Ideal gas EOS}

Most of gases can be modeled by the ideal gas law
\begin{equation}
p = (\gamma-1) \rho e,
\label{eq:ideal}
\end{equation}
where $\gamma>1$ is the adiabatic exponent.

\subsubsection*{\bf Stiffened gas EOS}

When considering water under high pressures, the following stiffened 
gas EOS is often used \cite{Rallu2009,Wang2008}:
\begin{equation}
p = (\gamma-1) \rho e-\gamma p_\infty,
\label{eq:stiffened}
\end{equation}
where $\gamma>1$ is the adiabatic exponent, and $p_\infty$ is a
constant. 

\subsubsection*{\bf Polynomial EOS}

The polynomial EOS \cite{Jha2014} can be used to model various
materials
\begin{equation}
p = 
\begin{cases}
  A_1\mu + A_2\mu^2 + A_3\mu^3 + (B_0+B_1\mu)\rho_0 e, & 
  \mu>0, \\
  T_1 \mu  + T_2\mu^2 + B_0 \rho_0e,  &\mu \le 0,
\end{cases}
\label{eq:poly}
\end{equation}
where $\mu = {\rho}/{\rho_0} -1$ and $A_1,A_2,A_3,B_0,B_1, T_1,T_2,\rho_0$ are
positive constants. In this paper, we take an alternative formulation in the
tension branch \cite{Autodyn2003}, where $p=T_1 \mu  + T_2\mu^2 +
(B_0+B_1\mu)\rho_0e$ for $\mu\le0$, to ensure the continuity of the speed of
sound at $\mu=0$. Such a formulation avoids the occurance of anomalous waves in
the Riemann problem, which does not exist in real physics. When $B_1\le B_0\le
B_1+2$ and $T_1\ge 2T_2$, the polynomial EOS satisfies the conditions
\textbf{(C1)} and \textbf{(C3)}. In addition, if the density $\rho\ge
{B_0\rho_0}/{(B_1+2)}$, then the polynomial EOS also satisfies the condition
\textbf{(C2)}.

\subsubsection*{\bf JWL EOS}

Various detonation products of high explosives can be characterized
by the JWL EOS \cite{Smith1999}
\begin{equation}
p = A_1\left(1-\frac{\omega\rho}{R_1\rho_0}\right)
 \exp\left(-\frac{R_1\rho_0}{\rho}\right) +
 A_2\left(1-\frac{\omega\rho}{R_2\rho_0}\right) \exp\left(-\frac{
 R_2\rho_0}{\rho}\right) + \omega\rho e,
\label{eq:jwl}
\end{equation}
where $A_1,A_2,\omega,R_1,R_2$ and $\rho_0$ are positive constants.
Obviously the JWL EOS \eqref{eq:jwl} satisfies the conditions
\textbf{(C1)}
and \textbf{(C2)}. To enforce the condition \textbf{(C3)} we first
notice that 
\[
\lim_{\rho\rightarrow 0^+} h'(\rho) = 0.
\]
Then it suffices to ensure that $h''(\rho)\ge 0$, which is equivalent
to
the following inequality in terms of $\nu=\rho_0/\rho$:
\[
R_1\nu-2-\omega
\ge G(\nu):=\dfrac{A_2R_2}{A_1R_1}(2+\omega -R_2\nu)
\exp((R_1-R_2)\nu).
\]
A simple calculus shows that the maximum value of the function 
$G(\nu)$ above is given by
\[
\alpha = \dfrac{A_2R_2^2}{A_1R_1(R_1-R_2)}
\exp\left(\dfrac{(2+\omega)(R_1-R_2)-R_2}{R_2}\right).
\]
Therefore a sufficient condition for
\textbf{(C3)} is that the density satisfies
\[
\rho\le \dfrac{R_1}{2+\omega+\alpha}\rho_0,
\]
which is valid for most cases.

\subsubsection*{\bf Cochran-Chan  EOS}

The Cochran-Chan EOS is commonly used to describe the reactants of 
condensed phase explosives \cite{Saurel2007,Lee2013}, which can be
formulated as
\begin{equation}
p = \dfrac{A_1(R_1-1-\omega)}{R_1-1}
    \left(\dfrac{\rho}{\rho_0}\right)^{R_1}-
    \dfrac{A_2(R_2-1-\omega)}{R_2-1}
    \left(\dfrac{\rho}{\rho_0}\right)^{R_2}+
    \omega \rho e,
\label{eq:CC}
\end{equation}
where $A_1,A_2,\omega,R_1,R_2$ and $\rho_0$ are positive constants.
The Cochran-Chan EOS satisfies the conditions \textbf{(C1), (C2)} and 
\textbf{(C3)} if $1<R_2\le  1+\omega \le  R_1$. 

\setlength{\tabcolsep}{2pt}
\setlength{\rotFPtop}{0pt plus 1fil}
\begin{sidewaystable}
 \centering
 \caption*{List of coefficients and their derivatives for several equations of state.}
 \small
\begin{tabular}{cccccc}
\toprule
& \multicolumn{1}{c}{\bf Ideal}&
 \multicolumn{1}{c}{\bf Stiffened} &
 \multicolumn{1}{c}{\bf Polynomial} &
 \multicolumn{1}{c}{\bf JWL}&
 \multicolumn{1}{c}{\bf Cochran-Chan} \\ \midrule
$\varGamma(\rho)$ &
$\gamma-1$ &
$\gamma-1$  &
$B_1+{(B_0-B_1)\rho_0}/{\rho}$ &
$\omega$ &
$\omega$ \\  [10mm]
$\varGamma'(\rho)$  &
$0$ &
$0$ &
$-{(B_0-B_1)\rho_0}/{\rho^2}$ & 
$0$ &
$0$ \\   [10mm]
$\varGamma''(\rho)$  &
$0$ &
$0$ &
${2(B_0-B_1)\rho_0}/{\rho^3}$ &
$0$ &
$0$ \\  [10mm]
$h(\rho)$ &
$0$ &
$-\gamma p_\infty$ &
$\begin{cases}  
A_1\mu + A_2 \mu^2 + A_3 \mu^3, &
\rho > \rho_0 \\ 
T_1 \mu+T_2\mu^2,
& \rho \le \rho_0
  \end{cases}$  &
 \bgroup
 \footnotesize
\begin{tabular}{r}
$A_1\left(1-\dfrac{\omega\rho}{R_1\rho_0}\right)
     \exp\left(-\dfrac{R_1\rho_0}{\rho}\right)$ \\
$+A_2\left(1-\dfrac{\omega\rho}{R_2\rho_0}\right)
     \exp\left(-\dfrac{R_2\rho_0}{\rho}\right)$     
\end{tabular} 
\egroup &
\bgroup 
 \footnotesize
\begin{tabular}{r}
	$\dfrac{A_1(R_1-1-\omega)}{R_1-1}\left(\dfrac{\rho}{\rho_0}\right)^{R_1}$ \\
	$-\dfrac{A_2(R_2-1-\omega)}{R_2-1}\left(\dfrac{\rho}{\rho_0}\right)^{R_2}$     
\end{tabular} 
\egroup
\\   [10mm]
$h'(\rho)$  &
$0$ &
$0$ &
$\begin{cases}  
(A_1+2A_2\mu+3A_3\mu^2)/\rho_0, &
\rho > \rho_0 \\ 
(T_1+2T_2\mu)/\rho_0,
& \rho \le \rho_0
  \end{cases}$  & 
\bgroup 
 \footnotesize
\begin{tabular}{r}
$A_1\left(\dfrac{R_1\rho_0}{\rho^2}-\dfrac{\omega}{\rho}
-\dfrac{\omega}{R_1\rho_0}\right)
\exp\left(-\dfrac{R_1\rho_0}{\rho}\right)$ \\
$+A_2\left(\dfrac{R_2\rho_0}{\rho^2}-\dfrac{\omega}{\rho}
-\dfrac{\omega}{R_2\rho_0}\right)
\exp\left(-\dfrac{R_2\rho_0}{\rho}\right)$
\end{tabular} 
\egroup & 
\bgroup 
 \footnotesize
\begin{tabular}{r}
	$\dfrac{A_1R_1(R_1-1-\omega)}{R_1-1}\left(\dfrac{\rho}{\rho_0}\right)^{R_1-1}$ \\
	$-\dfrac{A_2R_2(R_2-1-\omega)}{R_2-1}\left(\dfrac{\rho}{\rho_0}\right)^{R_2-1}$     
\end{tabular} 
\egroup
\\ [10mm]
$h''(\rho)$  &
$0$ &
$0$ &
$\begin{cases}  
(2A_2+6A_3\mu)/\rho_0^2, &
\rho > \rho_0 \\ 
2T_2/\rho_0^2,
& \rho \le \rho_0
\end{cases}$  & 
\bgroup
 \footnotesize
\begin{tabular}{r}
	$\dfrac{A_1R_1\rho_0}{\rho^3}
	\left(\dfrac{R_1\rho_0}{\rho}-2-\omega\right)
	\exp\left(-\dfrac{R_1\rho_0}{\rho}\right)$ \\
	$+\dfrac{A_2R_2\rho_0}{\rho^3}
	\left(\dfrac{R_2\rho_0}{\rho}-2-\omega\right)
	\exp\left(-\dfrac{R_2\rho_0}{\rho}\right)$
\end{tabular} 
\egroup &
\bgroup
 \footnotesize
\begin{tabular}{r}
	${A_1R_1(R_1-1-\omega)}\left(\dfrac{\rho}{\rho_0}\right)^{R_1-2}$ \\
	$-{A_2R_2(R_2-1-\omega)}\left(\dfrac{\rho}{\rho_0}\right)^{R_2-2}$     
\end{tabular} 
\egroup
\\ \bottomrule
\end{tabular}
\end{sidewaystable}
\clearpage


\section*{\refname}
\bibliographystyle{unsrt}
\bibliography{reference}
\end{document}